\documentclass{amsart}

\RequirePackage{amspaperstyle}

\newcommand{\cpt}{\operatorname{cpt}}
\newcommand{\ii}{\mathrm{i}}
\newcommand{\jj}{\mathrm{j}}
\newcommand{\kk}{\mathrm{k}}
\newcommand{\HW}{\mathrm{HW}}
\newcommand{\pure}[1]{{#1}^{(0)}}
\newcommand{\sicover}[1]{{#1}^{+}}
\newcommand{\Gk}[1]{\G_{(#1)}}

\usepackage{scalerel}
\usepackage{stackengine}
\newcommand\widesim[1]{
	\mathrel{\ThisStyle{%
  	\setbox0=\hbox{$\SavedStyle\mkern3mu_{#1}\mkern3mu$}%
  \mkern1mu\stackengine{1\LMpt}{%
    \stretchto{\scaleto{\SavedStyle\mkern-1mu\sim\mkern-1mu}{.54\wd0}}{1\ht0}%
  }{$\SavedStyle_{#1}$}{O}{c}{F}{T}{S}\mkern1mu%
}}}

\begin{document}

\title{Linnik's problems and maximal entropy methods}

\author{Andreas Wieser}
\address{Departement Mathematik, ETH Z{\"u}rich, R{\"a}mistrasse 101, 8092 Z{\"u}rich, Switzerland}
\email{andreas.wieser@math.ethz.ch}
\thanks{The author was supported by the SNF (grants 152819 and 178958).}



\date{\today}



\begin{abstract}
We use maximal entropy methods to examine the distribution properties of primitive integer points on spheres and of CM points on the modular surface. The proofs we give are a modern and dynamical interpretation of Linnik's original ideas and follow techniques presented by Einsiedler, Lindenstrauss, Michel and Venkatesh in 2011.
\end{abstract}

\maketitle

\section{Introduction}

\subsection{Integer points on spheres}
Consider the set of primitive integral solutions to the equation
\begin{align*}
x^2 + y^2 + z^2 = d
\end{align*}
for some positive integer $d$. The question whether or not such a solution exists was raised by Legendre (amongst others), who claimed that the equation $x^2+ y^2 + z^2 = d$ has a solution if and only if $d$ is not of the form $d= 4^a (8b+7)$ for non-negative integers $a,b$. A full proof of Legendre's so-called three-squares theorem was given by Gauss \cite{gauss}. In fact, a primitive integral solution to $x^2+ y^2 + z^2 = d$ exists if and only if $d$ satisfies Legendre's condition
\begin{align*}
d \in \mathbb{D} := \setc{d\in \mathbb{N}}{d \not\congruent 0,4,7  \mod 8}.
\end{align*}
A further question treated by Gauss concerns a refinement of the above: As $d$ tends to infinity with $d$ satisfying Legendre's condition, how many primitive integral solutions to $x^2+ y^2 + z^2 = d$ are there? The number of such solutions turns out to be closely related to the class number of the quadratic number field $\Q(\sqrt{-d})$ (see for instance \cite{Linnikthmexpander}) for which an asymptotic as $d \to \infty$ is well-known thanks to Dirichlet's class number formula and Siegel's lower bound. 
These results imply that the number of primitive integral solutions to $x^2+ y^2 + z^2 = d$ is $d^{\frac{1}{2}+o(1)}$ for $d \to \infty, d \in \mathbb{D}$. In this paper, we will be interested in the distribution properties of these solutions: Let
\begin{align*}
\mathcal{I}_d := \tfrac{1}{\sqrt{d}}\setc{(x,y,z) \in \Z^3}{\gcd(x,y,z) =1,\ x^2 + y^2 + z^2 = d}
\end{align*}
be the set of primitive integral solutions to $x^2+ y^2 + z^2 = d$ projected onto the unit sphere $\mathbb{S}^2$. Using ergodic-theoretic methods, we will give a proof of the following non-effective result due to Linnik \cite{linnik}:

\begin{namedtheorem}{Linnik's}{A}[Equidistribution of primitive integer points on the sphere]
Let $p$ be an odd prime and let $\mathbb{D}(p) = \setc{d \in \mathbb{D}}{-d \mod p \in (\Fp^\times)^2}$. As $d$ tends to infinity with $d \in \mathbb{D}(p)$, the normalized sums of Dirac measures $\frac{1}{|\mathcal{I}_d|} \sum_{x \in \mathcal{I}_d} \delta_x$ equidistribute to the uniform probability measure on the unit sphere $\mathbb{S}^2$.
\end{namedtheorem}

Here and in what follows, $(\Fp^\times)^2$ denotes the set of squares in $\Fp^\times$.

Since the choice of the prime $p$ is arbitrary, the splitting\footnote{In fact, note that  $-d \in (\Fp^\times)^2$ if and only if $p$ is split in the field $\Q(\sqrt{-d})$.} condition $-d \in (\Fp^\times)^2$, known as \emph{Linnik's condition} at the prime $p$, seems to be superfluous. Indeed, Linnik was able to eliminate it assuming GRH.
In 1988, Duke \cite{duke88} succeeded in proving Linnik's theorem unconditionally using entirely different methods building on work of Iwaniec \cite{iwaniecforduke}.

A modern exposition of Linnik's theorem using expander graphs is given by Ellenberg, Michel and Venkatesh in \cite{Linnikthmexpander}. The present article aims to give an ergodic theoretic proof of Linnik's theorem using maximal entropy methods and following Einsiedler, Lindenstrauss, Michel and  Venkatesh \cite{duke}. The motivation for such a proof originates from a refinement of Linnik's theorem by Aka, Einsiedler and Shapira in \cite{AES15}.

In \cite{duke}, the authors prove a theorem due to Duke \cite{duke88} concerning equidistribution of collections of closed geodesics (associated to positive discriminants) on the complex modular curve $Y_0(1) = \lquot{\SL_2(\Z)}{\mathbb{H}}$. 
In analogy to \cite{duke}, we will study certain packets of orbits in the adelic extension $X_\adele$ of $\lquot{\SO_3(\Z)}{\SO_3(\R)}$ which arise through the stabilizer subgroup in $\SO_3$ of a primitive integer point of length $\sqrt{d}$. 
We note that one dynamical reason for working in an extension instead of the real quotient $\lquot{\SO_3(\Z)}{\SO_3(\R)}$ is that the acting subgroup $\SO_3(\Qp)$ is non-compact for all odd primes $p$.
 
Furthermore, if $v \in \Z^3$ is a primitive integer point then the stabilizer subgroup $\mathbb{H}_v = \setc{g \in \SO_3}{gv = v}$ is the orthogonal group of the quadratic form $x^2+y^2+z^2$ restricted to the plane $v^{\perp}$. Thus, $\mathbb{H}_v(\Q_p)$ is split if and only if $v^{\perp}$ contains an isotropic vector. 
One can show by elementary means that the latter is equivalent to Linnik's condition for $d = v_1^2+v_2^2+v_3^2 =: Q(v)$ (see also Lemma \ref{lemma:Linnik's condition}). 
Therefore, Linnik's condition is an artefact of our dynamical proof (it is comparable to the positivity assumption on discriminants in \cite{duke}). 
Using the vector $v$ the packet $\mathcal{P}(v)$ is defined to be the orbit of the identity coset in $X_\adele$ under $\BH_v(\adele)$.

We will see that the \wstar-limits of the uniform measures on the packets $\mathcal{P}(v)$ when $Q(v) \in \mathbb{D}(p)$ goes to infinity have a large invariance subgroup.
This is the content of our main result -- Theorem \ref{thm:main} -- that is phrased in a more general setting.
From this, Linnik's Theorem A is readily obtained by projecting onto the real quotient. 
The main step in the proof of Theorem \ref{thm:main} for $\SO_3$ is to show that any \wstar -limit of the measures has maximal entropy with respect to the action of a fixed diagonalizable element in $\SO_3(\Q_p)$ (cf.\ Theorem \ref{thm: Maximal entropy}). 
This maximal entropy statement can then be used to deduce the theorem by means of a uniqueness result on measures of maximal entropy see e.g. \cite[Theorem~7.9]{pisa} or \cite{MT94} (cf.\ Section~\ref{subsection: From maximal entropy to additional invariance}). To show maximal entropy one verifies, roughly speaking, the following two claims which in combination yield the desired ``chaotic behaviour''.
\begin{enumerate}[(1)]
\item The total volume of $\mathcal{P}(v)$ grows quickly enough. In fact, the total volume is $Q(v)^{\frac{1}{2}+o(1)}$ as one can attach to any $\Hv(\R \times \widehat{\Z})$-orbit in $\mathcal{P}(v)$ a uniquely determined integer point of length $\sqrt{d}$ (cf. Section \ref{section:generating intpts}) and then apply the result mentioned earlier.
\item There are not too many pairs of orbits in $\mathcal{P}(v)$ that lie close together. This is the content of Linnik's basic lemma (Proposition \ref{prop: Linnik's basic lemma}), for which we shall use a bound on the number of representations of a binary quadratic form by $x^2+y^2 + z^2$ (cf. Theorem \ref{thm: representations qf}).
\end{enumerate}

\subsection{CM points} 

A number $d \in \Z$ is a \emph{discriminant} if it is of the form $d = b^2-4ac$ for integers $a,b,c\in \Z$ or equivalently if it is the discriminant of a binary form $ax^2+bxy+cy^2$.
In this case, we will call $(a,b,c) \in \Z^3$ a representation of $d$.
Notice that an integer $d \in \Z$ is a discriminant if and only if $d \equiv 0,1\mod 4$.

There are $h_d = |d|^{\frac{1}{2}+o(1)}$ ``inequivalent'' primitive representations where $h_d$ is the cardinality of the Picard group of the order $R_d = \Z[\frac{d+\sqrt{d}}{2}]$.
In fact, as shown for instance in \cite[Section 2]{duke} there is a correspondence between
\begin{enumerate}[(i)]
\item $\GL_2(\Z)$-equivalence classes of primitive integral quadratic forms with discriminant $d$ and
\item $K^\times$-homothety classes of proper $R_d$-ideals in $K = \Q(\sqrt{d})$.
\end{enumerate}
In order to phrase the next equidistribution result, we change the viewpoint slightly. Fixing a negative discriminant $d$, a \emph{CM point} of discriminant $d$ is a point of the form
\begin{align*}
x_{a,b,c} = \frac{-b + \sqrt{|d|}\, \mathrm{i}}{2a}\in \mathbb{H}.
\end{align*}
where $(a,b,c)$ is a primitive representation of $d$.
One readily verifies that if two quadratic forms are $\SL_2(\Z)$-equivalent, then their associated CM points also lie on the same $\GL_2(\Z)$-orbit where $\SL_2(\Z)$ acts on $\mathbb{H}$ by M\"obius transformations. By the correspondence above, there are finitely many CM-points of discriminant $d$ up to the $\SL_2(\Z)$-action. The following result is also due to Linnik \cite{linnik}:

\begin{namedtheorem}{Linnik's}{B}[Equidistribution of CM points]
Let $p$ be an odd prime. Let $H_d$ be the set of CM points associated to a discriminant $d<0$.~Then
\begin{align*}
\frac{1}{|\SL_2(\Z).H_d|}\sum_{z \in \SL_2(\Z).H_d}\delta_z \to m_{\lquot{\SL_2(\Z)}{\mathbb{H}}} 
\end{align*}
as $d \to -\infty$ amongst the discriminants that satisfy $d \mod p \in (\Fp^\times)^2$ (Linnik's condition). 
\end{namedtheorem}

Note that there is an analogue of this theorem for positive discriminants where a CM point is replaced by the geodesic connecting the two points on the real axis (the real roots of $ax^2+bx+c$) -- see Theorem~1.3 in \cite{duke}. 
For a reformulation of Linnik's Theorem B in a fashion similar to Linnik's Theorem A see Theorem 1.1 in \cite{duke}.

In Section \ref{section:Equidistribution of CM points} of this paper, we prove Linnik's Theorem B using maximal entropy methods again. Up to some complications due to non-compactness of the homogeneous space $\lquot{\PGL_2(\Zinvp)}{\PGL_2(\R\times \Qp)}$ we consider, the proof is largely analogous to the proof of Linnik's Theorem~A. We note that it also studies collections of orbits under certain tori (given by stabilizer subgroups constructed in Section \ref{section:CM-algebraictori}).

\subsection{Some notation and facts}\label{section:intro-notation}

For a set $S\subset \mathcal{V}^\Q$ of places we let $\Q_S$ be the restricted product of the completions $\Q_\sigma$ for $\sigma\in S$.
We also write $\adele$ for the ring of adeles of $\Q$ ($S= \mathcal{V}^\Q$), $\mathbb{A}_f$ for the ring of finite adeles of $\Q$ ($S= \mathcal{V}^\Q\setminus\set{\infty}$) and $\hat{\Z} = \prod_p \Zp$.
Througout this article we will identify $\Z^S = \Z[\frac{1}{p}:p \in S \text{ finite}]$ with its image under the diagonal embedding into $\Q_S$.

For an algebraic group $\G<\SL_d$ defined over $\Q$ we set $\G(\Z) = \SL_d(\Z) \cap \G(\Q)$ as well as $\G(\Zp) = \SL_d(\Zp)\cap \G(\Qp)$, $\G(\Z^S) = \SL_d(\Z^S) \cap \G(\Q_S)$ and so forth.
We will implicitly identify $\G(\Q_S)$ with the restricted product of the groups $\G(\Q_\sigma)$ over all $\sigma\in S$.

If $\G$ is connected and semisimple and $S$ contains the archimedean place, $\G(\Z^S)$ is a lattice in $\G(\Q_S)$ by a theorem of Borel and Harish-Chandra \cite[Thm.~5.5, 5.7]{platonov}. 
Furthermore, the \emph{$S$-arithmetic extension} of the real quotient $X_\infty = \lquot{\G(\Z)}{\G(\R)}$
\begin{align*}
X_S = \lquot{\G(\Z^S)}{\G(\Q_S)}
\end{align*}
is compact if and only if $\G$ is anisotropic over $\Q$.
The group $\G(\Q_S)$ acts on $X_S$ via $g.x = xg^{-1}$ for $x \in X_S$ and $g \in \G(\Q_S)$.

We say that $\G$ has \emph{class number one} if
$\G(\adele) = \G(\Q) \G(\R \times \widehat{\Z})$.
In this case, there are well-defined projections $X_S \to X_{S'}$ whenever $S'\subset S$ which are obtained by taking the quotient with $\G(\prod_{p\in S\setminus S'}\Zp)$ on the right.

For any dimension $d$ we equip $\Qp^d$ with the norm $\norm{a}_p = \max(|a_1|_p,\ldots ,|a_d|_p)$. 
For any $A \in \Mat_d(\Qp)$ and $x \in \Qp^d$ we have that $\norm{Ax}_p \leq \norm{A}_p \norm{x}_p$ and that $ A \in \GL_d(\Zp) $ if and only if $\norm{A}_p = 1$. The groups $\G(\Q_S)$ will always be equipped with a left-invariant metric $d$, which for $\G(\R)$ may be obtain from a left-invariant Riemannian metric and for $\G(\Qp)$ from the norm $\norm{\cdot}_p$ on $\Mat_d(\Zp)$ (see \cite{ruhr}).
The left-invariant metric on $\G(\Q_S)$ induces a left-invariant metric on $X_S$ which we will also denote by $d$.

\stoptocwriting
\subsection*{Acknowledgements} 
This project started with my master thesis. I would like to thank Manfred Einsiedler for suggesting the topic and for many enthusiastic discussions as well as Menny Akka and Manuel L\"uthi for commenting on preliminary versions of this paper.
I am also very grateful towards the anonymous referee for suggesting a clean proof of Proposition \ref{prop:CM-number of orbits} and a much improved and generalized exposition of the article.
\resumetocwriting

\tableofcontents
\section{The dynamical theorem}\label{section:formulation theorem}

In the following we fix a quaternion algebra $\quat$ over $\Q$.
The reader interested mainly in the theorems from the introduction should keep in mind the following cases:
\begin{itemize}
\item The case $\quat = \quat_{\infty,2}$ where $\quat_{\infty,2}$ denotes the quaternion algebra ramified at $\infty$ and $2$.
More explicitly, $\quat_{\infty,2}(\Q)$ is the $\Q$-algebra $\Q[\ii,\jj,\kk]$ of Hamiltonian quaternions where $\ii^2=\jj^2=\kk^2 = -1,\ \ii\jj=-\jj\ii=\kk$.
\item The totally split case $\quat = \Mat_2$.
\end{itemize}
We denote by $\Nr$ the (reduced) norm on $\quat$ (given by $\Nr(x) = x \overline{x}$ for $x \in \quat$) and by $\Tr$ the (reduced) trace on $\quat$ (given by $\Tr(x) = x+\overline{x}$ for $x \in \quat$).
Let $\pure{\quat} \subset \quat$ be the variety of traceless (pure) quaternions. 

Furthermore, we will fix throughout the whole discussion a maximal order $\mathcal{O}$ of $\quat(\Q)$ and let $\pure{\mathcal{O}}$ be the subset of pure elements in $\mathcal{O}$.
This maximal order defines an integral structure on $\quat$ (and on other groups we are yet to define).
Note that there are interesting cases where no canonical choice of a maximal order exists (cf.~Section~\ref{section:maximal orders}).
However, in the case $\quat = \quat_{\infty,2}$ the ring of Hurwitz quaternions
\begin{align*}
\mathcal{O}_{\HW} = \Z\left[\mathrm{i},\mathrm{j},\mathrm{k},\frac{1+\mathrm{i}+\mathrm{j}+\mathrm{k}}{2}\right] \subset \quat_{\infty,2}(\Q)
\end{align*}
is up to $\quat^\times(\Q)$-conjugacy the unique maximal order and the same is true for the maximal order $\Mat_2(\Z) \subset \Mat_2(\Q)$.

\begin{remark}[A general equidistribution problem]
In this setup, both problems from the introduction relate to analysing the distribution of the sets
\begin{align}\label{eq:general equi problem}
\tfrac{1}{\sqrt{|D|}}\setc{x \in \pure{\mathcal{O}}}{x \text{ primitive and } \Nr(x) = D}
\end{align}
inside the $\R$-points of one of the varieties 
\begin{align*}
\mathbf{V}^{\pm} = \setc{x \in \pure{\quat}}{\Nr(x) = \pm1}
\end{align*}
(depending on the sign of $D$) as $|D|$ goes to infinity (see \cite[Sec.~1.1]{duke} for an explicit case).
For instance, if $\quat= \quat_{\infty,2}$ and $\mathcal{O} = \mathcal{O}_{\HW}$ we may identify $\mathbf{V}^{+}(\R)$ with the sphere $\mathbb{S}^2$ and the sets in \eqref{eq:general equi problem} for $D > 0$ with the sets $\mathcal{I}_D$ defined in the introduction. 
This case will be treated in Section \ref{section:intpts on spheres}. 
For simplicity of the exposition we will however not discuss the general case in this paper.
\end{remark}

\subsection{Acting groups}\label{section:acting groups}

We denote by $\BG = \mathbf{PB}^\times$ the projective group of invertible quaternions in $\quat$ and by $\BG^{(1)}$ the $\Q$-group of norm one quaternions in $\BG$.
Note at this point that $\BG$ (resp. $\G^{(1)}$) is a $\Q$-form of $\PGL_2$ (resp.~$\SL_2$) and that any $\Q$-form of $\PGL_2$ (resp.~$\SL_2$) arises in this fashion (see for instance \cite[Chp.~III, Sec.~1.4]{serregaloiscohom}).

\subsubsection{Projective units and orthogonal groups}\label{section:def isogeny}

Notice that the group $\BG$ acts on $\pure{\quat}$ (or $\quat$) via $g.x = gxg^{-1}$ for $g \in \G$ and $x \in \pure\quat$ (or $x \in \quat$) and that no element of $\G$ acts trivially.

We represent $\Nr|_{\pure\quat}$ in a basis of $\pure{\mathcal{O}}$ to obtain a ternary form $Q$. 
The action of $\G$ on $\pure\quat$ yields an isogeny
\begin{align}\label{eq:isogeny to orthogonal}
\G \to \SO_{Q}<\SL_3
\end{align}
which is in fact an isomorphism of $\Q$-groups (cf.~\cite[Ch.~1, Thm.~3.3]{vigneras}).

\subsubsection{Integral structures}\label{section:integralstruc on G}

The integral structure on $\G$ is immediately defined by pulling back the integral structure on $\SO_Q$ under the isomorphism in \eqref{eq:isogeny to orthogonal} where the latter was introduced in Section \ref{section:intro-notation}.
So for instance, $\G(\Z)$ is the set of elements of $\G(\Q)$ which preserve $\mathcal{O}$ under conjugation and $\G(\Zp)\subset \G(\Qp)$ is the set of elements which preserve $\mathcal{O}\otimes \Zp$.
Furthermore, for any subgroup $\BH < \BG$ defined over $\Q$ we set $\BH(\Z) = \BG(\Z)\cap \BH(\Q)$, $\BH(\Zp) = \BG(\Zp) \cap \BH(\Qp)$ and so on.

Notice that in general $\G(\Z)$ is not necessarily equal to the image of the units $\mathcal{O}^\times$ under the projection to $\G(\Q)$.
For instance, if $\quat = \quat_{\infty,2}$ and $\mathcal{O} = \mathcal{O}_\HW$ the image of $\mathcal{O}^\times$ in $\G(\Z)$ has index $2$ and does for example not contain $(1 + \ii) \in \G(\Z)$.

%
%
%

\subsubsection{Acting tori}\label{section:acting tori - general}
For a vector $v \in \pure{\mathcal{O}}$ of non-zero norm we define the algebraic $\Q$-torus
\begin{align*}
\torus_v = \setc{g \in \BG}{g.v = v}.
\end{align*}
Under the isogeny in \eqref{eq:isogeny to orthogonal} the torus $\torus_v$ has a corresponding $\Q$-torus $\BH_w<\SO_{Q}$, which is the stabilizer of the vector $w \in \Z^3$ obtained from $v$ by the above choice of basis.
Naturally, one can identify $\BH_w$ with the special orthogonal group of the restriction of $Q$ to the orthogonal complement $w^\perp$ of $w$ (with respect to $Q$).

\begin{lemma}[On Linnik's condition]\label{lemma:Linnik's condition}
Let $K$ be a field of characteristic zero and let $v\in \pure{\mathcal{O}}$ be a vector of non-zero norm.
Then the torus $\torus_v$ is isotropic over $K$ if and only if the norm satisfies $\Nr(v) \in -(K^\times)^2$.
\end{lemma}

\begin{proof}
Let $Q$ be defined as in Section \ref{section:def isogeny}.
Represent the restriction of $Q$ to the orthogonal complement of $w$ (defined as above) in an orthogonal basis as $\alpha y^2+\beta z^2$.
Notice that $\torus_v$ is isotropic over $K$ if and only if $\alpha y^2+\beta z^2$ is isotropic. The latter is equivalent to $-\alpha\beta$ being a square in $K$.
Since the discriminant of $Q$ is a square in $\Q^\times$ and $Q$ is equivalent to the form $\Nr(v)x^2+\alpha y^2 +\beta z^2$, $\alpha\beta$ differs from $\Nr(v)$ by a square.
\end{proof}


Lemma \ref{lemma:Linnik's condition} together with Hensel's lemma shows that $\torus_v(\Qp)$ is split if the norm $\Nr(v)$ satisfies Linnik's condition at the prime $p$
\begin{align*}
-\Nr(v)\mod p \in (\Fp^\times)^2.
\end{align*}

\subsection{Toral orbits and packets}

The dynamical aim of this paper is to study limit measures of probability measures on (compact) toral orbits in the adelic (or $S$-arithmetic) extension of the quotient $X_\infty = \lquot{\G(\Z)}{\G(\R)}$.

\subsubsection{Adelic homogeneous space}
By the adelic extension of $\lquot{\G(\Z)}{\G(\R)}$ we mean the finite-volume quotient
\begin{align*}
X_{\adele} = \lquot{\G(\Q)}{\G(\adele)}
\end{align*}
which is compact if and only if $\quat$ is not split over $\Q$ (i.e.~$\quat\simeq \Mat_2$).
As the class number of $\BG$
\begin{align*}
\left| \lrquot{\G(\Q)}{\G(\adele)}{\G(\R \times \widehat{\Z})} \right|
\end{align*}
is possibly different from one (cf.~Section \ref{section:maximal orders}), there might not be a meaningful projection of $X_{\adele}$ onto $X_\infty$. 
Instead, $X_{\adele}$ can be partitioned into finitely many open ``components'' (orbits under $\G(\R \times \widehat{\Z})$) and there is a projection of the identity component (principal genus) onto $X_\infty$.

\subsubsection{Packets}\label{section:mainthm-packets}
Given a vector $v \in \pure{\mathcal{O}}$ with non-zero norm we can consider the orbit $\G(\Q)\torus_v(\adele)$, which we will also call the adelic \emph{packet} for the vector $v$.
If $-\Nr(v) \not\in (\Q^\times)^2$, the torus $\torus_v$ is anisotropic over $\Q$ (see Lemma \ref{lemma:Linnik's condition}) and $\G(\Q)\torus_v(\adele)$ is compact.
To complement the discussion of \cite{duke} we will in fact assume that $\torus_v$ is anisotropic over $\R$. That is, $\Nr(v) >0$ or equivalently $\torus_v(\R)$ is compact.

Let $p$ be an odd prime at which $\quat$ is split.
We will call a sequence of primitive vectors $(v_\ell)_\ell$ in $\pure{\mathcal{O}}$ \emph{admissible} (for $p$) if for every $\ell$ the norm $d_\ell = \Nr(v_\ell)$ satisfies Linnik's condition at $p$ and if $d_\ell \to \infty$ as $\ell \to \infty$.

Given such a sequence $(v_\ell)_\ell$ we may choose for any $\ell$ an element $g_{\ell,\infty} \in \G(\R)$ such that
\begin{align*}
\torus_{v_\ell}(\R) = g_{\ell,\infty} K_\infty g_{\ell,\infty}^{-1}
\end{align*}
where $K_\infty$ is a fixed choice of a proper maximal compact subgroup in $\G(\R)$.
We also let $\mu_{v_\ell}$ be the Haar measure on the shifted compact packet $\G(\Q)\torus_{v_\ell}(\adele)g_{\ell,\infty}$, normalized to be a probability measure.

\subsubsection{Main result}

For any set of places $S$ define $\sicover\BG(\Q_S)$ to be the image of $\BG^{(1)}(\Q_S)$ inside $\BG(\Q_S)$.
Linnik's Theorems as phrased in the introduction will follow from the following version of Duke's theorem (compare to \cite[Thm.~4.6]{Dukeforcubic}).

\begin{theorem}[Toral packets and limit measures]\label{thm:main}
Let $p$ be an odd prime at which $\quat$ is split.
Let $(v_\ell)_\ell$ be an sequence of primitive vectors in $\mathcal{O}^{(0)}$ which is admissible for $p$ (and in particular satisfies Linnik's condition at $p$).
Define $\mu_{v_\ell}$ as above.
Then any \wstar-limit of the measures $\mu_{v_\ell}$ is a probability measure and is invariant under the group $\sicover\G(\adele)$.
\end{theorem}

Note that stronger versions of Theorem \ref{thm:main} are known (such as \cite[Thm.~4.6]{Dukeforcubic}), but require additional input.
We also oberve the following:
\begin{itemize}
\item The shift chosen in Theorem \ref{thm:main} in the real place is in some sense articifial (cf.~\cite[Thm.~4.6]{Dukeforcubic}).
In view of the goals phrased in the introduction, it is however natural to include it as we aim to project the whole packet onto the real points $\mathbf{V}^{+}(\R)= \rquot{\G(\R)}{K_\infty}$ of the variety $\mathbf{V}^{+}$ when $\G$ has class number one.
\item In the proof of Theorem \ref{thm:main} we will destinguish two cases.
\begin{itemize}
\item $\quat$ is not totally split. 
Here the compactness of $X_{\adele}$ allows for a simplified treatment (see Section \ref{section:cocompact}) due to non-espace of mass.
\item $\quat= \Mat_2$. Here (see Section \ref{section:Equidistribution of CM points}) we use additional arguments (including a geometrical argument on the Bruhat-Tits tree of $\PGL_2(\Qp)$) to rule out escape of mass.
\end{itemize}
\end{itemize}
We deduce Linnik's Theorem A resp.~B stated in the introduction in Section \ref{section:intpts on spheres} resp.~Section \ref{subsection:cm points}.

\section{Generation of integer points}

In this section we would like to show how an orbit $\G(\Q)\torus_v(\adele)$ for $v \in \pure{\mathcal{O}}$ can generate a collection of ``integer'' points.
The absence of the class number one assumption on $\G$ implies that these integer points may lie in different maximal orders.

This procedure of generating from one integer point other integer points is in essence well-known (see for instance  \cite[Prop.~3.2]{AES15}, \cite[Thm.~5.2]{Dukeforcubic}, \cite[Thm.~8.1]{platonov} and \cite{localglobalEV}). 
For convenience we will repeat it here in our setup.
Note that a simplified discussion of what follows can be found in the first arXiv-version of this paper \cite{self-Linnikarxiv} where only the case $\quat= \quat_{\infty,2}$ is treated.

We keep the notation of Section \ref{section:formulation theorem}.

\subsection{Maximal orders}\label{section:maximal orders}

Recall that the class number of $\BG$ is finite (cf.~\cite[Thm.~5.1]{borel-classnumberfinite} or \cite[Thm.~5.1]{platonov}) i.e.~the double quotient
\begin{align}\label{eq:param max orders}
\lrquot{\G(\Q)}{\G(\adele)}{\G(\R\times \widehat{\Z})}
\simeq
\lrquot{\G(\Q)}{\G(\adele_f)}{\G(\widehat{\Z})}
\end{align}
consists of finitely many points. 
In fact, it parametrizes $\G(\Q)$-conjugacy classes of maximal orders in $\quat(\Q)$ as we will now explain.

\subsubsection{Local action on orders}
For an order $\mathcal{O}' \subset \quat(\Q)$ we denote by $\mathcal{O}'_p = \mathcal{O}'\otimes \Zp$ the completion of $\mathcal{O}'$ at $p$ and consider the map of completions
\begin{align}\label{eq:local orders}
\mathcal{O}' \text{ order in } \quat(\Q) \mapsto \widehat{\mathcal{O}'} = \mathcal{O}' \otimes \widehat{\Z}= (\mathcal{O}'_p)_p \subset \quat(\adele_f).
\end{align}
Note that $\mathcal{O}'_p$ is a $\Z_p$-order in $\quat(\Qp)$ for any prime $p$ and that $\mathcal{O}'_p = \mathcal{O}_p$ is satisfied for all but finitely primes $p$.
Recall that there is an inverse defined on tuples $(\mathcal{O}'_p)_p$ with this property (cf.~\cite[Ch.~III, Sec.~5]{vigneras}) given by
\begin{align}\label{eq:compl order inverse}
(\mathcal{O}'_p)_p \mapsto \bigcap_p (\mathcal{O}'_p \cap \quat(\Q))
\end{align}
Furthermore, an order $\mathcal{O}'\subset \quat(\Q)$ is maximal if and only if all of its completions $\mathcal{O}'_p$ are maximal (cf.~\cite[(11.2)]{reiner} or \cite[Ch.~III, Sec.~5]{vigneras}).
The map in \eqref{eq:local orders} thus yields a bijection between maximal orders $\mathcal{O}'$ in $\quat(\Q)$ and tuples $(\mathcal{O}'_p)_p$ of maximal orders $\mathcal{O}'_p\subset \quat(\Qp)$ with $\mathcal{O}'_p = \mathcal{O}_p$ for all but finitely many primes $p$.

The classification of maximal orders in quaternion algebras over local fields (see \cite[Thm.~12.8, 17.3]{reiner}) shows that the action of $\G(\adele_f)$ on maximal orders via
\begin{align*}
g.\widehat{\mathcal{O}'} = (g_p. \mathcal{O}'_p)_p = (g_p \mathcal{O}'_p g_p^{-1})_p
\end{align*}
is transitive.
Notice that the stabilizer of $\widehat{\mathcal{O}}$ under this action is exactly $\G(\widehat{\Z})$ so that the set of maximal orders up to $\G(\Q)$-conjugacy can be identified with the double quotient in \eqref{eq:param max orders}.
We extend the action of $\G(\adele_f)$ on maximal orders to an action of $\G(\adele) = \G(\R) \times \G(\adele_f)$ where $\G(\R)$ acts trivially.

\subsubsection{Examples}
Recall that if $\quat= \quat_{\infty,2}$ or $\quat= \Mat_2$ the class number of $\G$ is one or in other words there is only one maximal order up to $\G(\Q)$-conjugacy.

There are however interesting examples where this is not the case. 
For instance, let $\quat=\quat_{\infty,p}$ be the quaternion algebra over $\Q$ which is ramified exactly at $\infty$ and $p$ for an odd prime $p$.
The set of maximal orders up to $\G(\Q)$-conjugacy corresponds to the set of isomorphism classes of supersingular elliptic curves over $\overline{\Fp}$ identified up to the action of $\operatorname{Gal}(\BF_{p^2}|\BF_p)$ (cf.~\cite{deuring41}), which has cardinality $Cp+O(1)$ for some constant $C$ (cf.~\cite[Thm.~4.1]{silverman-aec}).

\subsubsection{Choices}
In the following we will fix a set of representatives
\begin{align*}
\mathcal{O}_1,\ldots,\mathcal{O}_n \subset \quat(\Q)
\end{align*}
for the $\G(\Q)$-conjugacy classes of maximal orders with $\mathcal{O}_1 = \mathcal{O}$.
For every $k\in \set{1,\ldots,n}$ we also fix an element $g^{(k)} \in \G(\adele_f)$ with $g^{(k)}_p.\mathcal{O}_p = (\mathcal{O}_k)_p$ for every prime $p$.
Note that the double cosets $\G(\Q)g^{(k)}\G(\R\times\widehat{\Z})$ are representatives of the fibers for the reduction map
\begin{align}\label{eq:fibers for genus of G}
X_{\adele} = \lquot{\G(\Q)}{\G(\adele)} \to \lrquot{\G(\Q)}{\G(\adele)}{\G(\R \times \widehat{\Z})}
\end{align}
or in other words representatives of the $\G(\R\times \widehat{\Z})$-orbits in $X_{\adele}$.

For a given $k\in \set{1,\ldots,n}$ one can define, just as in Section \ref{section:integralstruc on G}, $\Z$-points of $\BG$ with respect to $\mathcal{O}_k$. We will denote these by $\Gk{k}(\Z)$. 
The groups $\Gk{k}(\Zp)$ for primes $p$ and $\Gk{k}(\widehat{\Z})$ are also defined as in Section \ref{section:integralstruc on G} for $\mathcal{O}$ replaced by $\mathcal{O}_k$.
It follows directly from \eqref{eq:compl order inverse} that $\Gk{k}(\R \times \widehat{\Z}) \cap \G(\Q) = \Gk{k}(\Z)$.

We note at this point that the different maximal orders in $\quat(\Q)$ correspond to the forms in the genus of the quadratic form $\Nr|_{\mathcal{O}^{(0)}}$ (see \cite[Ch.~22]{voight} for more details).
%

\subsection{Integer points}\label{section:generating intpts}
%
%

As mentioned at the beginning of this section we would like to generate from one primitive integer point $v \in \pure{\mathcal{O}}$ of norm $D$ other primitive integer points of the same norm.
The basic procedure is the following.

\begin{lemma}\label{lemma:genintpts}
Let $v \in \pure{\mathcal{O}}$ be a primitive vector.
For $h \in \torus_v(\adele)$ we choose $\gamma \in \G(\Q)$, $k \in \set{1,\ldots,n}$ and $g_{\operatorname{cpt}} \in \G(\R\times \widehat{\Z})$ such that $\gamma h = g^{(k)} g_{\operatorname{cpt}}$ holds. 
Then
\begin{align*}
\gamma.v = g^{(k)}g_{\cpt}.v
\end{align*}
is a primitive element of the lattice $\pure{\mathcal{O}_k}\subset\pure\quat(\Q)$.
\end{lemma}

\begin{proof}
Clearly, $v' = \gamma.v\in \quat(\Q)$ is pure. 
To see that $v' \in \mathcal{O}_k$ notice that
\begin{align*}
v' = \gamma h. v = g^{(k)} g_{\operatorname{cpt}}.v 
\in g^{(k)} g_{\operatorname{cpt}}. \widehat{\mathcal{O}}
= g^{(k)}. \widehat{\mathcal{O}} = \widehat{\mathcal{O}_k}
\end{align*}
and therefore, $v' \in \widehat{\mathcal{O}_k} \cap \quat(\Q) = \mathcal{O}_k$ by \eqref{eq:compl order inverse}.

It remains to show that $v'$ is primitive.
If not, there is a prime $p$ and $n \geq 1$ so that $p^{-n}v' \in (\mathcal{O}_k)_p$.
But then $p^{-n} v = p^{-n}(g^{(k)}g_{\cpt})_p^{-1}.v' \in \mathcal{O}_p$ so $v$ is also non-primitive.
\end{proof}

We will call the point $\gamma.v$ as above an \emph{integer point produced by $\G(\Q)h$}.
Since the $g^{(k)}$'s were chosen to be representatives for \eqref{eq:fibers for genus of G}, the choice of $k$ in the lemma above is unique.
A priori, the point $\gamma.v$ may however depend on the choice of $\gamma$ and $g_{\operatorname{cpt}}$.

\begin{lemma}[Independence of choices]\label{lemma:indep choices}
Let $v \in \pure{\mathcal{O}}$ be a primitive vector. 
Given $h \in \torus_v(\adele)$ we let $\gamma,\gamma' \in \G(\Q)$, $k \in \set{1,\ldots,n}$ and $g_{\operatorname{cpt}},g_{\operatorname{cpt}}' \in \G(\R\times \widehat{\Z})$ be such that $\gamma h = g^{(k)} g_{\operatorname{cpt}}$ and $\gamma' h = g^{(k)} g_{\operatorname{cpt}}'$.
Then the points $\gamma.v,\gamma'.v \in \mathcal{O}_k$ differ by conjugation with an element in $\Gk{k}(\Z)$.
\end{lemma}

\begin{proof}
Consider the group element $g = \gamma' \gamma^{-1} = g^{(k)}g_{\operatorname{cpt}}'g_{\operatorname{cpt}}^{-1}(g^{(k)})^{-1} \in \G(\Q)$ which satisfies $g.(\gamma.v) = \gamma'.v$ and $g.\widehat{\mathcal{O}_k} = \widehat{\mathcal{O}_k}$ (i.e.~$g \in \Gk{k}(\R \times\widehat{\Z})$).
We thus conclude~that $g\in \Gk{k}(\Z) = \G(\Q) \cap \Gk{k}(\widehat{\Z})$ as desired.
\end{proof}

Lemma \ref{lemma:indep choices} shows that we have obtained a well-defined point
\begin{align*}
[\gamma.v] \in \lquot{\Gk{k}(\Z)}{\pure{\mathcal{O}_k}}
\end{align*}
which we will also call the integer point produced by the points $\G(\Q)h$.
Conversely, one can ask when two points $\G(\Q)h, \G(\Q)h'$ yield the same integer point.

\begin{lemma}[Same production]\label{lemma:same production}
Let $\G(\Q)h, \G(\Q)h'\in \G(\Q)\torus_v(\adele)$ be two points in the $\G(\R \times \widehat{\Z})$-orbit through $\G(\Q)g^{(k)}$.
Then they produce the same integer points modulo $\Gk{k}(\Z)$ if and only if they lie on the same $\torus_v(\R \times \widehat{\Z})$-orbit.
\end{lemma}

\begin{proof}
Suppose first that the points $\G(\Q)h, \G(\Q)h'$ lie on the same $\torus_v(\R \times\widehat{\Z})$-orbit i.e.~$\G(\Q)h\tilde{h} = \G(\Q) h'$ for some $\tilde{h}\in \torus_v(\R \times \widehat{\Z})$.
Writing $\G(\Q) h = \G(\Q)g^{(k)}g_{\cpt}$ for some $g_{\cpt} \in \G(\R \times \widehat{\Z})$ we obtain that $ \G(\Q) h' = \G(\Q)g^{(k)}g_{\cpt}\tilde{h}$.
Therefore, an integer point produced by $\G(\Q) h'$ is $g^{(k)}g_{\cpt}\tilde{h}.v = g^{(k)}g_{\cpt}.v$.
In particular, the integer points produced by $\G(\Q)h$ and $\G(\Q)h'$ are the same modulo $\Gk{k}(\Z)$.

Conversely, assume that there are $\gamma,\gamma' \in \G(\Q)$ with $\gamma.v = \gamma'.v$ and
\begin{align*}
(g^{(k)})^{-1}\gamma h,(g^{(k)})^{-1}\gamma'h' \in \G(\R \times\widehat{\Z})
\end{align*}
Consider the element $\tilde{h} = h^{-1}\gamma^{-1}\gamma'h'\in \G(\adele)$. By assumption, we have $\gamma.v = \gamma'.v$ so $\tilde{h} \in \torus_v(\adele)$. Furthermore, $\tilde{h}  \in \G(\R \times\widehat{\Z})$
as
\begin{align*}
\tilde{h}.\widehat{\mathcal{O}} 
= h^{-1}\gamma^{-1}g^{(k)}.(g^{(k)})^{-1}\gamma'h'.\widehat{\mathcal{O}}
=\widehat{\mathcal{O}}.
\end{align*}
Therefore, $\tilde{h} \in \G(\R \times \widehat{\Z}) \cap \torus_v(\adele) = \torus_v(\R \times \widehat{\Z})$.
This implies that $\G(\Q)h\tilde{h} = \G(h')$ and finishes the proof of the lemma.
\end{proof}

\subsubsection{Transitivity}
In this short section we would like to treat the following question:
\begin{quote}
\emph{
Let $v\in \pure{\mathcal{O}}$ be a primitive vector and for some $k\in \set{1,\ldots,n}$ let $w \in \pure{\mathcal{O}_k}$  be a primitive vector with $\Nr(w) = \Nr(v)=:D$.
When is $w$ produced by (a point in) the stabilizer orbit of $v$?}
\end{quote}
By definition this is the case if and only if there exists $h \in \torus_v(\mathbb{A})$ and $\gamma \in \G(\Q)$ so that $\gamma.v=w $ and $(g^{(k)})^{-1}\gamma h \in \G(\R \times \hat{\Z})$.
This in turn is equivalent to the existence of some $g_{\cpt} \in \G(\R \times \widehat{\Z})$ with $g^{(k)}g_{\cpt}.v = w$.
Indeed, by Witt's Theorem (see e.g.~\cite[p.~21]{Cassels}) there is some $\ell \in \SO_{\Nr}(\Q)$ which maps $v$ to $w$.
Any such orthogonal transformation $\ell$ is given by conjugation with an element of $\quat^\times(\Q)$ i.e.~there is $\gamma\in \G(\Q)$ realizing $\ell$ so that in particular $\gamma.v = w$ \cite[Thm.~3.3]{vigneras}.

Since $(g^{(k)})^{-1}.w, v$ are two primitive elements of $\widehat{\mathcal{O}}$, the question whether or not $w$ is produced by the stabilizer orbit of $v$ can thus be answered locally.

\begin{proposition}[Transitivity]\label{prop:transitivity}
Let $p$ be an odd prime and let $v_1,v_2 \in \mathcal{O}_p^{(0)}$ be two primitive elements of equal norm.
Then there exists $g_p \in \mathcal{O}_p^\times$ with $g_p.v_1 = v_2$.

If $p=2$ then there are at most two $\mathcal{O}_2^\times$-orbit on the set of primitive elements of $\mathcal{O}_2^{(0)}$ of norm $D$ for $D\in \Z_2^\times$.
\end{proposition}

We refer to \cite[Prop.~3.7]{Linnikthmexpander} for a proof.
It shows that in the case $\quat = \quat_{\infty,2}$ (where $\quat(\Q_2)$ is division algebra), any point $w \in \pure{\mathcal{O}_k}=\pure{\mathcal{O}_1}$ with $\Nr(v) = \Nr(w)$ is produced by the stabilizer orbit of $v$.


\subsection{Packets}\label{section:packets}

Let $v \in \pure{\mathcal{O}}$ be primitive of norm $D>0$.
By the discussion of the previous subsection we may group the $\torus_v(\R \times \widehat{\Z})$-orbits in $\G(\Q)\torus_v(\adele)$ first according to the fibers for \eqref{eq:fibers for genus of G} and second according to the produced integer points.
Given any $k \in \set{1,\ldots,n}$ we may write for each fiber
\begin{align*}
\G(\Q)\torus_v(\adele) \cap \G(\Q)g^{(k)}\G(\R \times \widehat{\Z})
= \bigsqcup_{\rho\in \mathcal{R}_v(k)} \G(\Q)g^{(k)}\rho\torus_v(\R \times \widehat{\Z})
\end{align*}
for a finite (possibly empty) set of representatives $\mathcal{R}_v(k) \subset \G(\R\times\widehat{\Z})$.
For any $\rho\in \mathcal{R}_v(k)$ the vector $g^{(k)}\rho.v$ is then a primitive element of $\pure{\mathcal{O}_k}$ of norm $D$ (Lemma~\ref{lemma:genintpts}), is independent of the choice of the set $\mathcal{R}_v(k)$ up to $\Gk{k}( \Z)$-equivalence and for any other $\rho'\in \mathcal{R}_v(k)$ we have $g^{(k)}\rho.v \neq g^{(k)}\rho'.v$ modulo $\Gk{k}( \Z)$ (see Lemmas \ref{lemma:indep choices} and \ref{lemma:same production}).

\subsubsection{Volume on the packet}\label{section:def volume}
We equip the packet $\G(\Q)\torus_v(\adele)$ with the volume $\vol$ induced by choosing the normalized Haar measure on $\torus_v(\R \times \widehat{\Z})$ (recall that $\torus_v(\R)$ is compact).
In particular, we have
\begin{align}\label{eq:equal volume}
\vol\big(\G(\Q)g^{(k)}\rho\torus_v(\R \times \widehat{\Z})\big) = \frac{1}{|\torus_v(\Z)|}
\end{align}
for any $k$ and any $\rho \in \mathcal{R}_v(k)$ i.e.~all $\torus_v(\R \times \widehat{\Z})$-orbits get the same mass.
Indeed, observe that
\begin{align*}
\Stab_{\torus_v(\R \times \widehat{\Z})}(\G(\Q)g^{(k)}\rho)
&= \Stab_{\torus_v(\R \times \widehat{\Z})}(\G(\Q)h)
= \Stab_{\torus_v(\R \times \widehat{\Z})}(\G(\Q))\\
&= \G(\Q) \cap \torus_v(\R \times \widehat{\Z})
= \torus_v(\Z)
\end{align*}
where $h \in \torus_v(\adele)$ is chosen with $\G(\Q)h = \G(\Q)g^{(k)}\rho$ and where we used that $\torus_v(\adele)$ is an abelian group.
Recall that $\torus_v(\Z)$ corresponds to the group of integral isometries of the restriction of the norm form $\Nr$ to the orthogonal complement of $v$ in $\pure\quat$ (in a basis of $\pure{\mathcal{O}}$) which is a positive definite binary form. Thus, $|\torus_v(\Z)| \asymp 1$ as $D \to \infty$.
Note that if $\quat=\quat_{\infty,2}$ we have $|\torus_v(\Z)|\leq |\G(\Z)|<\infty$. See Claim \ref{claim:Q-points torus} for the case $\quat=\Mat_2$.

\subsubsection{Volume bounds}
One central number-theoretical ingredient in the proof of Theorem \ref{thm:main} is the following estimate on the size of the volume of the packet $\G(\Q)\torus_v(\adele)$ (or equivalently on the number of orbits by the discussion above).

\begin{proposition}[Total volume]\label{prop:total volume}
Let $v \in \pure{\mathcal{O}}$ be primitive of norm $D$.
Then we have
\begin{align*}
\vol(\G(\Q)\torus_v(\adele)) = D^{\frac{1}{2}+o(1)}.
\end{align*}
\end{proposition}

A complete treatment in the case $\quat = \quat_{\infty,2}$ can be found in \cite[Sec.~6.2]{Linnikthmexpander}.
The case $\quat = \Mat_2$ will be discussed in detail in Proposition \ref{prop:CM-number of orbits} later.
Let us point out the main tools and steps of reduction here. 

\subsubsection{Optimal embeddings}
As $v^2 = -v\overline{v} = -D$ holds there is an induced field embedding
\begin{align*}
\iota_v:\Q(\sqrt{-D}) \to \quat(\Q),\ a+b\sqrt{-D} \mapsto a+bv.
\end{align*}
Let $d= -D$ if $D \equiv 3 \mod 4$ and $d = -4D$ if $D\equiv 0,1,2 \mod 4$.
As $v$ is primitive we have that $\iota_v^{-1}(\mathcal{O})$ is the order of discriminant $d$
\begin{align*}
\iota_v^{-1}(\mathcal{O}) = R_d := \Z\Big[\frac{d+\sqrt{d}}{2}\Big].
\end{align*}
Conversely, given an embedding $\iota:\Q(\sqrt{-D}) \to \quat(\Q)$ with $\iota^{-1}(\mathcal{O}) = R_d$ (that is, an \emph{optimal embedding} of $R_d$ into $\mathcal{O}$) the image of $\sqrt{-D}$ is a primitive point in $\pure{\mathcal{O}}$ of norm $D$.
Note that if $v'$ is an other primitive vector in $\pure{\mathcal{O}}$ with $\gamma.v=v'$ for some $\gamma\in \G(\Z)$ then the induced optimal embeddings are equivalent in the sense that $\iota_{v'}(x) = \gamma.\iota_v(x)$ for all $x \in \Q(\sqrt{-D)}$.

\subsubsection{Counting optimal embeddings}\label{section:optimal embeddings}

From Section \ref{section:def volume} we know that the volume $\vol(\G(\Q)\torus_v(\adele))$ is up to a bounded factor equal to the class number of $\torus_v$ i.e.~the cardinality of the finite group
\begin{align*}
\lrquot{\torus_v(\Q)}{\torus_v(\adele)}{\torus_v(\R \times \widehat{\Z})}.
\end{align*}
In Section \ref{section:generating intpts} we showed that this double quotient generates points in $\lquot{\Gk{k}(\Z)}{\pure{\mathcal{O}_k}}$ for varying $k$.
As we explained above, the points in $\lquot{\Gk{k}(\Z)}{\pure{\mathcal{O}_k}}$ for a given $k$ are exactly the optimal embeddings of $R_D$ into $\mathcal{O}_k$ up to $\Gk{k}(\Z)$-equivalence.
By Proposition \ref{prop:transitivity} it suffices to estimate the number $r(D,k)$ of such equivalence classes of optimal embeddings for every $k$ or more precisely the sum $r(D) = \sum_k r(D,k)$.

It remains to explain why $r(D) = D^{\frac{1}{2}+o(1)}$.
One can show that (see for instance the preprint \cite[Ch.~30]{voight})
\begin{align*}
r(D) = c_D |\Cl(R_d)|
\end{align*}
where the factor $c_D$ can be computed from local quantities and can only take a finite number of values as a function in $D$.
It therefore suffices to understand the size of the Picard group $\Cl(R_d)$ (the group of invertible $R_d$-ideals -- see also Section~\ref{section:Picard group}).

\subsubsection{Size of Picard groups}\label{section:size picard group}
For square-free $D$ the order $R_d$ is simply the ring of integers in $\Q(\sqrt{-D})$. 
It is a consequence of Siegel's lower bound (cf.~\cite{siegellowerbound} or \cite[Thm.~5.28]{iwaniec-kowalski}) that
\begin{align*}
|\Cl(R_d)| = D^{\frac{1}{2}+o(1)}.
\end{align*}
The non-square-free case can be reduced to this (cf.~\cite[Thm.~7.24]{coxprimesoftheform}).
\section{The proof in the cocompact case}\label{section:cocompact}

We continue using the notation from the previous sections.
In this section we will exhibit a proof of Theorem \ref{thm:main} in the case where
\begin{align*}
X_\adele = \lquot{\G(\Q)}{\G(\adele)}
\end{align*}
is compact.
This case would for instance suffice to prove Linnik's Theorem A as stated in the introduction and is simpler than $\BG = \PGL_2$ where we need to rule out escape of mass as well.

\subsection{Conjugacy of stabilizer subgroups}
To be able to use dynamical arguments (cf.~Section \ref{section:maximal entropy}), we will first adapt the packets in Theorem \ref{thm:main} by an element in $\G(\Zp)$ so that their Haar measure are invariants under a common split torus $\torus(\Qp)$ in $\G(\Q_p)$.

Throughout this section, the vector $v$ is a fixed primitive element of $\pure{\mathcal{O}}$ (or $\pure{\mathcal{O}}\otimes\Zp = \pure{\mathcal{O}_p}$) of norm $D$ satisfying Linnik's condition at $p$.
Recall that this condition asserts that $\torus_v(\Qp)$ is split.
The following is a slight refinement of this splitting.

\begin{lemma}[Split tori]\label{lemma:on congruence condition}
There exists an isomorphism $f_v:\quat(\Q_p) \to \Mat_2(\Q_p)$ such that $f_v(\mathcal{O}_p) = \Mat_2(\Z_p)$ and
\begin{align*}
f_v(\setc{a+bv}{a,b\in \Q_p}) = \left\lbrace \begin{pmatrix}
x &  \\ 
 & y
\end{pmatrix} : x,y \in \Qp \right\rbrace.
\end{align*}
In particular, the induced isomorphism $\tilde{f}_v:\G(\Q_p) \to \PGL_2(\Q_p)$ satisfies
\begin{align*}
\tilde{f}_v(\torus_v(\Qp))= \left\lbrace \begin{pmatrix}
x &  \\ 
 & 1
\end{pmatrix} : x \in \Qp^\times \right\rbrace
\end{align*}
and similarly over $\Z_p$.
\end{lemma}

\begin{proof}
By assumption on the prime $p$ we have $\quat(\Q_p) \simeq \Mat_2(\Q_p)$. Since the image of $\mathcal{O}_p$ inside $\Mat_2(\Q_p)$ is a maximal order and all maximal orders in $\Mat_2(\Q_p)$ are conjugate (cf.~\cite[Thm.~17.3]{reiner}), we may adapt this isomorphism so that the image of $\mathcal{O}_p$ is $\Mat_2(\Z_p)$.

Let $f:\quat(\Q_p) \to \Mat_2(\Q_p)$ be the so obtained isomorphism.
Now notice that $f(v)\in \Mat_2(\Q_p)$ is traceless and satisfies
\begin{align*}
f(v)^2 = f(v^2) = -D.
\end{align*}
An elementary computation shows that there exists $g \in \GL_2(\Z_p)$ such that $g f(v) g^{-1}$ is diagonal (alternatively, see \cite[Prop.~3.7]{Linnikthmexpander}). 
Then $f_v:w \mapsto gf(w)g^{-1}$ satisfies all desired properties.
\end{proof}

\subsubsection{Conjugacy}\label{sec:conjugacy}
Let $v' \in \pure{\mathcal{O}}$ be another primitive element for which $\Nr(v')$ satisfies Linnik's condition at $p$.
We claim that there is $\alpha_{\cpt} \in \G(\Zp)$ with
\begin{align*}
\alpha_{\cpt} \torus_v(\Zp) \alpha_{\cpt}^{-1} = \torus_{v'}(\Zp), \quad \alpha_{\cpt} \torus_v(\Qp) \alpha_{\cpt}^{-1} = \torus_{v'}(\Qp).
\end{align*}
For this, notice first that the matrix $f_v(v')$ is $\GL_2(\Zp)$-conjugate to some traceless diagonal matrix (by the proof of Lemma \ref{lemma:on congruence condition}).
Let $g \in \GL_2(\Zp)$ be such that $g f_v(v')g^{-1}$ is diagonal and consider $\alpha = f_v^{-1}(g)\in \mathcal{O}_p^\times$.
By these choices $\alpha.v'$ is $\Zp^\times$-multiple of $v$ and hence $\alpha_{\cpt} = \alpha^{-1}\in \G(\Zp)$ has all required properties.

We would like to remark at this point that such a conjugating element $\alpha_{\cpt}$ can also be found by elementary arguments using the Gram-Schmidt process (see \cite{self-Linnikarxiv}).

\subsection{Maximal entropy}\label{section:maximal entropy}

In this section we formulate the dynamical ingredient of Theorem \ref{thm:main}.
To do so, we fix some primitive $v \in \pure{\mathcal{O}}$ so that $\Nr(v)$ fulfills Linnik's condition at $p$.
We will consider dynamics under the $p$-adic points of the torus $\torus_v =: \torus$ or more specifically under the fixed element $a = \tilde{f}_v^{-1}(\operatorname{diag}(p,1)) \in \torus(\Qp)$ (cf.~Lemma~\ref{lemma:on congruence condition}).

Now let $(v_\ell)_{\ell}$ be an admissible sequence of primitive vectors in $\pure{\mathcal{O}}$ as in Theorem~\ref{thm:main}.
By Section \ref{sec:conjugacy} we may choose for any $\ell$ a conjugating element $\alpha_\ell \in \G(\Zp)$ such that $\alpha_{\ell} \torus_v(\Zp) \alpha_{\ell}^{-1} = \torus_{v'}(\Zp)$.
We also let $g_{\ell,\infty}\in \G(\R)$ be as in Section \ref{section:mainthm-packets} and denote
\begin{align*}
P_\ell = \G(\Q) \torus_{v_\ell}(\adele)g_{\ell,\infty} \alpha_\ell
\end{align*}
for simplicity.

Since $\G(\Zp)$ is a compact group we may replace the packets in Theorem \ref{thm:main} by the sequence of packets $P_\ell$ and show the analogous statement for these packets.
Let $\mu_\ell$ be the normalized Haar measure on $P_\ell$, which is by definition of $\alpha_\ell$ invariant under $\torus(\Qp)$ and in particular under $a$.

\subsubsection{Entropy}
Recall that for a Borel probability measure $\nu$ on $X_\adele$ the \emph{entropy} of a finite partition $\mathcal{P}$ of $X_\adele$ is defined as
\begin{align*}
H_\nu(\mathcal{P}) = - \sum_{P \in \mathcal{P}} \nu(P) \log(\nu(P)).
\end{align*}
The \emph{(dynamical) entropy} of the measure $\nu$ with respect to $a$ is then given by
\begin{align*}
h_\nu(a) 
= \sup_{\mathcal{P} \text{ finite}} \big(\lim_{N \to \infty} \tfrac{1}{N} H_\nu(\mathcal{P}_{0}^N)\big)
\end{align*}
Here, for $N_1<N_2$ the partition $\mathcal{P}_{N_1}^{N_2}$ is given by the refinement $\bigvee_{n=N_1}^{N_2}a^{n}.\mathcal{P}$.
For a more thorough introduction to entropy we refer to the book project \cite{vol2}.

\subsubsection{Maximal entropy}

We will say that an $a$-invariant probability measure $\nu$ on $X_\adele$ has \emph{maximal entropy} for $a$ if $h_\nu(a) = \sup_{\nu'}h_{\nu'}(a)$ where the supremum is taken over all $a$-invariant probability measures $\nu'$ on $X_\adele$.
Note that in our case the supremum is finite and (for the specific choice of $a$) given by $\log(p)$ (see \eqref{eq:log(p) entropy} below).

\begin{theorem}[Maximal Entropy]\label{thm: Maximal entropy}
Let $(P_\ell)_\ell$ be the sequence of packets defined above and for each $\ell$ let $\mu_\ell$ be the normalized Haar measure on the packet $P_\ell$.
Then any \wstar-limit of the measures $\mu_\ell$ has maximal entropy with respect to $a$.
\end{theorem}

Notice that by compactness of $X_\adele$ any \wstar-limit of the measures $\mu_\ell$ is automatically a probability measure.
As mentioned in the introduction, the proof of Theorem \ref{thm: Maximal entropy} will, roughly speaking, use that the collection of orbits appearing in each packet is somewhat rich (see Proposition~\ref{prop:total volume}) and sparse.
In Section \ref{subsection: From maximal entropy to additional invariance} we will see how Theorem \ref{thm: Maximal entropy} implies Theorem \ref{thm:main}.

\subsection{Exponential map and horospherical subgroups}\label{subsection:exponential map}

For later purposes we recall here a few facts about the exponential map on (certain) $p$-adic Lie groups and about horospherical subgroups.
Let $p$ be an odd prime and $\mathbb{G}<\SL_d$ be an algebraic $\mathbb{Q}$-group.
The norm $\norm{A}_p = \max_{ij}|A_{ij}|_p$ on $\Mat_d(\Qp)$ induces a norm on the Lie algebra $\mathfrak{g}$ of $G = \G(\Qp)$ by restriction.
For simplicity, denote by $B_K^\mathfrak{g}$ resp.~$B_K^G$ the ball of radius $p^{-K}$ in $\mathfrak{g}$ resp. $G$ around $0$ resp.~the identity $I$ for the remainder of this subsection.
Just as for real linear groups, one can define a matrix exponential on $\Mat_d(\Qp)$ by the formula
\begin{align*}
\exp(A) = \sum_{n \geq 0} \frac{A^n}{n!}.
\end{align*}
We take the following facts for granted; proofs may be found in \cite{platonov}, \cite{ruhr} and \cite{serre-lie}:
\begin{enumerate}[(i)]
\item  The exponential map $\exp$ is defined on $B_1^\mathfrak{g}$ and forms an isometric bijection $\exp:B_1^\mathfrak{g} \to B_1^G$. It maps Lie subalgebras to Lie subgroups.
\item The image of a $\Zp$-subalgebra of $\mathfrak{g}$ (a $\Zp$-submodule which is stable under taking commutators) is a subgroup of $G$. In particular, every ball of radius less or equal $~p^{-1}$ is a subgroup of $G$, since $B_K^\mathfrak{g}$ is a $\Zp$-subalgebra of $\mathfrak{g}$ for every $K$.
\end{enumerate}

\subsubsection{Horospherical subgroups}
Let $a \in G$ be a diagonalizable element. Define the \textbf{stable/unstable horospherical subgroups} associated to $a$ as
\begin{align*}
G_a^- &= \setc{g \in G}{a^nga^{-n}\to e \ \text{as}\ n\to \infty} \\
G_a^+ &= \setc{g \in G}{a^nga^{-n}\to e \ \text{as}\ n\to -\infty}
\end{align*}
and let $G_a^0 = C_G(a)$ be the centralizer of $a$. The groups $G_a^-,G_a^+, G_a^0$ are closed subgroups and the Lie algebras corresponding to the horospherical subgroups are
\begin{align*}
\mathfrak{g}_a^\mp &= \setc{X \in \mathfrak{g}}{\Ad_a^n(X) \to 0 \ \text{as}\ n\to \pm\infty}.
\end{align*}
Moreover, the Lie algebra $\mathfrak{g}_a^-$ is the direct sum of the eigenspaces of $\Ad_a$ associated eigenvalues of norm (strictly) less than one and the analogous statements hold for $\mathfrak{g}_a^+$ and $\mathfrak{g}_a^0$ as $a$ is diagonalizable. We have the decomposition 
\begin{align}\label{eq:lie-decomposition lie algebra into horo.subalg.}
\mathfrak{g} = \mathfrak{g}_a^+ + \mathfrak{g}_a^- + \mathfrak{g}_a^0.
\end{align}
The main example relevant for our purposes (compare to Lemma \ref{lemma:on congruence condition}) is the following:

\begin{example}\label{exp: horospherical subgrps for PGL_2}
Let $\G = \PGL_2$ and $ a = \left(\begin{smallmatrix}
p &   \\ 
  & 1
\end{smallmatrix}\right) \in \PGL_2(\Qp) $. A direct computation shows that
 \begin{align*}
&G_a^- = \left\lbrace \begin{pmatrix}
1 & x \\ 
0 & 1
\end{pmatrix}: \ x \in \Qp \right\rbrace, \
G_a^+ = \left\lbrace \begin{pmatrix}
1 & 0 \\ 
x & 1
\end{pmatrix}: \ x \in \Qp \right\rbrace, \\
&G_a^0 = \left\lbrace \begin{pmatrix}
x & 0 \\ 
0 & 1
\end{pmatrix}: \ x \in \Qp^\times \right\rbrace
\end{align*}
as well as the identities in $\mathfrak{pgl}_2 = \mathfrak{sl}_2$
\begin{align*}
&\mathfrak{g}_a^- = \left\lbrace \begin{pmatrix}
0 & x \\ 
0 & 0
\end{pmatrix}: \ x \in \Qp \right\rbrace, \
\mathfrak{g}_a^+ = \left\lbrace \begin{pmatrix}
0 & 0 \\ 
x & 0
\end{pmatrix}: \ x \in \Qp \right\rbrace, \\
& \mathfrak{g}_a^0 = \left\lbrace \begin{pmatrix}
x & 0 \\ 
0 & -x
\end{pmatrix}: \ x \in \Qp \right\rbrace.
\end{align*}
The Lie algebra $\mathfrak{g}_a^-$ is the eigenspace of $\Ad_a$ for the eigenvalue $p$ and $\mathfrak{g}_a^+$ is the eigenspace of $\Ad_a$ for the eigenvalue $p^{-1}$.
Notice that the subgroup generated by the horospherical subgroups is exactly the image of $\SL_2(\Qp)$ in $\PGL_2(\Qp)$.

If $\G = \mathbf{PB}^\times$ and $\quat$ is split at $p$ (which is a standing assumption for us), the group $\G$ is isomorphic to $\PGL_2$ over $\Qp$ and the group of norm one quaternions $\G^{(1)}$ is isomorphic to $\SL_2$ over $\Qp$.
Thus, the subgroup generated by the horospherical subgroups for $a \in \G(\Qp)$ as in Section \ref{section:maximal entropy} is exactly the image $\sicover\G(\Qp)$ of $\G^{(1)}(\Qp)$ in $\G(\Qp)$.
\end{example}

\subsubsection{Open rectangles}\label{section:open rectangles}
We will usually consider rectangles of the kind
\begin{align*}
B_{K_+}^{\mathfrak{g}_a^+} + B_{K_-}^{\mathfrak{g}_a^-} + B_{K_0}^{\mathfrak{g}_a^0}
\end{align*}
for $K_+,K_-,K_0 \geq 1$ instead of balls in $\mathfrak{g}$ as these are well-behaved with respect to conjugation by $a$ (see Lemma \ref{lemma: Explicit Bowen balls}).
 These sets are open and induce the topology on $\mathfrak{g}$. In fact, by equivalence of norms there exists some $L\geq 0$ so that for all $K$
\begin{align}\label{eq:norm equivalence on frak g}
B_{K+L}^{\mathfrak{g}} \subset B_{K}^{\mathfrak{g}_a^+}+B_{K}^{\mathfrak{g}_a^-}+B_{K}
^{\mathfrak{g}_a^0} \subset B_{K-L}^{\mathfrak{g}}.
\end{align}
Also, note that $B_{K_+}^{\mathfrak{g}_a^+} + B_{K_-}^{\mathfrak{g}_a^-} + B_{K_0}^{\mathfrak{g}_a^0}$ is a $\Zp$-subalgebra; its image is thus a subgroup, which is explicitly given by
\begin{align}\label{eq: horospherical subgrps vs exp}
\exp\left(B_{K_+}^{\mathfrak{g}_a^+}+B_{K_-}^{\mathfrak{g}_a^-}+B_{K_0}
^{\mathfrak{g}_a^0}\right) = B_{K_+}^{G_a^+}B_{K_0}^{G_a^0}B_{K_-}^{G_a^-},
\end{align} 
where we may permute the factors on the right hand side. A proof of this fact based on the $p$-adic version of the Baker-Campbell-Hausdorff formula may be found in \cite{ruhr}. This together with \eqref{eq:norm equivalence on frak g} implies that there is some $L \geq 0$ so that
\begin{align}\label{eq: equivalence of balls in G(Qp)}
B_{K+L}^{G}\subset B_{K}^{G_a^0}B_{K}^{G_a^+}B_{K}^{G_a^-} \subset B_{K-L}^{G}
\end{align}
for all large enough $K$.

\begin{lemma}\label{lemma: Explicit Bowen balls}
For $\G = \PGL_2$ and any $N_1,N_2 \geq 0$ we have
\begin{align*}
\Intersect_{k=-N_1}^{N_2} a^{-k} B_{K}^{G_a^+}B_{K}^{G_a^0}B_{K}^{G_a^-} a^k
= B_{K+N_1}^{G_a^+}B_{K}^{G_a^0}B_{K+N_2}^{G_a^-}.
\end{align*}
\end{lemma}

Note that Lemma \ref{lemma: Explicit Bowen balls} holds in greater generality (see Lemma 4.1 in \cite{ruhr}).

\begin{proof}
The statement is true on the Lie-algebra level as
\begin{align*}
\Intersect_{k=-N_1}^{N_2} a^{-k}(B_{K}^{\mathfrak{g}_a^+}+B_{K}^{\mathfrak{g}_a^0}+B_{K}^{\mathfrak{g}_a^-}) a^k
&=\Intersect_{k=-N_1}^{N_2}  B_{K+k}^{\mathfrak{g}_a^+}+B_{K}^{\mathfrak{g}_a^0}+B_{K-k}^{\mathfrak{g}_a^-}\\
&= B_{K+N_1}^{\mathfrak{g}_a^+}+B_{K}^{\mathfrak{g}_a^0}+B_{K+N_2}^{\mathfrak{g}_a^-}
\end{align*}
by Example \ref{exp: horospherical subgrps for PGL_2}. Applying the exponential map and using Equation \ref{eq: horospherical subgrps vs exp}, one obtains the claim.
\end{proof}

We remark here that the maximal entropy of $a\in \torus(\Qp)$ as in Section \ref{section:maximal entropy} is
\begin{align}\label{eq:log(p) entropy}
-\log|\det(\Ad_a|_{\mathfrak{g}_a^-})| = \log(p).
\end{align}
This follows from \cite[Thm.~7.9]{pisa} (or \cite[Thm.~8.19]{vol2}) and Example \ref{exp: horospherical subgrps for PGL_2}.
As $X_\adele$ is compact, this is in fact the topological entropy $\htop(a)$ of the dynamical system $(X_{\adele},a)$.
One may verify that using the formula 
\begin{align*}
\htop(a) = \lim_{K \to \infty} \limsup_{n \to \infty} \frac{-\log(m_G(B_{K}^{G_a^+}B_{K}^{G_a^0}B_{K+n-1}^{G_a^-}))}{n}
\end{align*}
which in turn follows from Lemma \ref{lemma: Explicit Bowen balls} and Equation \eqref{eq: equivalence of balls in G(Qp)}.

\subsection{From maximal entropy to additional invariance}\label{subsection: From maximal entropy to additional invariance}

In this subsection, we will see how Theorem~\ref{thm: Maximal entropy} implies Theorem~\ref{thm:main} (in the cocompact case) using the following theorem (a special case of Theorem 7.9 in \cite{pisa}) to characterize measures of maximal entropy.

\begin{theorem}[Additional invariance for measures of maximal entropy]\label{thm: uniqueness of measures of max entropy} 
Let $\G$ be a linear algebraic $\Q$-group, let $\Gamma < \G(\adele)$ be a lattice and let $X = \lquot{\Gamma}{\G(\adele)}$. 
Suppose that $\mu$ is a Borel probability measure invariant under a diagonalizable element $a \in \G(\Qp)$ with $h_\mu(a) \geq -\log|\det(\Ad_a|_{\mathfrak{g}_a^-})|$. 
Then $\mu$ is invariant under $\G(\Qp)_a^+$ and $\G(\Qp)_a^-$.
\end{theorem}

We note that in the case where $X = \lquot{\Gamma}{\G(\adele)}$ as above is cocompact simpler proofs than e.g.~the proof in \cite{pisa} exist (cf.~\cite[Thm.~8.19]{vol2}).

\begin{proof}[Proof of Theorem \ref{thm:main} in the cocompact case]
As mentioned in Section \ref{section:maximal entropy} it suffices to show that any \wstar-limit $\mu$ of the sequence $(\mu_\ell)_\ell$ is invariant under $\sicover\G(\adele)$.
Notice that as $X_{\adele}$ is compact, $\mu$ is automatically a probability measure (non-escape of mass).
By Theorem \ref{thm: Maximal entropy} and Theorem \ref{thm: uniqueness of measures of max entropy}, $\mu$ is invariant under the horospherical subgroups $\G(\Qp)_a^+$ and $\G(\Qp)_a^-$ of $\G(\Qp)$.
Hence, $\mu$ is also invariant under the subgroup $\sicover\BG(\Qp)$ generated by these horospherical subgroups (cf.~Example \ref{exp: horospherical subgrps for PGL_2}).

It remains to show that any $\sicover\BG(\Qp)$-invariant probability measure $\mu$ on $X_{\adele} = \lquot{\G(\Q)}{\G(\adele)}$ is also $\sicover\BG(\adele_{\set{p}})$-invariant where $\adele_{\set{p}} = \Q_{\mathcal{V}^\Q \setminus\set{p}}$.

To prove this, we first decompose the measure $\mu$ into measures on $\sicover\BG(\adele)$-orbits.
Let $\mathcal{A}$ be the $\sigma$-algebra generated by the $\sicover\BG(\adele)$-orbits and note that $\mathcal{A}$ is countably generated\footnote{For an odd prime $p$ let $U_p<\BG(\adele)$ be the subgroup of $g \in \PGL_2(\adele)$ where $g_2\in \sicover\BG(\Q_2)$,\ldots, $g_p\in \sicover\BG(\Q_p)$.
Then the index of $U_p$ in $\BG(\adele)$ is finite as the index of $\sicover\G(\Q_q)$ in $\G(\Q_q)$ is finite for any prime $q$. 
In particular, $U_p$ has only finitely many orbits on $X_{\adele}$. The $\sigma$-algebra $\mathcal{A}$ is then generated by the set of these orbits for varying $p$, which is countable.
}.
Let $X' \subset X_{\adele}$ be a $\mu$-conull set so that the conditional measures $\mu_x^{\mathcal{A}}$ for $\mathcal{A}$ are defined for all $x \in X'$ (see for instance \cite[Ch.~5]{vol1} for definitions).
Recall that for any $x \in X'$ the probability measure $\mu_x^{\mathcal{A}}$ satisfies $\mu_x^{\mathcal{A}}(x \sicover\BG(\adele)) = 1$ so that we can identify $\mu_x^{\mathcal{A}}$ with a measure on
\begin{align*}
x\sicover\BG(\adele) \simeq \lquot{\Gamma_x}{\sicover\BG(\adele)}
\end{align*}
where $\Gamma_x = g^{-1}\G(\Q)g\cap \sicover\BG(\adele)$ for $g\in \G(\adele)$ with $\G(\Q)g = x$.

By construction of the conditional measures we have $\mu = \int \mu_x^{\mathcal{A}} \mathrm{d}\mu$, where we note that by uniqueness of this decomposition the measures $\mu_x^{\mathcal{A}}$ are $\sicover\G(\Qp)$-invariant.
It thus suffices to show that any $\sicover\G(\Qp)$-invariant probability measure on the orbit $x \sicover\G(\adele)$ for $x \in X'$ must be the Haar measure.

For this, notice that we have the following one-to-one correspondences for any point $x\in X'$:
\begin{center}
\begin{tabular}{l}
right $\sicover\BG(\Qp)$-invariant finite measures on $\lquot{\Gamma_x}{\sicover\BG(\adele)}$ \\ 
$ \longleftrightarrow$ right $\sicover\BG(\Qp)$-invariant and left $\Gamma_x$-invariant locally finite measures\\ 
\qquad\ on the group $\sicover\BG(\adele)$ \\
$ \longleftrightarrow$ left $\Gamma_x$-invariant locally finite measures on the quotient \\ 
\qquad\ $ \rquot{\sicover\BG(\adele)}{\sicover\BG(\Qp)} \isom \sicover\BG(\adele_{\set{p}})$.
\end{tabular}
\end{center}
%
We thus let $\nu$ be such a left $\Gamma_x$-invariant measure.
Strong approximation (cf.~\cite[Thm.~7.12]{platonov} or \cite[Thm.~2.3]{rapinchukstrongapprox})
for the group $\BG^{(1)}$ shows that $\sicover\BG(\Q)$ and thus $\Gamma_x$ is dense in $\sicover\BG(\adele_{\set{p}})$.
Given $g \in \sicover\BG(\adele_{\set{p}})$ we pick a sequence $(\gamma_k)$ in $\Gamma_x$ with $\gamma_k \to g$ and obtain that $(L_g)_*\nu \leftarrow (L_{\gamma_k})_*\nu = \nu$ i.e. $\nu$ is left-$\sicover\BG(\adele_{\set{p}})$-invariant.
Thus, there is only one such measure $\nu$ up to scalars, which proves the claim.
%
\end{proof}

\subsection{Linnik's basic lemma}\label{section:linnik basic lemma}

As mentioned after the statement of Theorem \ref{thm: Maximal entropy} one crucial step in showing maximal entropy is to obtain sufficient control on the spacing of orbits in each packet.
Some control of this kind can be obtained by using that distinct integer points have at least distance one from each other; this however would not be sufficient.
Instead, one can prove an averaged result.

We fix (using compactness of $X_{\adele} = \lquot{\G(\Q)}{\G(\adele)}$) a uniform injectivity radius $r < 1$.
Possibly after decreasing $r$ we may assume that balls of radius $r$ are contained in $\G(\R \times \widehat{\Z})$-orbits.

Given $x\in X_\adele$ and $y \in B_r(x)$ we will say that $x$ and $y$ are \emph{$p$-adically $\delta$-close} for $\delta \in (0,r)$ if $y = xg$ for $g \in B_r(e)$ with $d(g_p,e) \leq \delta$. We will also write $x \widesim{\delta}_p y$ in this case.

\begin{proposition}[Linnik's basic lemma]\label{prop: Linnik's basic lemma}
For any $\delta > 0$, any $\varepsilon>0$ and any $\ell$
\begin{align*}
\mu_\ell\times \mu_\ell\Big(\setc{(x,y)\in X_\adele^2}{
x \widesim{\delta}_p y \text{ and } d(x,y) < r}\Big)
\ll_\varepsilon \delta^3 \Nr(v_\ell)^\varepsilon 
\end{align*}
as long as $\Nr(v_\ell)^{-\frac{1}{4}}\leq \delta$.
\end{proposition}

The proof of Linnik's basic lemma uses a theorem on representations of binary quadratic forms by ternary quadratic forms. 

\subsubsection{Representations of integral quadratic forms}
Recall that a representation of an integral quadratic form $q$ on $ \Z^n$ by an integral quadratic form $Q$ on $\Z^m$ is a structure-preserving $\Z$-linear map $\iota:\Z^n \to \Z^m$ i.e. $\iota$ satisfies $Q(\iota(x)) = q(x)$ for all $x \in \Z^n$. Let $R_Q(q)$ be the set of representations of $q$ by $Q$ and observe that $\SO_Q(\Z)<\GL_m(\Z)$ acts on $R_Q(q)$ by post-composition.

\begin{example}\label{exp: estimate divisor function}
Consider the quadratic forms $q(z) = dz^2$ and $Q(x,y) = xy$ for some integer $d$. A representation $ \iota: \Z \to \Z^2$ corresponds to a choice of image $\iota(1) \in \Z^2$, that is, a point $(x,y) \in \Z^2$ with $xy = d$. The number $|R_Q(q)|$ is thus exactly the number of divisors of $d$. The divisor function $\chi(n) := \sum_{d|n}1$ satisfies $\chi(n) \ll_\varepsilon n^\varepsilon$ for any $\varepsilon > 0$.
\end{example}

The proof of Proposition \ref{prop: Linnik's basic lemma} needs the following number-theoretic input:

\begin{theorem}\label{thm: representations qf}
Let $Q$ be a non-degenerate integral ternary quadratic form and let $q(x,y) = ax^2 + bxy + cy^2$ be a non-degenerate integral binary quadratic form. Let $f^2 | \gcd(a,b,c)$ be the greatest common square divisor of $a,b,c$. The number of embeddings $(\mathbb{Z}^2,q)$ into $(\mathbb{Z}^3,Q)$ modulo the action of $\SO_Q(\mathbb{Z})$ is $\ll_{Q,\varepsilon} f \max(|a|,|b|,|c|)^\varepsilon$ where $\varepsilon >0$ is arbitrary.
\end{theorem}

Venkov \cite{venkov} provided a first proof of Theorem \ref{thm: representations qf} when $Q$ is the sum of three squares (which is also of interest to us).
The general case is due to Pall \cite[Thm.~5]{pall49} and is a special case of Siegel's mass formula.
A conceptual proof of Theorem \ref{thm: representations qf} by counting on the tree $\rquot{\SO_3(\Qp)}{\SO_3(\Zp)}$ may be found in \cite[Appendix A]{duke}.

\subsubsection{Proof of Linnik's basic lemma}
\begin{proof}[Proof]
Let $\ell$ be fixed with $\Nr(v_\ell)^{-\frac{1}{4}}\leq \delta$ and write $d = \Nr(v_\ell)$.
By the choice of the injectivity radius $r$ any two points $x_1,x_2$
in the packet $P_\ell$ of distance less than $r$ need to be in the same $\G(\R \times \widehat{\Z})$-orbit (see also \eqref{eq:fibers for genus of G}).

We may therefore fix $k \in \set{1,\ldots,n}$ and study pairs of points in the $\G(\R \times \widehat{\Z})$-orbit of $\G(\Q)g^{(k)}$.
We fix a $\Z$-basis $v_1,v_2,v_3$ of $\pure{\mathcal{O}_k}$ and let $Q$ be the representation of the norm form in this basis.
Furthermore, let $\norm{\cdot}_p$ be the norm on $\pure{\quat(\Qp)}$ given by $\norm{a_1v_1+a_2v_2+a_3v_3}_p = \max\set{|a_1|_p,|a_2|_p,|a_3|_p}$ where $|\cdot|_p$ is the $p$-adic norm on $\Qp$.
There exists an absolute constant $C>0$ such that $|\Nr(w_p)|_p \leq C \norm{w_p}_p^2$ for any $w_p \in \quat(\Qp)$.

We first claim that the finite set $I_{d,\delta}$ of diagonal $\Gk{k}(\Z)$-equivalences of pairs $(w_1,w_2)$ of primitive points $w_1,w_2$ in $\pure{\mathcal{O}_k}$ with $\Nr(w_1) =  \Nr(w_2)=d$ and with $0 <\norm{w_1-w_2}_p\leq\delta$ satisfies
\begin{align}\label{eq:proofLinnikbasic1}
\big| I_{d,\delta} \big| \ll_\varepsilon \delta^2 d^{1+\varepsilon}.
\end{align}
From this we will deduce the proposition by attaching to any pair of $\delta$-close points in the packet their associated integer points (see Section \ref{section:generating intpts}).

Let $\Gk{k}(\Z).(w_1,w_2) \in I_{d,\delta}$ be given and set $q$ to be the integral quadratic form
\begin{align*}
q(x,y) := \Nr(xw_1+yw_2) = dx^2+e xy + dy^2
\end{align*}
for some $e \in \mathbb{Z}$. 
This is simply the restriction of the norm form to the sublattice $\Z w_1 + \Z w_2 \subset \mathcal{O}$ represented in the basis $w_1,w_2$.
Note that the coefficients satisfy 
\begin{align}\label{eq:linnik's basic lemma 2}
|2d-e|_p = |q(1,-1)|_p \leq C\norm{w_1-w_2}_p^2 \leq C \delta^2.
\end{align}
We also have
\begin{align}\label{eq:linnik's basic lemma 1}
|2d-e| = |q(1,-1)| = \Nr(w_1-w_2)\leq 2(\Nr(w_1)+\Nr(w_2)) \leq 4 d.
\end{align}
For convenience, set $m = \lgauss{-2\log_p(\delta)-\log_p(C)}$ so that by \eqref{eq:linnik's basic lemma 2} we have $p^m | (2d-e)$. 
The quadratic form $q$ is non-degenerate: By Equation \eqref{eq:linnik's basic lemma 1} $e \neq 2d$, since otherwise $w_1 = w_2$ and by Equation \eqref{eq:linnik's basic lemma 2} $e \neq -2d$, since otherwise $1 = |4d|_p \leq \delta^2$ which contradicts $\delta < 1$. Denote by $N_{e,d}$ the number of inequivalent ways of representing $ dx^2+e xy + dy^2 $ by $Q$ which satisfies by Theorem \ref{thm: representations qf} 
\begin{align*}
N_{e,d} \ll_\varepsilon f \max(|d|,|e|)^\varepsilon \leq f\max(|d|,|2d-e|+|2d|)^\varepsilon \ll_\varepsilon f d^\varepsilon.
\end{align*}
Here, $f^2=:\gcsd{e,d}$ is the greatest common square divisor of $e$ and $d$. For $L \leq 4 d$ compute
\begin{align*}
|I_{d,\delta}| &\leq \sum_{\substack{e:|2d-e|\leq L,\ e \neq \pm 2d,\\ p^m | (2d-e)}} N_{e,d} = \sum_{\substack{e':|e'|\leq L, p^m|e', \\ e' \neq 0,4d}} N_{2d-e',d}\\
&\leq \sum_{f^2|d} \ \sum_{\substack{e':|e'|\leq L, p^m|e',\\ f^2 = \gcsd{e',d},e' \neq 0,4d}} N_{2d-e',d} 
\ll_\varepsilon \sum_{f^2|d}\ \sum_{\substack{e':|e'|\leq L, p^m|e',\\ f^2 = \gcsd{e',d}}} f d^\varepsilon \\
&= \sum_{f^2|d} f d^\varepsilon \sum_{\substack{e':|e'|\leq L, p^m|e',\\ f^2 = \gcsd{e',d}}} 1.
\end{align*}
Now observe that the number of $e'$ satisfying $p^m |e'$, $f^2|e'$ and $|e'| \leq L$ is $\ll \frac{L}{p^m f^2}$, since $f^2$ and $p^m$ are coprime ($p$ does not divide $d$). Thus, by Example \ref{exp: estimate divisor function}
\begin{align*}
|I_{d,\delta}| \ll_\varepsilon \sum_{f^2|d} f d^\varepsilon \frac{L}{p^m f^2} 
\ll \sum_{f^2|d} d^\varepsilon \frac{d \delta^2}{f} 
\leq d^{1+\varepsilon}\delta^2 \sum_{f^2|d}1 
\ll_\varepsilon d^{1+2\varepsilon}\delta^2
\end{align*}
which finishes the proof of the claim in \eqref{eq:proofLinnikbasic1}.

Now let $x_1 = \G(\Q)h_1g_{\ell,\infty}\alpha_\ell$, $x_2 = \G(\Q)h_2g_{\ell,\infty}\alpha_\ell$ be two points in the packet $P_\ell$ which lie in the $\G(\R \times \widehat{\Z})$-orbit through $\G(\Q)g^{(k)}$ and which are $p$-adically $\delta$-close.
We write $x_1 = \G(\Q)g^{(k)}g_1,\, x_2 = \G(\Q)g^{(k)}g_2$ for $g_1,g_2 \in \G(\R \times \widehat{\Z})$ with $d((g_1)_p,(g_2)_p)\leq \delta$.
Recall that $\alpha_\ell.v$ is a $\Zp^\times$-multiple of $v_\ell$ (Section \ref{sec:conjugacy}), say $v_\ell = \beta_\ell (\alpha_\ell.v)$ for $\beta_\ell \in \Zp^\times$.
By Section \ref{section:generating intpts} the points
\begin{align*}
w_1 = \beta_\ell \cdot (g^{(k)}g_1)_p.v = (g^{(k)}g_1)_p\alpha_\ell^{-1}.v_\ell,
\quad w_2 = \beta_\ell \cdot (g^{(k)}g_2)_p.v
\end{align*}
are primitive and pure elements of the maximal order $\mathcal{O}_k$, are of norm $d = Q(v_\ell)$ and satisfy $\norm{w_1-w_2}_p \leq \delta$. 

\textbf{Case 1} -- equal integer points. 
If $w_1=w_2$ then $x_1$ and $x_2$ lie on the same $K_\ell$-orbit by Lemma \ref{lemma:same production} where $K_\ell = g_{\ell,\infty}^{-1}\alpha_\ell^{-1}\torus_{v_\ell}(\R \times \widehat{\Z})\alpha_\ell g_{\ell,\infty}$.
The volume of a $\delta$-ball in $K_\ell$ is $\ll \delta$ and in particular, the set of $\delta$-close pairs $x_1,x_2 \in P_\ell$ that lie on the same orbit has volume $\ll_\varepsilon d^{\frac{1}{2}+\varepsilon}\delta$ by Fubini's theorem and Proposition \ref{prop:total volume}. 
After normalization, the contribution to the total mass is $\ll_\varepsilon d^{- \frac{1}{2}+\varepsilon}\delta \leq d^\varepsilon \delta^3$ in this case.

\textbf{Case 2} -- distinct integer points. For fixed $(w_1,w_2) \in I_{d,\delta}$ the set of pairs $(x_1,x_2)$ as above with associated integer pair $(w_1,w_2)$ has volume $\ll \delta$ by Fubini's theorem and thus the volume in total is $\ll |I_{d,\delta}| \delta \ll_\varepsilon d^{1+2\varepsilon}\delta^3 $. After normalization, the measure contribution in this case is therefore $\ll_\varepsilon d^{3\varepsilon} \delta^3$.
\end{proof}

\subsection{Proof of Theorem \ref{thm: Maximal entropy}}

Before turning to the proof of Theorem \ref{thm: Maximal entropy} we construct a partition $\mathcal{P}$ of the compact space $X_{\adele}= \lquot{\G(\Q)}{\G(\adele)}$ into measurable subsets so that the refinement of $\mathcal{P}$ for the dynamics of $a$ is very thin in the horospherical directions inside $\G(\Qp)$.

Given a fixed open set $U \subset \BG(\adele)$ (that we will choose presently) we call the set
\begin{align*}
B_N = \bigcap_{k=-N}^N a^k U a^{-k}
\end{align*}
the \emph{Bowen $N$-ball} in $\BG(\adele)$.
A Bowen $N$-ball in $X_{\adele}$ is then a set of the form $xB_N$ for $x \in X_{\adele}$.

We now choose the open set $U$ in a manner compatible with the action of $a$.
Notice that if $U\subset \G(\adele)$ is of the form $\prod_\sigma U_\sigma$ where $U_q = \G(\Z_q)$ for all but finitely many primes $q$, then
$a^k U a^{-k} = U_\infty\times \ldots \times a^k U_p a^{-k} \times \ldots$ and therefore $B_N = \prod_{\sigma\neq p}U_\sigma \times \bigcap_{k= -N}^N a^k U_p a^{-k}$.
We choose $U_p$ as an open rectangle in $\G(\Qp)$ (cf.~Section \ref{section:open rectangles}) i.e.~of the form 
\begin{align*}
U_p = B_{cp^{-s}}^{\G(\Qp)_a^+}B_{cp^{-s}}^{\G(\Qp)_a^0}B_{cp^{-s}}^{\G(\Qp)_a^-} = R_s
\end{align*}
for some $s \in \mathbb{N}$ and some $c = p^{-s'}$, where $c$ is chosen small enough so that $U_p \subset B_{p^{-s}}^{\G(\Qp)}$ holds for all $s \in \mathbb{N}$
(see \eqref{eq: equivalence of balls in G(Qp)}). 
Lemma \ref{lemma: Explicit Bowen balls} then describes the intersection $\bigcap_{k= -N}^N a^k U_p a^{-k}$,
which is thin in both horospherical directions.
Also, we can write it as a disjoint union of $p^N$ shifts of the open rectangle $R_{s+N}$.
More precisely, there exist $a_1,\ldots,a_{p^N} \in \torus(\Zp)$ with
\begin{align}\label{eq:chopping of Bowen}
\bigcap_{k= -N}^N a^k U_p a^{-k} = \bigsqcup_{k=1}^{p^N}R_{s+N} a_k.
\end{align}
We choose the set $U$ to be a set of the form above such that $U$ is contained in the injective set $B_{r}(e) \cap aB_{r}(e)a^{-1} \cap a^{-1}B_{r}(e)a$ around the identity in $\G(\adele)$.

\begin{lemma}[A suitable partition]\label{lemma:good partition}
There exists a finite partition $\mathcal{P}$ of $X_{\adele}$ into measurable subsets with the property that for $N \in \mathbb{N}$ any atom $[x]_{\mathcal{P}_{-N}^N}$ of
\begin{align*}
\mathcal{P}_{-N}^N = \bigvee_{n=-N}^N a^n.\mathcal{P}
\end{align*}
is contained in a Bowen $N$-ball.

Furthermore, given a Borel probability measure $\mu$ on $X_{\adele}$ the partition $\mathcal{P}$ may be chosen so that $\mu(\partial P) = 0$ for all $P \in \mathcal{P}$.
\end{lemma}

\begin{proof}
Let $\mu$ be as in the lemma.
We begin by constructing a partition $\mathcal{P}$ of $X_{\adele}$ into sets of small diameter so that $\mu(\partial P) = 0$ holds for all $P \in \mathcal{P}$.
Given $x \in X_{\adele}$ the function $s \mapsto \mu(B_s(x))$ is monotonely increasing and thus continuous at all but countably many radii $s$.
If it continuous at $s$, we have $\mu(\partial B_s(x)) = 0$.
In particular, we may choose for any point $x\in X_{\adele}$ a radius $s_x$  for which $\mu(\partial B_{s_x}(x)) = 0$ as well as $B_{2s_x}^{\G(\adele)}(e) \subset U$.
From a covering of $X_{\adele}$ by finitely many such balls $B_{s_x}(x)$ one readily constructs a partition $\mathcal{P}$ whose elements have diameter less than $r$ as desired.

We now want to show that for any $N \in \mathbb{N}$ and any $x \in X_{\adele}$ we have
\begin{align*}
[x]_{\mathcal{P}_{-N}^N} \subset xB_N.
\end{align*}
Let $ y = g_0.x \in \atom{x}{\mathcal{P}_{-N}^N}$ for $g_0 \in U$.
Since $a.y \in [a.x]_\mathcal{P}$ we can write $g_1.(a.y) = a.x$ for some $g_1 \in U$.
On the other hand, notice that $(ag_0a^{-1}).(a.y) = a.x$.
But $ag_0a^{-1}$ and $g_1$ both lie inside the ball of injectivity radius $r$ by the choice of the open set $U$ so we must have $g_1 = ag_0a^{-1}$.

Proceeding this way, we find elements $g_1,g_2,\ldots ,g_N $ in $U$ where $g_n = a^{n}g_0a^{-n}$ and $g_n.(a^n.x) = a^n.y$ for every $n \in \set{1,\ldots,N}$. 
Applying the same method to $a^{-1}$ instead of $a$ we obtain $g_0 \in B_N$
as desired.
\end{proof}

\begin{proof}[Proof of Theorem \ref{thm: Maximal entropy}]
Without loss of generality we may assume that the measures $\mu_\ell$ converge to a probability measure $\mu$ on $X_{\adele}$ in the \wstar-topology as $\ell$ goes to infinity.
Since the function $t \mapsto - \log(t)$ is convex, the inequality
\begin{align*}
H_{\mu_\ell}(\mathcal{P'}) \geq -\log\left(\sum_{P \in \mathcal{P}'}\mu_\ell(P)^2\right)
\end{align*}
holds for any partition $\mathcal{P}'$ of $X_\adele$ and any $\ell$.
Let $\mathcal{P}$ be the finite partition constructed in Lemma \ref{lemma:good partition} and let $N \in \mathbb{N}$.

By \eqref{eq:chopping of Bowen} there exist $a_1,\ldots,a_{p^N} \in \torus(\Zp)$ with
\begin{align*}
\Disjointunion_{S \in \mathcal{P}_{-N}^N} S\times S 
\subset 
\Union_{i=1}^{p^N}\left\lbrace 
(x,ya_i) \in X_{\adele}^2: 
x\widesim{cp^{-(s+N)}}_p y 
\text{ and } d(x,y) < r\right\rbrace.
\end{align*}

Using this, $\torus(\Zp)$-invariance of $\mu_\ell$ and Linnik's basic lemma for the choice $\delta = cp^{-(s+N)}$ we obtain for any $\ell$ and any $\varepsilon>0$
\begin{align*}
\sum_{S \in \mathcal{P}_{-N}^N} \mu_\ell(S)^2 
\ll_\varepsilon p^N d_\ell^\varepsilon \delta^3 
\ll p^{-2N} d_\ell^\varepsilon
\end{align*}
if $d_\ell = \Nr(v_\ell)$ satisfies $d_\ell^{-\frac{1}{4}} \leq \delta$ or equivalently $N \leq \frac{1}{4}\log_p(d\ell) -K$. 
Thus, set $N_\ell = \lgauss{\frac{1}{5}\log_p(d_\ell)}$. Let $C(\varepsilon)$ be the implicit constant appearing in the estimate above. Then
\begin{align*}
H_{\mu_\ell}(\mathcal{P}^{N_\ell}_{-N_\ell}) 
\geq -\log\bigg(\sum_{S \in \mathcal{P}_{-N_\ell}^{N_\ell}} \mu_\ell(S)^2\bigg) 
\geq -\log(C(\varepsilon)) -\varepsilon \log(d_\ell) +2N_\ell\log(p)
\end{align*}
if $\ell$ is large enough.

Note that $ \log(d_\ell) \leq 5N_\ell\log(p) +5\log(p)$ and $\log(C(\varepsilon)) + 5\varepsilon\log(p) \leq \varepsilon N_\ell \log(p)$ if $\ell$ is large enough. 
Hence,
\begin{align*}
H_{\mu_\ell}(\mathcal{P}^{N_\ell}_{-N_\ell}) \geq (2-6\varepsilon) N_\ell \log(p).
\end{align*}
We eliminate the dependency on $\ell$ in the refinement of the partition $\mathcal{P}$: 
For a given $n \in \mathbb{N}$ choose $k$ such that $2N_\ell+n \geq n k\geq 2N_\ell+1$. Now observe that the partition
\begin{align*}
\bigvee_{j=0}^{k-1} a^{jn}\mathcal{P}_{0}^{n-1} = \mathcal{P}_{0}^{n k-1}
\end{align*}
is finer than the partition $\mathcal{P}_{0}^{2N_\ell}$. Hence
\begin{align*}
H_{\mu_\ell}(\mathcal{P}_{-N_\ell}^{N_\ell}) 
= H_{\mu_\ell}(\mathcal{P}_{0}^{2N_\ell}) 
\leq k H_{\mu_\ell}(\mathcal{P}_{0}^{n-1}) 
\leq \frac{2N_\ell+n}{n} H_{\mu_\ell}(\mathcal{P}_{0}^{n-1})
\end{align*}
and therefore
\begin{align*}
H_{\mu_\ell}(\mathcal{P}_{0}^{n-1}) 
\geq \frac{n}{2N_\ell+n} (2-6\varepsilon) N_\ell \log(p).
\end{align*}
By the choice of the partition $\mathcal{P}$ we have $\mu(\partial A) = 0$ and thus $\mu_\ell( A) \to \mu(A)$ for any $A \in \mathcal{P}_{0}^{n-1}$.
Letting $\ell\to \infty$ therefore shows that
\begin{align*}
H_{\mu}(\mathcal{P}_{0}^{n-1}) \geq \frac{n}{2} (2-6\varepsilon) \log(p)
\end{align*}
and since $\varepsilon$ was arbitrary
\begin{align*}
\frac{1}{n} H_{\mu}(\mathcal{P}_{0}^{n-1}) \geq \log(p).
\end{align*}
This yields the theorem when taking the limit as $n \to \infty$. 
\end{proof}
\section{Application: Integer points on spheres}\label{section:intpts on spheres}

The goal of this section is to prove Linnik's Theorem A using Theorem \ref{thm:main}.
Let us quickly recall the notation needed for this special case.
We consider the algebra $\quat = \quat_{\infty,2}$ of Hamiltonian quaternions and the maximal order
\begin{align*}
\mathcal{O}_\HW = \Z\left[\ii,\jj,\kk,\frac{1+\ii+\jj+\kk}{2}\right] \subset \quat(\Q)
\end{align*}
of Hurwitz quaternions.
In particular, the set of traceless elements $\mathcal{O}_\HW^{(0)}$ in this order is equal to $\Z \ii + \Z \jj + \Z \kk$.
Using this basis $\ii,\jj,\kk$ and letting $\BG = \mathbf{PB}^\times$ acting on the pure quaternions $\pure{\quat}$ by conjugation we obtain a $\Q$-isomorphism $\BG \to \SO_3$.
If desired, the reader may thus replace $\BG$ by $\SO_3$ in the discussions to follow (see also Section \ref{section:acting groups}).

Note that $\G$ has class number one i.e.~$\BG(\Q)\BG(\R \times\widehat{\Z}) = \BG(\adele)$ as there is only one maximal order in $\quat$ up to $\G(\Q)$-conjugacy (see also \cite[Sec.~5]{Linnikthmexpander}).
Given any set of places $S \subset \mathcal{V}^\Q$ of $\Q$ containing the archimedean place we therefore have a well-defined (surjective) projection
\begin{align*}
\pi_S:\lquot{\G(\Q)}{\G(\adele)} \to \lquot{\G(\Z^S)}{\G(\Q_S)} = X_S
\end{align*}
onto the $S$-arithmetic extension $X_S$.

\subsection{Equidistribution of packets}\label{section:intpts-equi packets}

Now let $(v_\ell)$ be any sequence of primitive vectors in $\pure{\mathcal{O}_\HW}$ with $\Nr(v_\ell) \in \BD(p)$ and $\Nr(v_\ell) \to \infty$ as $\ell \to \infty$ (i.e.~$(v_\ell)$ is ``admissible'' in the sense of Theorem \ref{thm:main}).
Set $K_\infty = \torus_{v_1}(\R)$ and for any $\ell$ let $g_{\ell,\infty}\in \G(\R)$ with $\torus_{v_\ell}(\R) = g_{\ell,\infty} K_\infty g_{\ell,\infty}^{-1}$.

\subsubsection{Applying Theorem \ref{thm:main} for the real quotient}
Assume that the invariant probability measures $\mu_\ell$ on the orbits $\G(\Q)\torus_{v_d}(\adele)g_{\ell,\infty}$ converge in the \wstar-topology to a probability measure $\mu$.
By Theorem \ref{thm:main} we known that $\mu$ is $\sicover\BG(\adele)$-invariant.
In particular, we have $(\pi_S)_\ast \mu_\ell\to (\pi_S)_\ast \mu$ where $(\pi_S)_\ast \mu$ is $\sicover\BG(\Q_S)$-invariant.

Applying this discussion to $S = \set{\infty}$ we obtain that
\begin{align*}
(\pi_{\set{\infty}})_\ast \mu_\ell \to (\pi_{\set{\infty}})_\ast\mu = m_{X_{\set{\infty}}}
\end{align*}
where $m_{X_{\set{\infty}}}$ denotes the normalized Haar measure on $X_{\set{\infty}}$. 
Indeed, we have $\sicover\BG(\R) = \BG(\R)$ at the archimedean place as $\G(\R) \simeq \SO_3(\R)$ is connected.

Notice that $(\pi_{\set{\infty}})_\ast \mu_\ell$ is the $K_\infty$-invariant probability measure on the packet
\begin{align*}
P(v_\ell,\set{\infty}) 
:= \bigsqcup_{\rho\in \CR_{v_\ell}} \G(\Z)\rho_\infty g_{\ell,\infty} K_\infty
\end{align*}
which assigns to every orbit the same mass (see Section \ref{section:def volume}).
The union is in fact still disjoint as for any $\rho\in \mathcal{R}_{v_\ell}$ and any $k \in K_\infty$ we have that $\rho_\infty g_{\ell,\infty}k.v_1$ is a multiple of $\rho_\infty.v_\ell \in \mathcal{O}_\HW$.
The latter points however are never equivalent mod the action of $\G(\Z)$ for distinct point in $\mathcal{R}_{v_\ell}$.

\subsubsection{Projecting to the sphere}

We now consider the push-forward $\nu_\ell$ of the measures $ (\pi_{\set{\infty}})_\ast \mu_\ell$ under the map
\begin{align*}
\operatorname{pr}:X_{\set{\infty}} \to 
\lrquot{\G(\Z)}{\G(\R)}{K_\infty} \simeq \lquot{\G(\Z)}{\mathbb{S}^2} =: Y
\end{align*}
where we identified the sphere $\mathbb{S}^2$ with the quotient $\rquot{\G(\R)}{K_\infty}$.
By Section \ref{section:generating intpts}, the measure $\nu_\ell$ is the normalized sum of Dirac measures on (all) points of the form $\frac{v}{\sqrt{d}}$ where $v\in \pure{\mathcal{O}_\HW}$ is primitive and produced by the packet for $v_d$.

In summary, we have shown the subsets of integer points produced by the stabilizer orbit of $v_\ell$ (when projected to $Y$) are equidistributed inside $Y$ when $\ell$ goes to infinity.

\subsection{Proof of Linnik's Theorem A}

We now turn to the proof of Linnik's Theorem A, for which we proceed in two steps.

\subsubsection{Equidistribution on folded sphere}\label{section:linnik folded}

We begin by showing how the discussion in Section \ref{section:intpts-equi packets} can be used to show equidistribution of all primitive integer points on $Y$.
This is essentially the statement in Linnik's Theorem A, but on the folded sphere $Y$ instead of $\mathbb{S}^2$.
We defer the simple upgrade for $\mathbb{S}^2$ to Section \ref{section:lifting}.

Denote by $\mathcal{I}_d'$ the image of $\mathcal{I}_d$ in $Y$.
Given two points $w,w'\in \mathcal{I}_d'$, we say that $w$ is equivalent to $w'$ if for all $p$ there exists $g_p \in \mathbb{G}(\Zp)$ with $g_p w = w'$. 
In Section~\ref{section:generating intpts}, we have seen that the primitive integer points equivalent to a fixed primitive integer point $w$ are exactly the integer points produced by the adelic stabilizer orbit of~$w$.
Note that Proposition \ref{prop:transitivity} now easily implies equidistribution of the sets $\mathcal{I}_d'$.
Nevertheless, we present here a more elementary argument by averaging.

\begin{proof}[Proof of Linnik's Theorem A on $Y$]
For any finite set $F \subset Y$ set $\nu_F = \frac{1}{|F|} \sum_{x  \in F} \delta_x$ for simplicity. 
Suppose by contradiction that $\nu_{\mathcal{I}_{d_\ell}'} \to \nu \neq m_{\mathbb{S}^2}$ as $\ell \to \infty$ along a sequence of $d_\ell\in \BD(p)$ and choose $f \in C(Y)$ so that $\int f d\nu \neq \int f dm_{Y}$. 
For any $d \in \mathbb{D}(p)$ write $\mathcal{I}_d'$ as a finite union of equivalence classes $\mathcal{I}_{d,1}',\ldots,\mathcal{I}_{d,k_d}'$ for the equivalence relation defined above. In particular, we may view $\nu_{\mathcal{I}_d'}$ as a convex combination
\begin{align*}
\nu_{\mathcal{I}_d'} = \sum_{j=1}^{k_d} \frac{|\mathcal{I}_{d,j}'|}{|\mathcal{I}_d'|} \nu_{\mathcal{I}_{d,j}'}.
\end{align*}
Choose $\varepsilon > 0$ so that for all large enough $\ell$, we have $|\int f d\nu_{\mathcal{I}_{d_\ell}'} - \int f dm_Y| \geq \varepsilon$. 
In particular, there must exist some $1 \leq j_\ell \leq k_{d_\ell}$ with $|\int f d\nu_{\mathcal{I}_{d_\ell,j_\ell}'} - \int f dm_Y| \geq \varepsilon$ for every $\ell$. 
This contradicts the claim in Section \ref{section:intpts-equi packets} which implies that $\nu_{\mathcal{I}_{d_\ell,j_\ell}'} \to m_{Y}$ as $\ell\to \infty$.
\end{proof}

\subsubsection{Lifting to the sphere}\label{section:lifting}

To upgrade the above proof to the statement in Linnik's Theorem A it suffices to use the following lemma.

\begin{lemma}\label{lemma: lift of measures on quotients by finite groups}
The map $\operatorname{pr}_*$ restricted to the set of $\mathbb{G}(\Z)$-invariant probability measures on $\mathbb{S}^2$ is a homeomorphism
\begin{align*}
\set{\text{$\mathbb{G}(\Z)$-invariant prob. measures on } \mathbb{S}^2} \to \set{\text{prob. measures on } Y}.
\end{align*}
\end{lemma}

Given a probability measure $\mu$ on $Y$, we shall refer to the unique $\mathbb{G}(\Z)$-invariant probability measure $\bar{\mu}$ on  $\mathbb{S}^2$ as the \textit{lift of $\mu$}.

From this and Section \ref{section:linnik folded} the statement in Linnik's Theorem A follows readily. In fact, the lift of the Haar measure $m_Y$ is the normalized Haar measure on $\mathbb{S}^2$ and the lift of the normalized sum of the Dirac measures for points in $\mathcal{I}_d' \subset Y$ is the normalized sum\footnote{Since $\G(\Z)\simeq \SO_3(\Z)$ is finite and any non-trivial element of $\G(\Z)$ can fix only one rational line in $\pure\quat(\Q)\simeq \Q^3$, the stabilizer subgroup of a primitive vector $v\in \pure{\mathcal{O}_HW}$ is trivial for all but finitely many vectors $v$.} of the Dirac measures for points in $\mathcal{I}_d \subset \mathbb{S}^2$ for large enough $d$.

\begin{proof}
We will identify measures $\mu$ with the associated positive linear functionals $\mu(\varphi) = \int \varphi d\mu$. 
Notice that continuous functions $\varphi$ on $Y$ correspond linearly to left-$\G(\Z)$-invariant continuous functions $\tilde{\varphi}$ on $\mathbb{S}^2$. 
For an arbitrary continuous function $\varphi$ on $G$ we introduce the mean
\begin{align*}
\varphi_{\G(\Z)}(x) := \frac{1}{|\G(\Z)|}\sum_{\gamma \in \G(\Z)}\varphi(\gamma^{-1} x)
\end{align*}
which is a left-$\G(\Z)$-invariant function. 
Observe that if $\varphi$ is left-$\G(\Z)$-invariant, then $\varphi = \varphi_{\G(\Z)}$. 
Given a probability measure $\mu$ on $Y$ define a measure $f(\mu)$ on $\mathbb{S}^2$ through
\begin{align*}
f(\mu)(\varphi) = \mu(\widetilde{\varphi_{\G(\Z)}}).
\end{align*}
One now verifies directly that $f$ is a two-sided inverse of $\pi_*$.
\end{proof}

\section{Equidistribution of CM points and the $\PGL_2$-case}\label{section:Equidistribution of CM points}

In this section we would like to prove Theorem \ref{thm:main} for the split quaternion algebra $\quat = \Mat_2$.
We can consider the maximal order $\Mat_2(\Z)\subset \Mat_2(\Q)$ only as any other maximal order is $\GL_2(\Q)$-conjugate to it.

In this case, the group $\BG = \mathbf{PB}^\times = \PGL_2$ has class number one, that is
\begin{align}\label{eq:class number one for PGL2}
\PGL_2(\adele) = \PGL_2(\Q) \PGL_2(\R \times \widehat{\Z}).
\end{align}
In particular, there exist well-defined projections
\begin{align*}
X_\adele \to X_S = \lquot{\PGL_2(\Z^S)}{\PGL_2(\Q_S)}
\end{align*}
for any set of places $S \subset \mathcal{V}^\Q$ containing the archimedean place.

The strategy for Theorem \ref{thm:main} will consist in studying toral packets on the quotient $X_{\set{\infty,p}}$ for a fixed odd prime $p$ and in showing maximal entropy for this setup with respect to a diagonalizable element in $\PGL_2(\Qp)$ (as in Theorem \ref{thm: Maximal entropy}).
This entropy can then be transported to the adelic quotient using the formula of Abramov and Rokhlin \cite{abramovrokhlin}, from which Theorem \ref{thm:main} can be deduced as in the cocompact case -- see Section \ref{section:proof main thm Mat2}.

Some of the discussions in this section can be shortened using statements from Section \ref{section:cocompact}, but we made an effort to keep it as self-contained as possible.

\subsection{Algebraic tori associated to quadratic number fields} \label{section:CM-algebraictori}

For the current case of Theorem \ref{thm:main} we change the viewpoint slightly.

\subsubsection{Proper ideals and embeddings}
Consider the imaginary quadratic number field $K := \Q(\sqrt{d})$ for a negative discriminant $d$ and let
\begin{align*}
R_d = \Z\left[\frac{d+\sqrt{d}}{2}\right] \subset \Q(\sqrt{d})
\end{align*}
be the order of discriminant $d$.
Fix a \emph{proper} $R_d$-ideal $\Fa$ i.e.~a rank two 
$\Z$-lattice $\Fa \subset K$ with
\begin{align*}
R_d = \setc{\lambda \in K}{\lambda.\Fa \subset \Fa}.
\end{align*}
The reader may keep in mind the special case $d = \disc(K)$ so that $R_d = R_K$ is the ring of integers in $K$, which simplifies some of the arguments in what follows (e.g.~any ideal is proper for the ring of integers).

Let $a_1,a_2$ be a $\Z$-basis of $\Fa$.
Given an element $\lambda \in K$, we represent the multiplication by $\lambda$ on $K$ in the basis $(a_1,a_2)$ to obtain a matrix $\psi_\Fa(\lambda) \in \Mat_2(\Q)$ and choose $\psi_\Fa(\lambda) \in \Mat_2(\Q)$ to act on row vectors in $\Q^2$ from the right. 
This yields an embedding of $\Q$-algebras
\begin{align*}
\psi_\Fa: K \embed \Mat_2(\Q).
\end{align*}
Denoting by $\iota_\Fa: \Q^2 \to K$ the isomorphism induced by the choice of basis of $\Fa$ we obtain the commutative diagram
\vspace*{-7pt}
\begin{center}
\begin{tikzpicture}[every node/.style={midway}]
  \matrix[column sep={6em,between origins}, row sep={3em}] at (0,0) {
    \node(NW) {$\Q^2$}  ; & \node(NE) {$\Q^2$}; \\
    \node(SW) {$K$}; & \node (SE) {$K$};\\
  };
  \draw[->] (NW) -- (SW) node[anchor=east]  {$\iota_\Fa$};
  \draw[->] (NW) -- (NE) node[anchor=south] {$\psi_\Fa(\lambda)$};
  \draw[->] (NE) -- (SE) node[anchor=west] {$\iota_\Fa$};
  \draw[->] (SW) -- (SE) node[anchor=north] {$\cdot\lambda$};
  \node at (1,-1) {$  $};
\end{tikzpicture}
\end{center}

Observe that $\psi_\Fa(\lambda)$ has integer entries if and only if $\lambda$ preserves $\mathfrak{a} = \iota_\Fa(\Z^2)$. That is, 
\begin{align}\label{eq:CM-Mat_2 <-> O_K}
\psi_\Fa(\lambda) \in \Mat_2(\Z) \iff \lambda \in R_d
\end{align}
by properness of $\Fa$ and in particular 
\begin{align}\label{eq:CM-integral valued images of K}
\psi_\Fa(\lambda) \in \GL_2(\Z) \iff \lambda \in R_d^\times.
\end{align}

\subsubsection{A choice of traceless element}\label{section:CM-on the traceless guy}
Set by \eqref{eq:CM-Mat_2 <-> O_K}
\begin{align*}
v_\Fa:= \psi_\Fa(\sqrt{d}) \in \Mat_2(\Z)
\end{align*}
and notice that $v_\Fa$ is traceless and of determinant $-d$ as $v_{\Fa}^2 = \psi_{\Fa}(d) = d$.

Notice that $v_\Fa$ might not be primitive.
However, $\alpha v_\Fa = \psi_\Fa(\alpha\sqrt{d}) \in \Mat_2(\Z)$ for $\alpha\in \Q^\times$ if and only if $\alpha\sqrt{d} \in R_d$.
Thus, if $d \equiv 1 \mod 4$, $\alpha$ has to be an integer (i.e.~$v_\Fa$ is primitive) and if $d \equiv 0 \mod 4$, only $2\alpha$ has to be an integer (i.e.~$\frac{1}{2}v_\Fa\in \Mat_2(\Z)$ is primitive).

Furthermore, we note for later use that the off-diagonal entries of $v_\Fa$ are divisible by two as $\psi_\Fa(\frac{d+\sqrt{d}}{2}) = \frac{d}{2}+\frac{1}{2}v_\Fa$ is an integral matrix.

\subsubsection{The associated torus}
As in Section \ref{section:acting tori - general} we define the $\Q$-algebraic torus $\torus_\Fa = \torus_{v_\Fa}$ which satisfies
\begin{align*}
\torus_\Fa(R) &:= \setc{h \in \PGL_2(R)}{h v_\Fa h^{-1} = v_\Fa} \\
&= \setc{h \in \PGL_2(R)}{\forall \lambda \in K: h \psi_\Fa(\lambda) h^{-1} = \psi_\Fa(\lambda)}
\end{align*}
for any algebra $R$ over $\Q$.

The eigenvalues of $v_\Fa$ (or more precisely its conjugacy class) yield a lot of information about the group~$\torus_\Fa$ as we shall presently see. 
Denote by $\overline{\psi_\Fa}$ the composition of $\psi_\Fa:K^\times \to \GL_2(\Q)$ and the projection $\GL_2(\Q) \to \PGL_2(\Q)$.

\begin{claim}\label{claim:Q-points torus}
We have $\psi_\Fa(K) = \setc{ h\in \Mat_2(\Q)}{hv_\Fa = v_\Fa h}$ and $\torus_\Fa(\Q) =~\overline{\psi_\Fa}(K^\times)$. 
Furthermore, $\torus_\Fa(\Z) := \torus_\Fa(\Q) \cap \PGL_2(\Z) = \overline{\psi_\Fa}(R_d^\times)$.
\end{claim}

\begin{proof}
The dimension of $\setc{h \in \Mat_2(\Q)}{hv_\Fa = v_\Fa h}$ over $\Q$ is the same as the dimension of $\setc{h \in \Mat_2(\overline{\Q})}{hv_\Fa = v_\Fa h}$ over the algebraic closure $\overline{\Q}$, the latter being $2$ as $v_\Fa$ is diagonalizable over $\overline{\Q}$. 
The last statement follows from the observation we made in \eqref{eq:CM-integral valued images of K}. 
\end{proof}


\begin{lemma}[Splitting at $p$]\label{lemma:CM-R-points and Qp points of the torus} $ $
\begin{enumerate}[(i)]
\item The group of $\R$-points $\torus_\Fa(\R)$ is conjugate to the compact group $\PO_2(\R)$.
\item Let $p$ be a prime with\footnote{Equivalently, $p$ is split in $K$.} $d \mod p \in (\Fp^\times)^2$.
Then $\torus_\Fa(\Qp)$ is conjugate to the diagonal subgroup
\begin{align*}
\setc{\begin{pmatrix}
a & 0 \\ 
0 & 1
\end{pmatrix} }{a \in \Qp^\times} < \PGL_2(\Qp).
\end{align*}
\end{enumerate}
\end{lemma}

Notice that the statement in (ii) in Lemma \ref{lemma:Linnik's condition}. As the proof here is relatively concrete, we give it nevertheless.
Similarly to Lemma \ref{lemma:Linnik's condition}, we will say that a discriminant $d$ satisfies Linnik's condition at $p$ if $d \mod p \in (\Fp^\times)^2$ holds.

\begin{proof}
$(i)$: The matrix $v_\Fa$ is conjugate over $\R$ to the matrix
\begin{align}\label{eq:CM-def v_d,infty}
\begin{pmatrix}
0 & \sqrt{|d|} \\ 
-\sqrt{|d|} & 0
\end{pmatrix} =: v_{d,\infty}
\end{align}
and the subgroup of matrices in $\PGL_2(\R)$ centralzinh $v_{d,\infty}$ is indeed given by
\begin{align*}
\setc{\begin{pmatrix}
a & b \\ 
-b & a
\end{pmatrix} \in \PGL_2(\R)}{a,b \in \R} = \PO_2(\R).
\end{align*}
$(ii)$: By Hensel's lemma, the polynomial $x^2-d$ splits over $\Qp$ and has distinct roots.
Thus, $v_\Fa$ is diagonalizable. Let $\varepsilon,-\varepsilon$ be the eigenvalues of $v_\Fa$. The matrix $v_\Fa$ is therefore conjugate to
\begin{align}\label{eq:CM-def v_d,p}
\begin{pmatrix}
\varepsilon & 0 \\ 
0 & -\varepsilon
\end{pmatrix} =: v_{d,p}
\end{align}
so that its centralizer subgroup is conjugate to
\begin{align*}
\setc{g \in \PGL_2(\Qp)}{g v_{d,p} = v_{d,p} g } = \setc{\begin{pmatrix}
a & 0 \\ 
0 & 1
\end{pmatrix} }{a \in \Qp^\times}.
\end{align*}
This concludes the proof.
\end{proof}

\subsection{Compact torus orbits}\label{subsection:CM-Compact torus orbits}

Let $d<0$ be a discriminant, let $K = \Q(\sqrt{d})$ and let $\Fa$ be a proper $R_d$-ideal.
In this subsection we will study the toral packet
\begin{align*}
\PGL_2(\Q)\torus_\Fa(\adele) \subset X_\adele 
= \lquot{\PGL_2(\Q)}{\PGL_2(\adele)}
\end{align*}
and project it to the $p$-adic extension
\begin{align*}
X_{\set{\infty,p}} = \lquot{\PGL_2(\Zinvp)}{\PGL_2(\R \times \Qp)}.
\end{align*}
If projected further to the complex modular curve $Y_0(1) = \lquot{\SL_2(\Z)}{\mathbb{H}}$, this packet will essentially yield the CM points for the discriminant $d$.

It can be proven using standard methods that the packet $\PGL_2(\Q)\torus_\Fa(\adele) $ is compact (this will also follow from Proposition \ref{prop:CM-number of orbits}).
As the subgroup $\torus_\Fa(\R\times \widehat{\Z})$ of $\torus_\Fa(\adele)$ is open, we can write
\begin{align}\label{eq:CM-orbit decomposition}
\PGL_2(\Q)\torus_\Fa(\adele) = \Disjointunion_{\rho \in \CR_\Fa} \PGL_2(\Q)\rho \torus_\Fa(\R \times \widehat{\Z})
\end{align}
where $\mathcal{R}_\Fa  \subset \PGL_2(\R \times \widehat{\Z})$ is a finite set of representatives.
Here, we used that the fact that $\PGL_2$ has class number one (see \eqref{eq:class number one for PGL2} and compare to Section \ref{section:packets}).

\subsubsection{Counting orbits and the Picard group}\label{section:Picard group}
Note that the number of $\torus_\Fa(\R \times \widehat{\Z})$-orbits $|\CR_\Fa|$ is exactly the cardinality of the finite abelian group
\begin{align*}
\lrquot{\torus_\Fa(\Q)}{\torus_\Fa(\adele_f)}{\torus_\Fa(\widehat{\Z})} \simeq \lrquot{\torus_\Fa(\Q)}{\torus_\Fa(\adele)}{\torus_\Fa(\R\times\widehat{\Z})}.
\end{align*}
As we will now discuss, this group is isomorphic to the Picard group $\Cl(R_d)$ of the order $R_d$, which is by definition the (finite) group of $K^\times$-homothety classes of proper $R_d$-ideals.
Note that an $R_d$-ideal $\Fb$ is invertible if and only if it is proper and that in this the inverse is given by $\Fa^{-1} = \setc{\lambda\in K}{\lambda \Fa \subset R_d}$ (cf.~\cite[Prop.~2.1]{duke}, \cite[Lemma 7.5]{coxprimesoftheform})
The Picard group of $R_d$ is very strongly connected to the set of binary forms of discriminant $d$ (cf.~\cite{duke} and \cite{coxprimesoftheform}).

\begin{proposition}[Cardinality of $\CR_\Fa$]\label{prop:CM-number of orbits}
There is an isomorphism
\begin{align*}
\lrquot{\torus_\Fa(\Q)}{\torus_\Fa(\adele_f)}{\torus_\Fa(\widehat{\Z})} 
\simeq \Cl(R_d).
\end{align*}
In particular, $|\mathcal{R}_\Fa| = |\Cl(R_d)| =: h_d$.
\end{proposition}

We begin by giving an idelic interpretation of the Picard group of the order $R_d$.
For this, we denote by\footnote{In fact, we have $K \otimes \Qp \simeq \prod_{\mathfrak{p} \mid p\mathcal{O}_K} K_{\mathfrak{p}}$ with  $R_d \otimes \Zp \simeq\prod_{\mathfrak{p} \mid p\mathcal{O}_K} (R_d)_{\mathfrak{p}}$.} $\adele_{K,f} = \restrprod{p} K\otimes \Qp$
the ring of finite adeles of $K$.
Moreover, we set $\widehat{R_d} = \prod_{p}(R_d)_p$ where $(R_d)_p = R_d \otimes \Zp$ for any prime $p$.

\begin{lemma}[Idelic interpretation of the Picard group]
We have
\begin{align*}
\Cl(R_d) \simeq \lrquot{K^\times}{\adele_{K,f}^\times}{\widehat{R_d}^\times}.
\end{align*}
\end{lemma}

\begin{proof}
We consider the map of completions $\Fb \mapsto (\Fb_p)_p = (\Fb\otimes \Zp)_p$ on the set of non-zero $R_d$-ideals.
Notice that for any non-zero $R_d$-ideal $\Fb$ we have $\Fb_p = (R_d)p$ for all but finitely many primes $p$.

Recall that $\Fb$ is proper if and only if $\Fb$ is locally principal (cf.~\cite[Prop.~2.1]{duke}) i.e.~for every $p$ there exists $ \lambda_p\in K\otimes \Qp$ such that $\Fb_p = \lambda_p(R_d)_p$.
Notice that the choice of $\lambda_p$ is uniquely determined up to a unit in $(R_d)_p$ or in other words $\lambda = (\lambda_p)_p\in \adele_{K,f}^\times$ is uniquely determined up to a unit in $\widehat{R_d}^\times$.

Conversely, given any tuple $(\Fb_p)_p$ of ideals, where $\Fb_p$ is a principal $(R_d)_p$-ideal for every $p$ and $\Fb_p = (R_d)_p$ for all but finitely many primes $p$ there is a proper ideal $\Fb$ with $\Fb \otimes \Zp = \Fb_p$ for all $p$.
It is given by $\Fb = \bigcap_p (K \cap \Fb_p)$.

We thus obtain a bijection
\begin{align*}
\Phi:\setc{\Fb}{\Fb \text{ proper } R_d-\text{ideal}}
\to \rquot{\adele_{K,f}^\times}{\widehat{R_d}^\times}
\end{align*}
Notice that $K^\times$ acts on both sides and $\Phi(\alpha \Fb) = \alpha \Phi(\Fb)$ for any $\alpha \in K^\times$. 
Taking the quotient with $K^\times$ on both domain and target of the above map $\Phi$ shows the lemma.
\end{proof}

\begin{proof}[Proof of Proposition \ref{prop:CM-number of orbits}]
Let $\BL = \setc{x\in \Mat_2}{\psi_\Fa(\lambda) x = x \psi_\Fa(\lambda) \text{ for all }\lambda\in K}$ be the subspace defined by the image of $K$ under $\psi_\Fa$ and let $\BL^\times$ be the group of invertible elements in $\BL$.
Set $\BL^\times(\Zp) = \GL_2(\Zp) \cap \BL(\Qp)$ for  any prime $p$ and $\BL^\times(\widehat{\Z}) = \prod_{p} \BL^\times(\Zp)$.
Since $\G_m$ has class number one, the natural map
\begin{align*}
\lrquot{\BL^\times(\Q)}{\BL^\times(\adele_f)}{\BL^\times(\widehat{\Z})}
\to
\lrquot{\torus_\Fa(\Q)}{\torus_\Fa(\adele_f)}{\torus_\Fa(\widehat{\Z})}
\end{align*}
is bijective and we may consider the quotient on the left just as well.

The map $\psi_\Fa$ induces an isomorphism between $K$ and $\BL(\Q)$ (see also Claim \ref{claim:Q-points torus}), between $K_p$ and $\BL(\Qp)$ for any prime $p$ (for similar reasons) and therefore also between $\adele_{K,f}$ and $\BL(\adele_f)$.
It remains to show that the image of $(R_d)_p$ under $\psi_\Fa: K_p\to\BL(\Qp)$ is equal to $\BL(\Zp) = \BL(\Qp) \cap \Mat_2(\Zp)$.

For this, we notice that $\BL(\Zp)$ is by definition equal to the set of $X \in \BL(\Qp)$ such that $X.\Zp^2 \subset \Zp^2$. Since $\Zp^2 = \iota_\Fa(\Fa_p)$ this shows that
\begin{align*}
\psi_\Fa^{-1}(\BL(\Zp)) = \setc{\lambda\in K_p}{\lambda.\Fa_p \subset \Fa_p}.
\end{align*}
Thus, $\psi_\Fa^{-1}(\BL(\Zp)) \supset (R_d)_p$.
Conversely, if $\lambda\in K_p$ satisfies $\lambda.\Fa_p \subset \Fa_p$ then $\lambda.(R_d)_p \subset(R_d)_p$ since $\Fa_p$ is principal and thus $\lambda_p \in (R_d)_p$ as $1\in(R_d)_p$.
\end{proof}

\subsubsection{Generation of integer points and ideals}
In Section \ref{section:generating intpts} we already used the packet $\G(\Q)\torus_{v_\Fa}(\adele)$ to generate additional traceless elements in $\Mat_2(\Z)$ of determinant $d$.
We repeat the argument here, but phrase everything in terms of ideals.

\begin{lemma}[Generating ideals]\label{lemma:CM-generating ideals}
For every $\rho \in \CR_\Fa$ there is a proper $R_d$-ideal $\Fa_\rho$ so that $\rho v_\Fa \rho^{-1} = v_{\Fa_\rho}$ (with respect to a specific basis) and in particular $\rho \torus_\Fa \rho^{-1} = \torus_{\Fa_\rho}$.
\end{lemma}

\begin{proof}
Write $\rho \in \CR_\Fa$ as $\rho = \gamma h$ for $\gamma \in \PGL_2(\Q)$ and $ h \in \torus_\Fa(\adele)$ and choose a representative $\gamma \in  \GL_2(\Q)$. 
We first claim that $\Z^2 \gamma$ is preserved under right-multiplication with $\psi_\Fa(R_d)$, which then implies that $\Fa_{\rho}:= \iota_{\Fa}(\Z^2 \gamma)$ is an $R_d$-ideal. Let $b \in R_d$ and note that $\Z^2 \gamma\psi_\Fa(b) = \Z^2 \gamma\psi_\Fa(b)\gamma^{-1}\gamma$. By the choice of $ \gamma $,
\begin{align*}
\Mat_2(\Q) \ni \gamma\psi_\Fa(b)\gamma^{-1} = \rho \psi_\Fa(b)\rho^{-1} \in \Mat_2(\R \times \widehat{\Z})
\end{align*}
Therefore, $\gamma\psi_\Fa(b)\gamma^{-1} \in \Mat_2(\Z)$ and $\Z^2 \gamma\psi_\Fa(b) \subset \Z^2 \gamma$. Observe that by definition of $\Fa_\rho$, we have $v_{\Fa_\rho} = \gamma v_\Fa \gamma^{-1} = \rho v_\Fa \rho^{-1}$ in the basis $\iota_{\Fa}(e_1 \gamma),\iota_{\Fa}(e_2 \gamma)$.

It remains to show that $\Fa_\rho$ is proper.
So let $b \in K$ with $b\Fa_\rho \subset \Fa_\rho$ or in other words with $\Z^2 \gamma\psi_\Fa(b) \subset \Z^2\gamma$.
Then $\gamma\psi_\Fa(b)\gamma^{-1} = \rho \psi_\Fa(b) \rho^{-1} \in \Mat_2(\Z)$.
In particular, we have $\rho_p\psi_\Fa(b)\rho_p^{-1} \in \Mat_2(\Zp)$ for any prime $p$. But $\rho_p \in \PGL_2(\Zp)$ so that $\psi_\Fa(b) \in \Mat_2(\Zp)$ for any prime $p$. Thus, $\psi_\Fa(b) \in \Mat_2(\Q)\cap\Mat_2(\widehat{\Z}) = \Mat_2(\Z)$ and the claim follows from properness of $\Fa$ (see \eqref{eq:CM-Mat_2 <-> O_K}).
\end{proof}

\subsection{Packets in the $p$-adic extension}
Let $p$ be a fixed odd prime and let $d<0$ be be a discriminant satisfying Linnik's condition at $p$.
As before, let $\Fa$ be a proper $R_d$-ideal in $K = \Q(\sqrt{d})$.
The projection $\mathcal{P}(\Fa, \set{\infty,p})$ of the packet $\G(\Q)\torus_{\Fa}(\adele)$ in $X_{\adele}$
onto the $p$-adic extension
\begin{align*}
X:= X_{\set{\infty,p}} = \lquot{\PGL_2(\Zinvp)}{\PGL_2(\R \times \Qp)}.
\end{align*}
is $\torus_\Fa(\R \times \Qp)$-invariant and equal to the disjoint\footnote{Compare to the proof of Lemma \ref{lemma:CM-generating ideals}.} union
\begin{align*}
\Disjointunion_{\rho \in \CR_\Fa}\PGL_2(\Zinvp)(\rho_\infty,\rho_p) \torus_\Fa(\R \times \Zp).
\end{align*}
We denote
\begin{align*}
M:&= \PO_2(\R) \times \setc{\begin{pmatrix}
x & 0 \\ 
0 & 1
\end{pmatrix} \in \PGL_2(\Zp) }{x \in \Zp^\times}
\end{align*} 
and $(v_{d,\infty},v_{d,p})=: v_d$ where $v_{d,\infty},v_{d,p}$ were defined in Equations \eqref{eq:CM-def v_d,infty} and \eqref{eq:CM-def v_d,p}. 
An elementary computation shows that for any proper $R_d$-ideal $\Fb$ there is an element $g_\Fb \in \PGL_2(\R \times \Zp)$ with
\begin{align}\label{eq:CM-Zp conjugacy}
g_\Fb^{-1} v_\Fb g_\Fb = v_d.
\end{align}
\begin{remark}\label{remark:CM-special choice R-conjugacy}
Let $\phi:K \to \C \simeq \R^2$ be a field embedding. For any proper $R_d$-ideal $\Fb$ and a $\Z$-basis $b_1,b_2$ of $\Fb$ we may choose $g_{\Fb,\infty}$ as
\begin{align*}
g_{\Fb,\infty} = \begin{pmatrix}
\phi(b_1) \\ 
\phi(b_2)
\end{pmatrix}.
\end{align*}
\end{remark}

Furthermore, the choice of $g_\Fb$ is unique up to a right factor in $M$ and we have $g_\Fa^{-1}\torus_\Fa(\R \times \Zp)g_\Fa = M$. Observe that for $\rho \in \CR_\Fa$ the element $(\rho_\infty,\rho_p) g_\Fa =: g_{\Fa_\rho}$ satisfies $g_{\Fa_\rho}^{-1}v_{\Fa_\rho}g_{\Fa_\rho} = v_d$ and therefore
\begin{align*}
\mathcal{P}(\Fa, \set{\infty,p})g_\Fa 
= \Disjointunion_{\rho \in \CR_\Fa}\PGL_2(\Zinvp)g_{\Fa_\rho} M
\end{align*}
The orbit $\PGL_2(\Zinvp)g_{\Fb} M$ associated to an ideal $\Fb$ is independent of the choice of basis on $\Fb$. If $\Fb,\Fb'$ are two equivalent ideals, then respective bases may be chosen so that $v_\Fb = v_{\Fb'}$ and in particular, $\PGL_2(\Zinvp)g_{\Fb} M = \PGL_2(\Zinvp)g_{\Fb'} M$. Thus, we will write the projection of the packet $\PGL_2(\Q)\torus_\Fa(\adele)$ onto the $p$-adic extension after right multiplication with $g_\Fa$ as
\begin{align*}
\mathcal{G}_d := \Disjointunion_{[\Fb]}\PGL_2(\Zinvp)g_{\Fb} M
\end{align*}
where the union runs over all ideal classes by Proposition \ref{prop:CM-number of orbits}. Note that $\mathcal{G}_d$ does not depend on the initial choice of the ideal $\Fa$ and is not only invariant under $M$ but also under the non-compact group
\begin{align*}
A:&= \setc{\begin{pmatrix}
x & 0 \\ 
0 & 1
\end{pmatrix} \in \PGL_2(\Qp) }{x \in \Qp^\times}
\end{align*}
since $\mathcal{P}(\Fa, \set{\infty,p})$ was invariant under $\torus_\Fa(\Qp)$.

\subsubsection{Invariant measures, volume bounds and entropy}

For any discriminant $d<0$ the packet $\mathcal{G}_d$ is naturally equipped with a volume obtained by pushing forward the volume measure on the packet $\PGL_2(\Q)\torus_{\Fa}(\adele)g_\Fa$ where $\Fa$ is a proper $R_d$-ideal.
For any $\rho\in \mathcal{R}_{\Fa}$ we have
\begin{align*}
\vol\big(\PGL_2(\Q)\rho\torus_\Fa(\R \times \widehat{\Z})g_\Fa\big)
&= \vol\big(\PGL_2(\Q)h\torus_\Fa(\R \times \widehat{\Z})\big)\\
&= \vol\big(\PGL_2(\Q)\torus_\Fa(\R \times \widehat{\Z})\big)
= |\torus_\Fa(\Z)|^{-1}
\end{align*}
where $h\in \torus_\Fa(\adele)$ is such that $\PGL_2(\Q)\rho = \PGL_2(\Q)h$.
Therefore, any $M$-orbit in the packet $\mathcal{G}_d$ has volume $|\torus_\Fa(\Z)|^{-1} \asymp |R_d^\times|^{-1}\asymp 1$.
Thus, $\vol(\mathcal{G}_d) \asymp h_d$ where $h_d$ is the size of the Picard group of $R_d$ (see Proposition \ref{prop:CM-number of orbits}).
As in Section \ref{section:size picard group} this implies (essentially by Siegel's lower bound) the asymptotics
\begin{align}\label{eq:CM-total volume}
\vol(\mathcal{G}_d) = |d|^{\frac{1}{2} + o(1)}.
\end{align}
Let $\mu_d := \frac{1}{\vol(\mathcal{G}_d)}\vol$ be the invariant probability measure on the packet $\mathcal{G}_d$.

Consider the map $T:X \to X, x \mapsto xa$, where
\begin{align*}
a := \begin{pmatrix}
1 & 0 \\ 
0 & p
\end{pmatrix} \in \PGL_2(\Qp). 
\end{align*}

\begin{theorem}[Equidistribution of the collections $\mathcal{G}_d$]\label{thm:CM equidistribution toral orbits}
Let $p$ be an odd prime.
As $d\to -\infty$ amongst the discriminants satisfying Linnik's condition at $p$ any \wstar-limit of the measures $\mu_d$ is a probability measure of maximal entropy $\log(p)$ with respect to $T$.
\end{theorem}

In fact, proceeding as in the proof of Theorem \ref{thm:main} the statement of Theorem~\ref{thm:CM equidistribution toral orbits} 
implies that any \wstar-limit $\mu$ is invariant under $\PSL_2(\R\times \Qp)$.
As $\mu$ is also $A$-invariant and the group generated by $\PSL_2(\R\times \Qp)$ and $A$ is $\PGL_2(\R \times \Qp)$, Theorem \ref{thm:CM equidistribution toral orbits} shows equidistribution of the packets $\mathcal{G}_d$.

\subsection{Proof of Theorem \ref{thm:main} in the case $\quat =\Mat_2$}\label{section:proof main thm Mat2}

As previously announced, we will use the formula of Abramov and Rokhlin \cite{abramovrokhlin} (or \cite[Cor.~2.21]{vol2}) for entropy transport to deduce Theorem~\ref{thm:main} from Theorem~\ref{thm:CM equidistribution toral orbits}.

Let us quickly state formula in our context.
Let $a \in A \subset \PGL_2(\adele)$ be as in Theorem \ref{thm:CM equidistribution toral orbits}, let $\nu$ be a probability measure on $X_\adele$ and let $\tilde{\nu}$ be the pushforward of $\nu$ under the projection $\pi_{\set{\infty,p}}:X_\adele \to X = X_{\set{\infty,p}}$. Then
\begin{align*}
h_\nu(a) = h_{\tilde{\nu}}(T) + h_\nu(a|\mathcal{A})
\end{align*}
where $\mathcal{A}$ is the preimage of the Borel $\sigma$-algebra on $X$ in $X_\adele$ and $h_\nu(a|\mathcal{A})$ denotes the entropy of $\nu$ with respect to $a$ conditional on $\mathcal{A}$.

\begin{proof}[Proof of Theorem \ref{thm:main} in the case $\quat =\Mat_2$]
Let $d_\ell<0$ be a sequence of negative discriminants satisfying Linnik's condition at $p$ with $d_\ell \to -\infty$ as $\ell$ goes to infinity.
Given $\ell$ we let $\Fa_{\ell}$ be a proper $R_{d_\ell}$-ideal.
Let $\nu_\ell$ be the normalized Haar measure on the packet $\G(\Q)\torus_{\Fa_\ell}(\adele)g_{\Fa_\ell}$ where $g_{\Fa_\ell} \in \PGL_2(\R \times \Zp)$ was defined in \eqref{eq:CM-Zp conjugacy}.

Since $\PGL_2(\Zp)$ is compact, it suffices to show that any \wstar-limit of the measures $\nu_\ell$ is a probability measure and is invariant under $\PSL_2(\adele):= \sicover\PGL_2(\adele)$ (see also Section \ref{section:CM-on the traceless guy} for primitivity).
We may assume without loss of generality that $\nu_\ell \to \nu$ as $\ell \to \infty$ for a finite measure $\nu$ on $X_\adele$.

First, notice that the pushforward of $\nu_\ell$ under the map $\pi_{\set{\infty,p}}$ is (by definition) the measure $\mu_{d_\ell}$ on the packet $\mathcal{G}_{d_\ell}$.
By Theorem \ref{thm:CM equidistribution toral orbits} we have that $\tilde{\nu} = (\pi_{\set{\infty,p}})_\ast\nu$ as the limit of the measures $\mu_{d_\ell}$ is a probability measure and hence $\nu$ is also a probability measure.

By definition, $\nu$ is $a$-invariant.
Furthermore, we have $h_{\nu}(a) \leq \log(p)$ (cf.~\cite[Thm.~7.9]{pisa}).
On the other hand, by Theorem \ref{thm:CM equidistribution toral orbits}
\begin{align*}
h_\nu(a) = h_{\tilde{\nu}}(T) + h_\nu(a|\mathcal{A})\geq \log(p)
\end{align*}
so $\nu$ has maximal entropy.
The arguments in the proof of Theorem \ref{thm:main} in the case where $\quat$ is ramified at $\infty$ (see Section \ref{subsection: From maximal entropy to additional invariance}) now apply to show that $\nu$ is indeed $\PSL_2(\adele)$-invariant.
\end{proof}

\subsection{Obtaining CM points and the proof of Linnik's Theorem B}\label{subsection:cm points}

In this subsection, we show that Theorem~\ref{thm:main} implies Linnik's Theorem B. 
(Note that one could just as well use the comments after Theorem \ref{thm:CM equidistribution toral orbits}.)

For a fixed discriminant $d<0$ the image of the packet $\mathcal{G}_d$ under the projections
\begin{align*}
\lquot{\PGL_2(\Zinvp)}{\PGL_2(\R \times \Qp)} \to \lquot{\PGL_2(\Z)}{\PGL_2(\R)} \to \lquot{\PGL_2(\Z)}{\mathbb{H}} =: \mathcal{X}_2
\end{align*}
is the finite set
\begin{align*}
\mathcal{H}_d := \left\lbrace\PGL_2(\Z).(g_{\Fa,\infty}.\ii): [\Fa] \in \Cl(R_d)\right\rbrace
\end{align*}
of cardinality $h_d$ (see also Lemma \ref{lemma:CM-generating qf} below).

Notice that the projections of the packets $\mathcal{G}_d$ onto the real quotient $\lquot{\PGL_2(\Z)}{\PGL_2(\R)}$ are equidistributed (when $d \to -\infty$ amongst the discriminants satisfying Linnik's condition at $p$).
In fact, by Theorem \ref{thm:main} (and continuity of projections) any \wstar-limit the push-forward of the natural invariant measures $\mu_d$ on the packets $\mathcal{G}_d$ to the real quotient has to be $\PSL_2(\R)$-invariant. Also, $\lquot{\PSL_2(\Z)}{\PSL_2(\R)} \simeq \lquot{\PGL_2(\Z)}{\PGL_2(\R)}$.

This shows that the subsets $\mathcal{H}_d\subset \mathcal{X}_2$ are equidistributed. 
It thus suffices to show that the set $\mathcal{H}_d$ is exactly the set of CM points associated to $d$ in order to prove Linnik's Theorem B.
To illustrate this (and for further use) we will first explain the connection between orbits in $\mathcal{G}_d$ and integral~forms.

\subsubsection{Producing binary forms}
Recall that there is a correspondence between binary quadratic forms over $\R$, real symmetric $2$-by-$2$ matrices and real traceless $2$-by-$2$ matrices given by
\begin{align*}
ax^2+bxy+cy^2 \leftrightarrow
\begin{pmatrix}
a & \frac{b}{2} \\ 
\frac{b}{2} & c
\end{pmatrix} 
\leftrightarrow
\begin{pmatrix}
b & -2a \\ 
2c & -b
\end{pmatrix} 
\end{align*}
The action of $\GL_2(\R)$ on $\mathfrak{sl}_2(\R)$ by conjugation induces an action on $\Sym_2(\R)$ via
\begin{align*}
\begin{pmatrix}
a & \frac{b}{2} \\ 
\frac{b}{2} & c
\end{pmatrix}  \mapsto \frac{1}{\det(g)} g\begin{pmatrix}
a & \frac{b}{2} \\ 
\frac{b}{2} & c
\end{pmatrix} g^T
\end{align*}
for $g \in \GL_2(\R)$ under this correspondence. The analogous statement holds over $\Qp$ or more generally any field of characteristic not $2$. 

\begin{lemma}[Points in $\mathcal{G}_d$ and quadratic forms]\label{lemma:CM-generating qf}
Let $d<0$ be a discriminant.
To a point $\PGL_2(\Zinvp)g \in \mathcal{G}_d$ where $g \in \PGL_2(\R \times \Zp)$ we associate the quadratic form corresponding to the traceless matrix $gv_{d}g^{-1}$.
This quadratic form is integral, primitive, has discriminant $d$ and is uniquely determined up to $\GL_2(\Z)$-equivalence. Furthermore, the quadratic forms associated to two points on different $M$-orbits in $\mathcal{G}_d$ are inequivalent.
\end{lemma}

As the set of primitive binary forms of discriminant $d$ is in bijection with the Picard group (see \cite[Sec.~2]{duke}) any such binary form can be constructed as in the lemma.

\begin{proof}
Let $\PGL_2(\Zinvp)g \in \mathcal{G}_d$ with $g \in \PGL_2(\R \times \Zp)$ and consider the traceless matrix $gv_{d}g^{-1}$. This matrix has integral entries: Writing $g = \gamma g_\Fa m$ for~$m \in ~M$, $\gamma \in~\PGL_2(\Z)$ and a proper $R_d$-ideal $\Fa$, we see that
\begin{align*}
gv_dg^{-1} = \gamma g_\Fa v_d g_\Fa^{-1} \gamma^{-1} = \gamma v_\Fa \gamma^{-1} \in \Mat_2(\Z).
\end{align*}
By the discussion of Section \ref{section:CM-on the traceless guy} the quadratic form attached to $v_\Fa$ is integral (as the off-diagonal entries are divisible by two) and primitive.
Thus, the quadratic form $q$ associated to $gv_dg^{-1}$ is integral, primitive and has discriminant $-\det(gv_dg^{-1}) = d$ as desired.

Now let $\PGL_2(\Zinvp)g,\PGL_2(\Zinvp)\bar{g} \in \mathcal{G}_d$ with $g,\bar{g} \in \PGL_2(\R \times \Zp)$ so that there exists $\gamma\in \PGL_2(\Z)$ with $gv_{d}g^{-1} = \gamma\bar{g}v_d\bar{g}^{-1}\gamma^{-1}$. By replacing $\bar{g}$ we may assume that $\gamma = I$. Notice that $h:= \bar{g}^{-1}g$ commutes with $v_d$ and thus lies in $M$.
\end{proof}

\begin{proof}[Proof of Linnik's Theorem B assuming Theorem \ref{thm:CM equidistribution toral orbits}]
Consider a CM point
\begin{align*}
x = \frac{-b + \sqrt{-d}i}{2a} \in \BH, \quad d= b^2 - 4ac.
\end{align*}
of discriminant $d$ where we assume $a\geq 0$ for the sake of concreteness. The matrix
\begin{align*}
g_x:= \frac{1}{|d|^{\frac{1}{4}}\sqrt{2a}} \begin{pmatrix}
\sqrt{-d} & -b \\ 
0 & 2a
\end{pmatrix} \in \SL_2(\R)
\end{align*}
yields $x$ as $g_x.\ii = x$ and satisfies the equation
\begin{align*}
g_x v_{d,\infty} g_x^{-1} = \begin{pmatrix}
b & 2c \\ 
-2a & -b
\end{pmatrix}.
\end{align*}
By Lemma \ref{lemma:CM-generating qf} there is a proper $R_d$-ideal $\Fa$, $k \in \PO_2(\R)$ and $\gamma \in \PGL_2(\Z)$ so that $\gamma g_{\Fa,\infty} k = g_x$. This proves that 
\begin{align*}
\PGL_2(\Z).x \in \mathcal{H}_d = \PGL_2(\Z).(g_x.\ii) = \PGL_2(\Z).(g_{\Fa,\infty}.\ii) \in \mathcal{H}_d.
\end{align*}
On the other hand, the equation above shows that the CM point reproduces its underlying quadratic form.

As there are exactly $h_d$ $\GL_2(\Z)$-equivalence classes of primitive integral quadratic forms of discriminant $d$ there are $h_d$ $\PSL_2(\Z)$-equivalence classes of CM points of discriminant $d$, which proves the other inclusion.
\end{proof}

\subsection{Ideal classes and heights}\label{subsection:CM-Ideal classes and heights}

Here, we derive two important estimates concerning the measures $\mu_d$. 
The first roughly states that there are not too many orbits in $\mathcal{G}_d$ ``high'' in the cusp. 
The second estimate answers the question as to how many ``low lying'' orbits in $\mathcal{G}_d$ are close together (Linnik's basic lemma). 

\subsubsection{Mass in the cusp}
The \emph{height} $\height(x)$ of a point $x = [\Lambda] \in X$ is defined by
\begin{align*}
\frac{1}{\height(\Lambda)} = \frac{\min_{\lambda\in \Lambda \setminus \set{0}} \norm{\lambda_\infty}_\infty \norm{\lambda_p}_p}{\covol(\Lambda)^{\frac{1}{2}}}
\end{align*}
where $\Lambda$ is a lattice\footnote{Recall that a point in $X$ is naturally identified with a homothety class of lattices in $(\R \times \Qp)^2$, where a lattice is a $\Zinvp$-submodule of the form $\Zinvp^2g$ with $g \in \GL_2(\R \times \Qp)$.} representing $x$.
It is straightforward to verify that the height of a point in $X$ is equal to the height of its image under the projection to $\lquot{\PGL_2(\Z)}{\PGL_2(\R)}$. 
Let $X_{\geq H}$ be the set of points in $X$ of height bigger or equal than $H$ and similarly define $X_{< H}$, $X_{\leq H}$ and $X_{> H}$.
Note that $X_{\leq H}$ is compact and that a uniform injectivity radius on it is for instance given by $\frac{1}{3}H^{-2}$.

Recall that the norm $\Nr_d(\Fa)$ of a proper $R_d$-ideal $\Fa\subset\Q(\sqrt{d})$ (with respect to $R_d$) is given by
\begin{align*}
\Nr_d(\Fa) = \frac{[R_d:R_d\cap\Fa]}{[\Fa:R_d\cap \Fa]},
\end{align*}
satisfies $\Nr_d(\alpha R_d) = \Nr_{\Q(\sqrt{d})/\Q}(\alpha)$
and is multiplicative (cf.~\cite[Lemma~7.14]{coxprimesoftheform}).

\begin{proposition}[Orbits high in the cusp]\label{prop:CM-orbits high in the cusp}
Let $d <0$ be a discriminant and let $\Fa$ be a proper $R_d$-ideal in $K = \Q(\sqrt{d})$.
Choose $g_\Fa$ as in \eqref{eq:CM-Zp conjugacy}. The following statements are equivalent:
\begin{enumerate}[(i)]
\item $\PGL_2(\Zinvp)g_\Fa M \intersect X_{\geq H}$ is non-empty.
\item There exists $\lambda \in \Fa$ with $\Nr_d(\lambda \Fa^{-1}) \leq \frac{1}{2} \sqrt{|d|} H^{-2}$.
\end{enumerate}
In particular, $\mathcal{G}_d$ does not contain a point of height $> |d|^{\frac{1}{4}}$. Furthermore, the number of orbits in $\mathcal{G}_d$, which intersect $X_{\geq H}$, is bounded by the number of proper $R_d$-ideals $\Fb \subset K$ of norm $\Nr_d(\Fb) \leq \frac{1}{2} \sqrt{d} H^{-2}$.
\end{proposition}

Note that for any $\Fa$ all points in $\PGL_2(\Zinvp)g_\Fa M$ have the same height. In particular, $\PGL_2(\Zinvp)g_\Fa M \intersect X_{\geq H}$ is either empty or equal to $\PGL_2(\Zinvp)g_\Fa M$.

\begin{proof}
Let $\phi: K \to \C\isom \R^2$ be a field embedding and choose $g_{\Fa,\infty}$ as in Remark~\ref{remark:CM-special choice R-conjugacy}. 
Now observe that $\PGL_2(\Z)g_{\Fa,\infty}$ has height $\geq H$ if and only if $\phi(\Fa)$ contains an element of norm $ \leq \sqrt{\covol(\phi(\Fa))}H^{-1}$. 
The latter is equivalent to the condition that $\Fa$ contains an element $\lambda$ with $\Nr_{K/\Q)}(\lambda) \leq \frac{1}{2} \sqrt{|d|} H^{-2} \Nr_d(\Fa)$. 
In particular, $\mathcal{G}_d$ does not contain a point of height $> |d|^{\frac{1}{4}}$, since the ideal $\lambda \Fa^{-1}$ is contained in $R_d$ by definition of the inverse $\Fa^{-1}$ and such ideals have integral norm.

Any primitive $\lambda \in \Fa$ as in $(ii)$ is unique up to a sign. To any $\Fa$ which satisfies~$(i)$ thus corresponds the unique integral ideal $\lambda\Fa^{-1}$ where $\lambda\in \Fa$ is chosen to be primitive and as in $(ii)$.
%
%
\end{proof}

\begin{proposition}[''Not too much mass high in the cusp'']
Let $d<0$ be a discriminant.
For all $\varepsilon > 0$ and $H > 1$ we have
\begin{align*}
\mu_d(X_{\geq H}) \ll_\varepsilon |d|^{\varepsilon} H^{-2}
\end{align*}
\end{proposition}

\begin{proof}
By Proposition \ref{prop:CM-orbits high in the cusp}, the number of orbits in $\mathcal{G}_d$ which intersect $X_{\geq H}$ is bounded by the number of proper $R_d$-ideals $\Fa\subset K = \Q(\sqrt{d})$ of norm $\leq \frac{1}{2} \sqrt{|d|} H^{-2}$. Counting lattice points shows that the latter is 
\begin{align*}
\ll_\varepsilon \left(\frac{ h_d }{\sqrt{|d|}} \sqrt{|d|} H^{-2}\right)^{1+\varepsilon} \ll \left({ h_d } H^{-2}\right)^{1+\varepsilon}.
\end{align*}
The same bound holds for the volume of $\mathcal{G}_d \intersect X_{\geq H}$, as all $M$-orbits have length $\asymp 1$. This yields the right estimate after normalization by the total volume (see Equation \eqref{eq:CM-total volume}).
\end{proof}

\subsubsection{Linnik's basic lemma}

We now turn to the following analogue of Proposition~\ref{prop: Linnik's basic lemma}.

As in Section \ref{section:linnik basic lemma} we will say that two points $x_1,x_2 \in X$ are $p$-adically $\delta$-close if $x_2 = x_1g$ for $g\in \PGL_2(\R \times \Zp)$ in an injective ball around the identity with $d(g_p,e)<\delta$.
In this case we also write $x_1 \widesim{\delta}_p x_2$.

\begin{proposition}[Linnik's basic lemma]\label{prop:CM-Linnik's basic lemma}
Let $d<0$ be a discriminant which fulfills Linnik's condition at $p$. 
Let $H \geq 1$ and let $r= \frac{1}{3}H^{-2}$.
For any $\delta>0$ with $|d|^{-\frac{1}{4}} \leq \delta \leq r$ and any $\varepsilon > 0$ we have
\begin{align*}
\mu_d\times\mu_d\Big(\setc{(x,y)\in (X_{\leq H})^2}{x_1 \widesim{\delta}_p x_2}\Big) 
\ll_\varepsilon H^{4} \delta^{3} |d|^{\varepsilon}.
\end{align*}
\end{proposition}

\begin{proof} Set $S$ be the set of pairs $(x_1,x_2)$ of points $x_1,x_2 \in X_{\leq H} \intersect \mathcal{G}_d$ so that $x_1,x_2$ are $p$-adically $\delta$-close and $d(x_1,x_2) < r$.

Let $(x_1,x_2) \in S$ and choose $g_1,g_2 \in \PGL_2(\R \times \Zp)$ with $\PGL_2(\Zinvp)g_i = x_i$ for $i= 1,2$ and $d_p(g_1,g_2) \leq \delta$. 
Then $\norm{g_i}_\infty \ll H$ where $\norm{\cdot}_\infty$ is given by $\norm{g}_\infty = \Tr(g_\infty^tg_\infty)^{\frac{1}{2}}$.
We attach to both points the integral quadratic form $q_i$ constructed in Lemma~\ref{lemma:CM-generating qf} and distinguish as in the proof of Proposition \ref{prop: Linnik's basic lemma} two cases:

\textbf{Case 1}: Assume that $q_1,q_2$ are equivalent or in other words that $x_1,x_2$ lie on the same $M$-orbit in $\mathcal{G}_d$. As the volume of a $\delta$-ball in $M$ is $\ll \delta$, the set of such $x_1,x_2$ has total volume $\ll_\varepsilon \delta |d|^{\frac{1}{2}+\varepsilon}$ which yields a contribution of $\ll_\varepsilon \delta |d|^{-\frac{1}{2}} |d|^\varepsilon \leq \delta^3  |d|^\varepsilon$ to the measure of $S$ in this case.

\textbf{Case 2}: Assume that $q_1, q_2$ are inequivalent and write $q_i= a_i x^2 + b_i xy + c_i y^2$ for $i=1,2$. The bound $\norm{g_i} \ll H$ yields $\max(|a_i|,|b_i|,|c_i|) \ll |d|^{\frac{1}{2}}H^2$. 
On the other hand, the bound $d_p(g_1,g_2) \leq \delta$ implies
\begin{align*}
\max(|a_1-a_2|_p,|b_1-b_2|_p,|c_1-c_2|_p) \ll \delta.
\end{align*}
Consider the integral quadratic form 
\begin{align*}
Q(x,y) = \disc\big(x(a_1,b_1,c_1)+y(a_2,b_2,c_2)\big) = dx^2 + e  xy + dy^2
\end{align*}
for some $e $. The bounds on the coefficients of $q_1$ and $q_2$ yield
\begin{align}
|2d-e | = |Q(1,-1)| &\ll |d| H^4.\label{eq:CM:Linnik basic lemma - bound 1} \\
|2d-e |_p &\ll \delta^2\label{eq:CM:Linnik basic lemma - bound 2}
\end{align}
For convenience we let $m = \lgauss{-2\log_p(\delta)-\log_p(C)}$ where $C>0$ is the implicit constant in \eqref{eq:CM:Linnik basic lemma - bound 2}.
Thus, $p^m | (2d-e)$.

We claim that $Q$ is non-degenerate.
Indeed, if we had $e  = \pm 2 d$ this would contradict the assumption that $d<0$ as
\begin{align*}
d(a_2 \mp a_1)^2 = Q(a_2,-a_1) = \disc(a_2(a_1,b_1,c_1)-a_1(a_2,b_2,c_2)) = (a_2b_1-a_1b_2)^2 .
\end{align*}
Let $N_{e ,d}$ be the number of inequivalent ways to represent the binary quadratic form $dx^2+ e  xy + dy^2$ by the ternary quadratic form $\disc$ up to $\SO_{\disc}(\Z)$-equivalence. By Theorem \ref{thm: representations qf}
\begin{align*}
N_{e ,d} \ll_\varepsilon f \max(|d|,|e |)^\varepsilon \ll_\varepsilon f |d|^\varepsilon
\end{align*}
where $f^2= \gcsd{d,e }$ is the greatest common square divisor. 
Here we used that $|e | \leq 2|d|+|2d-e |\ll |d|^{\frac{3}{2}}$ by \eqref{eq:CM:Linnik basic lemma - bound 1}.
By commensurability we may replace $\SO_{\disc}(\Z)$ by $\PGL_2(\Z)$ above. Let
\begin{align*}
\PGL_2(\Z)(q_1^{(1)},q_2^{(1)}),..., \PGL_2(\Z)(q_1^{(k)},q_2^{(k)})
\end{align*}
be a complete list of pairs of inequivalent quadratic forms, where $q_i^{(j)}$ is obtained from $ \PGL_2(\Zinvp)g_i^{(j)}$ as in the beginning. The number $k$ satisfies the bound
\begin{align*}
k &\leq \sum_{\substack{e:|2d-e|\leq L,\, e \neq \pm 2d,\\ p^m | (2d-e)}}N_{e ,d}
=\sum_{\substack{e':|e'|\leq L,\, p^m|e', \\ e' \neq 0,4d}} N_{2d-e ',d}  \\
&\leq 
\sum_{f^2|d} \ \sum_{\substack{e':|e'|\leq L, p^m|e',\\ f^2 = \gcsd{e',d},e' \neq 0,4d}} N_{2d-e ',d}
\ll_\varepsilon \sum_{f^2|d} f|d|^\varepsilon\sum_{\substack{e':|e'|\leq L, p^m|e',\\ f^2 = \gcsd{e',d}}} 1 \\
&\ll \sum_{f^2|d} f|d|^\varepsilon \frac{L}{f^2p^m} 
\ll_\varepsilon |d|^{1+2\varepsilon} H^4\delta^2 
\end{align*}
for $L \ll |d|H^4$ where the implicit constant is as in Equation \eqref{eq:CM:Linnik basic lemma - bound 1}. 

If now $(x_1,x_2)$, $(g_1,g_2)$ and $(q_1,q_2)$ are as in the beginning of the proof, there is some $j$ and $\gamma \in \PGL_2(\Z)$ so that $(\gamma. q_1,\gamma. q_2) = (q_1^{(j)},q_2^{(j)})$. 
Therefore, $x_i$ lies on the same $M$-orbit as $\PGL_2(\Zinvp)g_i^{(j)}$ for $i=1,2$. For fixed $j$, the set of $\delta$-close pairs $(x_1,x_2)$ lying on the orbit $\PGL_2(\Zinvp)g_i^{(j)}$ has measure $\ll \delta$. Thus, the total volume is
\begin{align*}
\ll \delta k \ll_\varepsilon \delta |d|^{1+2\varepsilon} H^4\delta^2 = |d|^{1+2\varepsilon} H^4\delta^3
\end{align*}
before normalization. After normalization, we obtain that the contribution to the measure of $S$ is $\ll_\varepsilon |d|^{3\varepsilon} H^4 \delta^3$ in this case.
\end{proof}

\subsection{Maximal Entropy}

In this subsection, we prove Theorem \ref{thm:CM equidistribution toral orbits} along the lines of \cite{duke} using the estimates derived in the last subsection. 

Let $d_\ell$ be a sequence of negative discriminants satisfying Linnik's condition for $p$ and write for simplicity $\mu_\ell$ for the measure $\mu_{d_\ell}$ defined in Section \ref{subsection:CM-Compact torus orbits}. 
Denote $\delta_\ell =~|d_\ell|^{-\frac{1}{4}}$ so that Proposition \ref{prop:CM-Linnik's basic lemma} applies for any height  $\ll \delta_\ell^{-1/2} = |d_\ell|^{1/8}$.
By restricting to a subsequence, we may assume that $\mu_\ell$ converges to some finite measure $\mu$ on $X$ with total mass at most $1$.
%

We will use the following proposition and postpone the proof to the next subsection.

\begin{proposition}\label{prop:CM-visits to the cusp}
Let $H > 1$ be a height. For $N \geq 1$ and a set of times $V$ in $ [-N, N]$ let
\begin{align*}
Z(V) := \setc{x \in X}{T^{\pm N}(x) \in X_{<H},\ \forall n \in [-N,N]: T^n(x) \in X_{\geq H} \iff n \in V}.
\end{align*}
Then $Z(V)$ can be covered by $\ll_H p^{2N-\frac{1}{2}|V|}$ Bowen $N$-balls and is non-empty for
\begin{align*}
\ll_H e^{2\log(p) \frac{\log(\log(H)) + c}{\log(H)}N} 
\end{align*}
many subsets $V \subset [-N,N]$ where $c$ is an absolute constant.
\end{proposition}

In this context, a (two-sided) Bowen $N$-ball in $X$ will always be a set of the kind $xB_N$ where $x$ is a point in $X$ and
\begin{align*}
B_N = \Intersect_{n=-N}^N a^{-n}B_\eta a^{n}
\end{align*}
is a Bowen ball in the group $\PGL_2(\R\times \Qp)$. The statement in Proposition \ref{prop:CM-visits to the cusp} is independent of the choice of radius $\eta > 0$: Given two radii $0 < \eta' < \eta$, the ball $B_\eta$ in $\PGL_2(\R \times \Qp)$ is covered by $\ll_{\eta,\eta'} 1$ shifts of the ball $B_{\eta'}$. For the purposes of this subsection, one fixed choice of radius $\eta > 0$ usually suffices.

\begin{lemma}\label{lemma:CM-weak star limits are prob.measures}
For all large enough heights $H$
\begin{align*}
\mu(X_{<H}) \geq 1- 2\log(p)\frac{\log(\log(H))}{\log(H)}.
\end{align*}
In particular, $\mu$ is a probability measure.
\end{lemma}

The proof is up to minor details the proof of Lemma 4.4 in \cite{duke} and uses the geometric interpretation provided by the Hecke tree (see Section \ref{section:CM-hecketree}) -- we will omit it here. The same conclusion applies to the following lemma.

\begin{lemma}\label{lemma:CM-construction of partition}
For any height $H> 1$ there is a finite partition $\mathcal{P}$ of $X$ such that for every $0<\kappa<1$ and every $N$ there is a measurable subset $ X' \subset T^{-N}X_{<H}$ satisfying the following conditions.
\begin{enumerate}
\item $\nu(X') \geq 1-2\nu(X_{\geq H})\kappa^{-1}$ for any $T$-invariant probability measure $\nu$.
\item $X'$ is a union of partition elements $S_1,\ldots,S_\ell \in \mathcal{P}_{-N}^N$, each of which is covered by at most $p^{\kappa(2N+1)}$ Bowen $(N,\eta)$-balls. Here, $\eta$ is assumed to be smaller than $1/p$ times an injectivity radius on $X_{<H}$.
\end{enumerate}
Fixing an invariant measure $\nu$ with $\nu(\partial X_{\geq H}) = 0$ the partition $\mathcal{P}$ may be constructed so that all partition elements have boundaries of measure zero. 
\end{lemma}

\begin{proof}[Proof of Theorem \ref{thm:CM equidistribution toral orbits}]
Let $H> 1$ be a fixed height so that the boundary of $X_{\geq H}$ has $\mu$-measure zero and let $\mathcal{P}$ be the partition from Lemma \ref{lemma:CM-construction of partition}. Define $\kappa = \mu(X_{\geq H})^{\frac{1}{2}}$, $N_i = \ugauss{-\log_p(\delta_i)}$ and choose $X_i \subset X$ according to Lemma \ref{lemma:CM-construction of partition}.

We define a new partition $\mathcal{Q}_i$, which is finer than $\mathcal{P}_{-N_i}^{N_i}$, by splitting all the $S$ in $\mathcal{P}_{-N_i}^{N_i}$, which are contained in  $X_i$, into at most $p^{\kappa(2N_i+1)}$ sets, which are contained in Bowen $N_i$-balls. As $\mathcal{Q}_i$ is finer than $\mathcal{P}_{-N_i}^{N_i}$ we have
\begin{align*}
|H_{\mu_i}(\mathcal{Q}_i) - H_{\mu_i}(\mathcal{P}_{-N_i}^{N_i}) | &= H_{\mu_i}(\mathcal{Q}_i | \mathcal{P}_{-N_i}^{N_i})
= \sum_{S \in \mathcal{P}_{-N_i}^{N_i}, S \subset X_i} \mu_i(S) H_{\mu_i|_S}(\mathcal{Q}_i)\\ &\leq  \kappa (2N_i + 1) \log(p).
\end{align*}

\begin{claim*}
$H_{\mu_i}(\mathcal{Q}_i) \geq (1-2\kappa^{-1}\mu_i(X_{\geq H})) (2-6\varepsilon)\log(p)N_i$
\end{claim*}

The claim implies the theorem as follows: By the claim and the computation above the claim
\begin{align*}
H_{\mu_i}(\mathcal{P}_{-N_i}^{N_i}) \geq (1-2\kappa^{-1}\mu_i(X_{\geq H})) (2-6\varepsilon)\log(p)N_i - \kappa (2N_i + 1) \log(p)
\end{align*}
Proceeding as in the proof of Theorem \ref{thm: Maximal entropy}, we obtain that for $\varepsilon > 0$ and all large enough $N_0$
\begin{align*}
H_{\mu_i}(\mathcal{P}_{-N_0}^{N_0}) \geq (1-2\kappa^{-1}\mu_i(X_{\geq H})) (2-6\varepsilon)\log(p)N_0 - \kappa (2N_0 + 1) \log(p) - \varepsilon N_0.
\end{align*}
By Lemma \ref{lemma:CM-construction of partition}, we may assume that boundaries of all partition elements in $\mathcal{P}$ are $\mu$-null sets. Thus, taking the limit as $i \to \infty$
\begin{align*}
H_{\mu}(\mathcal{P}_{-N_0}^{N_0}) \geq (1-2\kappa) (2-6\varepsilon)\log(p)N_0 - \kappa (2N_0 + 1) \log(p) - \varepsilon N_0.
\end{align*}
Dividing by $2N_0+1$ and letting $N_0$ go to infinity
\begin{align*}
h_\mu(T) \geq (1-2\mu(X_{\geq H})^\frac{1}{2}) (1-3\varepsilon)\log(p) - \mu(X_{\geq H})^\frac{1}{2} \log(p) - \tfrac{\varepsilon}{2}.
\end{align*}
Taking the limit $H \to \infty$ and $\varepsilon \to 0$, we have $\mu(X_{\geq H}) \to 0$ and thus $h_\mu(T) \geq \log(p)$ as desired.

To the proof of the claim: The entropy of $\mathcal{Q}_i$ satisfies
\begin{align*}
H_{\mu_i|_{X_i}}(\mathcal{Q}_i) 
\geq H_{\mu_i}(\mathcal{Q}_i|\{X_i,X\setminus X_i\}) \geq \mu_i(X_i) H_{\mu_i|_{X_i}}(\mathcal{Q}_i).
\end{align*}
The right hand side is bounded from below by
\begin{align*}
H_{\mu_i|_{X_i}}(\mathcal{Q}_i) 
&\geq -\log\bigg(\sum_{S \in \mathcal{Q}_i, S \subset X_i}\frac{\mu_i(S)^2}{\mu_i(X_i)^2} \bigg)\\
&= 2\log(\mu_i(X_i)) -\log\bigg(\sum_{S \in \mathcal{Q}_i, S \subset X_i}\mu_i(S)^2 \bigg).
\end{align*}
As any atom of $\mathcal{Q}_i$, which lies in $X_i$, is contained in a Bowen $N_i$-ball we obtain
\begin{align*}
\Union_{S\in \mathcal{Q}_i, S\subset X_i}S \times S 
\subset \Union_{j=1}^k \setc{(x,ya_i)}{
x \widesim{\eta p^{-N_i}}_p y}
\end{align*}
for $k \ll p^{N_i}$ and $a_1,\ldots,a_k \in A $. By Linnik's basic lemma (Proposition \ref{prop:CM-Linnik's basic lemma})
\begin{align*}
\sum_{S \in \mathcal{Q}_i, S \subset X_i}\mu_i(S)^2 \ll_\varepsilon p^{-(3-5\varepsilon)N_i}p^{N_i} = p^{(-2+5\varepsilon)N_i}
\end{align*}
for all large enough $i$. Let $C_\varepsilon$ by the implicit constant. Overall, we obtain
\begin{align*}
H_{\mu_i|_{X_i}}(\mathcal{Q}_i) \geq 2\mu_i(X_i)\log(\mu_i(X_i))
-\mu_i(X_i)\log(p)(-2+5\varepsilon)N_i - \mu_i(X_i)\log(C_\varepsilon).
\end{align*}
Observe that only the middle term is unbounded as $\mu_i(X_i)$ is bounded from below. Thus, for $i$ large enough
\begin{align*}
H_{\mu_i|_{X_i}}(\mathcal{Q}_i) \geq \mu_i(X_i)\log(p)(2-6\varepsilon)N_i
\geq (1-2\kappa^{-1}\mu_i(X_{\geq H}))\log(p)(2-6\varepsilon)N_i.
\end{align*}
which concludes the proof of the claim and thus also of the theorem.
\end{proof}

\subsection{Visiting the cusp}\label{section:CM-hecketree}

In this subsection, we provide a proof of Proposition \ref{prop:CM-visits to the cusp} using the geometric picture supplied by the Hecke tree (the Bruhat-Tits tree of $\PGL_2(\Qp)$) and adapting \cite{duke} correspondingly. Notice that under the projection
\begin{align*}
\pi: X \to Y:= \lquot{\PGL_2(\Z)}{\PGL_2(\R)}
\end{align*}
a $\PGL_2(\Qp)$-orbit gets mapped to a set isomorphic to $\rquot{\PGL_2(\Qp)}{\PGL_2(\Zp)}$. This follows from the fact that the action of $\PGL_2(\Qp)$ on $X$ has trivial stabilizers. The quotient $\rquot{\PGL_2(\Qp)}{\PGL_2(\Zp)}$ is equipped with the structure of a $(p+1)$-regular tree; we refer to Section 3.2 in Einsiedler and Ward \cite{aque} for the details. Given
\begin{align*}
\mathcal{N} := \set{\begin{pmatrix}
1 & 0 \\ 
0 & p
\end{pmatrix}, \begin{pmatrix}
p & 0 \\	 
0 & 1
\end{pmatrix},\begin{pmatrix}
p & 1 \\ 
0 & 1
\end{pmatrix},\ldots,\begin{pmatrix}
p & p-1 \\ 
0 & 1
\end{pmatrix} }
\end{align*}
we declare the neighbours of $g\PGL_2(\Zp)$ to be the points $\setc{gh\PGL_2(\Zp)}{h \in \mathcal{N}}$. As is verified in Proposition 3.15 in \cite{aque}, this really imposes the structure of a $(p+1)$-regular tree on $\rquot{\PGL_2(\Qp)}{\PGL_2(\Zp)}$. 
Given a point $y = \PGL_2(\Z)g \in Y$ the image of $\PGL_2(\Zinvp)(g,I)\PGL_2(\Q_p)$ under $\pi$ is called the embedded Hecke tree through $y$ and is equipped with a tree-structure by identification with the tree $\rquot{\PGL_2(\Qp)}{\PGL_2(\Zp)}$.

Given a point $x \in X$, the point $\pi(T(x))$ is always a neighbour of $\pi(x)$ in the Hecke tree by definition of $\mathcal{N}$. Taking the right quotient by $\PO_2(\R)$ of $Y$, the neighbours of $z = \pi(x) \PO_2(\R) \in \lquot{\PSL_2(\Z)}{\mathbb{H}}$ are exactly
\begin{align*}
\PSL_2(\Z)\set{pz, \frac{z}{p},\frac{z+1}{p},\,\ldots\, ,\frac{z+p-1}{p}}.
\end{align*}
If the height of $z$ is large enough (for instance $\height(z) \geq p$), then exactly one of its neighbours (that is, $pz$) is ``above'' $z$ (precisely $\height(pz) = p \height(z)$) and the other neighbours (that is, $\frac{z}{p},\frac{z+1}{p},\ldots,\frac{z+p-1}{p}$) are ``below'' $z$ (precisely of height $\height(z)/p$). Furthermore, the point $T^2(x)$ cannot be equal to $x$ due to the tree structure. We will use this observation as follows.

\begin{remark}\label{rem:CM-orbits that move down once}
Let $H>1$ be large enough, let $x\in X$ be a point with height $\geq H$ and suppose that the height of $T(x)$ is smaller than $x$. Then the point $T^2(x)$ cannot be ``above'' $T(x)$ as it would equal to $x$ in that case and is therefore ``below'' $T(x)$. The only condition we need to impose here, is that all points are above height $1$. In other words, the $T$-orbit of $x$ moves downwards for at least $\lgauss{\log_p(H)}$ time steps (as $\height(T^k(x)) = \height(T^{k-1}(x))/p$ for these $k$) until it ``crosses'' height one. The minimum time to reach height $H$ from height one is also at least $\lgauss{\log_p(H)}$.
\end{remark}

For the proof of Proposition \ref{prop:CM-visits to the cusp} we proceed exactly as in Section 5.1 of \cite{duke} and begin with the second assertion as the proof only depends on the remark above.

\begin{proof}[Proof of the second assertion in Proposition \ref{prop:CM-visits to the cusp}]
Consider the partition
\begin{align*}
\mathcal{P}_{H,N} = \bigvee_{n=-N}^N T^{-n}(\set{X_{<H},X_{\geq H}}).
\end{align*}
Every $V \subset [-N,N]$ with $Z(V) \neq \emptyset$ defines an atom of $\mathcal{P}_{H,N}$ and thus it suffices to prove that $\mathcal{P}_{H,N}$ contains $\ll_H e^{2 \frac{\log(\log(H))}{\log(H)}N} $ atoms. 
Consider first an atom of $\mathcal{P}_{H,\lgauss{\log_p(H)}}$ and a point $x$ in it. 
If for some $n \in \Z$  with $|n| \leq \lgauss{\log_p(H)}$ the point $T^n(x)$ is above height $H$ and $T^{n+1}(x)$ is below height $H$, then the orbit of $x$ stays below height $H$ for all times $>n$ in this interval by Remark \ref{rem:CM-orbits that move down once}. 
Thus, every point can leave $X_{< H}$ at most once. 
In particular, the time interval contains at most one stretch of times for which the orbit of the point can be above $H$. 
Therefore, the starting and the end point of that time interval uniquely determine an atom in $\mathcal{P}_{H,\lgauss{\log_p(H)}}$ and in particular there are $ \leq (2\lgauss{\log_p(H)}+1)^2 $ many atoms in $\mathcal{P}_{H,\lgauss{\log_p(H)}}$. 
The partition $T^{-N}(\mathcal{P}_{H,N})$ is coarser than a refinement over $\ugauss{\frac{2N+1}{2\lgauss{\log_p(H)}+1}}$ many partitions of the kind $T^{-j}(\mathcal{P}_{H,\lgauss{\log_p(H)}}) $. 
Hence $T^{-N}(\mathcal{P}_{H,N})$ (and thus also $\mathcal{P}_{H,N}$) contains at most
\begin{align*}
((2\lgauss{\log_p(H)}+1)^2)^{\ugauss{\frac{2N+1}{2\lgauss{\log_p(H)}+1}}} &\ll_H e^{4N\frac{ \log(2\lgauss{\log_p(H)}+1)}{2\lgauss{\log_p(H)}+1}}
\leq e^{4N \frac{ \log(2\log_p(H))}{2\log_p(H)}}\\
&\leq e^{2\log(p)N \frac{ \log(\log(H)) + \log(2)-\log(\log(p))}{\log(H)}}
\end{align*}
many atoms.
\end{proof}

The main geometric idea for the second assertion of Proposition \ref{prop:CM-visits to the cusp} is the following.

\begin{remark}[Moving up half of the time]\label{rem:CM-reaching maximal height}
Let $x \in X$ be a point for which $T^n(x)$ is below height $H>1$ at some times $n=N,N'$ for $N < N'$ and for which $T^n(x)$ is above height $H$ for all times $n$ with $N<n<N'$. Then the orbit of $x$ is ``moving upwards'' (the first) $50\%$ of the time (in $[N,N']$). This is a consequence of the fact that the ``speed of moving up or down'' is always $p$.

Writing $x = \PGL_2(\Zinvp) (g_\infty,g_p)$ for $(g_\infty,g_p) \in \PGL_2(\R \times \Zp)$ this means that
\begin{align*}
T^k(x) = \PGL_2(\Zinvp) (a^{-k}g_\infty,a^{-k}g_pa^{k})
\end{align*}
projects to the point $p^kg_\infty.\mathrm{i} = a^{-k}g_\infty.\mathrm{i}$ on the complex modular curve $Y_0(1)$ and therefore $a^{-k}g_pa^{k} \in \PGL_2(\R \times \Zp)$ for all $k$ in the first half of the interval $[N,N']$.
\end{remark}

\begin{proof}[Proof of the first assertion in Proposition \ref{prop:CM-visits to the cusp}]
For simplicity we denote the horospherical subgroups $G_a^+,G_a^-$ associated to $a$ by $U^+,U^- < \PGL_2(\Qp)$ respectively (see Example \ref{exp: horospherical subgrps for PGL_2}) and by $U^0:= \PGL_2(\R) \times G_a^0$.

It suffices to show that given a set of times $V \subset [0,N]$ and an open neighbourhood $\mathcal{O}$ of a point $x_0\in X$ of the form $x_0B_{\eta/2}^{U^+}B_{\eta/2}^{U^-U^0}$ the set
\begin{align*}
Z_\mathcal{O}^+(V) = \setc{x \in \mathcal{O}\intersect T^{-N}X_{<H}}{\forall n \in [0,N]: T^n(x) \in X_{\geq H}\iff n \in V }
\end{align*}
can be covered by $\ll p^{N-\frac{1}{2}|V|}$ forward Bowen balls. This follows by compactness of $X_{\leq H}$.
We partition the interval $[0,N]$ as follows. Decompose $V$ into maximal intervals containing consecutive times in $V$. By Remark \ref{rem:CM-orbits that move down once} two such intervals in $V$ have to be separated at least by $2\lgauss{\log_p(H)}$. Therefore we may thicken the above intervals in $V$ on both sides by $\lgauss{\log_p(H)}$ to, obtain disjoint intervals $I_1,\ldots,I_k$ covering $V$. Note that
\begin{align*}
|I_1|+\ldots+|I_k| = 2k\lgauss{\log_p(H)} + |V|.
\end{align*}
By slightly enlarging the interval $[0,N]$ if necessary we may assume that $I_1,\ldots,I_k$ are contained in $[0,N]$. This does not effect the estimate we aim for as the difference in $N$ depends on $H$ only and can thus by taken into the multiplicative constant. We fill the gaps between the intervals $I_i$ with maximal intervals $J_j$, $1\leq j\leq \ell$ so that $[0,N] = I_1 \union \ldots \union I_k \union J_1 \union \ldots\union J_\ell $ and prove the following claim by induction:

\begin{claim*}
For any $K \leq N$ with $[0,K] =  I_1 \union \ldots \union I_i \union J_1 \union \ldots\union J_j $ the set $Z_\mathcal{O}^+(V)$ can be covered by
\begin{align*}
\leq p^{|J_1|+\ldots+|J_j|+ i \lgauss{\log_p(H)} + \frac{1}{2}(|I_1|+\ldots+|I_i|)}
\end{align*}
preimages under $T^K$  of sets of the form 
\begin{align}\label{eq:CM-images of Bowen balls}
T^K(x_0)u^+B_{\eta/2}^{U^+}a^{-K}B_{\eta/2}^{U^-U^0}a^K
\end{align}
where $u^+ \in U^+$.
\end{claim*}

In particular, for $K= N$ this yields that $Z_\mathcal{O}^+(V)$ is covered by
\begin{align*}
\leq p^{|J_1|+\ldots+|J_\ell|+ k \lgauss{\log_p(H)} + \frac{1}{2}(|I_1|+\ldots+|I_k|)}
=  p^{N+ k \lgauss{\log_p(H)} - \frac{1}{2}(|I_1|+\ldots+|I_k|)} = p^{N-\frac{1}{2}|V|}
\end{align*}
sets of the form
\begin{align*}
T^{-N}(T^N(x_0)u^+B_{\eta/2}^{U^+}a^{-N}B_{\eta/2}^{U^-U^0}a^N)
&=  x_0 (a^Nu^+a^{-N})  a^N B_{\eta/2}^{U^+}a^{-N} B_{\eta/2}^{U^-U^0}
\end{align*}
contained in a Bowen $N$-ball. Thus the claim implies the proposition. 

Assume that the claim holds for $K \leq N$. We distinguish two cases:

\textbf{Case 1}: Suppose that $[0,K]$ is followed by $J_{j+1} = [K+1,K+S]$. Taking a set $T^K(x_0)u^+B_{\eta/2}^{U^+}a^{-K}B_{\eta/2}^{U^-U^0}a^K$ obtained in the previous step, its image under $T^S$ splits into $p^{S} = p^{|J_{j+1}|}$ sets of the form \eqref{eq:CM-images of Bowen balls} by properties of $U^+$, thus proving the claim.

\textbf{Case 2}: Suppose that $[0,K]$ is followed by $I_{i+1} = [K+1,K+S]$. Let 
\begin{align*}
R_K:= T^K(x_0)u^+B_{\eta/2}^{U^+}a^{-K}B_{\eta/2}^{U^-U^0}a^K
\end{align*}
be a set of the kind \eqref{eq:CM-images of Bowen balls} obtained in the previous step. As in the last case, we may split the image of $R_K$ into $\leq p^{S}$ sets of the kind \eqref{eq:CM-images of Bowen balls}. In this case, we claim that we may discard some of these sets as we are only interested in points in $y \in R_K$, which satisfy
\begin{align*}
T^n(y) \in X_{\geq H} \iff K + n \in V
\end{align*}
for $1 \leq n \leq S$. Let $y_1,y_2 \in R_K \intersect T^K(Z_\mathcal{O}^+(V))$. Then $y_2 \in y_1B_{\eta/2}^{U^+}a^{-K}B_{\eta}^{U^-U^0}a^K$, say $y_2 = y_1(g_\infty,h^+ h^b)$. We claim that
\begin{align}\label{eq:CM-visitstothecusp-claim2}
h^+ \in B_{p^{-S/2}}^{U^+}.
\end{align}

By definition of $I_{i+1}$, the points $y_1,y_2$ satisfy
\begin{align*}
\height(y_i),\height(y_ia),\, \ldots\, ,\height(y_ia^{\lgauss{\log_p(H)}}) < H \\
\height(y_ia^{\lgauss{\log_p(H)}+1}),\, \ldots\, ,\height(y_ia^{S-\lgauss{\log_p(H)}}) \geq  H \\
\height(y_ia^{S-\lgauss{\log_p(H)}+1}),\, \ldots\, ,\height(y_ia^{S}) < H
\end{align*}
In particular, $y_1$ and $y_2$ move upwards during the first $\frac{S}{2}$-time steps by Remark~\ref{rem:CM-reaching maximal height} and therefore  $a^{-j}h^+a^j \in \PGL_2(\Zp)$ for all $j \in [0,S/2]$ also by Remark \ref{rem:CM-reaching maximal height}. Conjugation by $a^{-1}$ stretches $h^+$ by a factor of $p$; if the size of $h^+$ is $\leq 1$ after $S/2$ conjugations, we must have had $h^+ \in B_{p^{-S/2}}^{U^+}$ initially. This concludes the claim made in \eqref{eq:CM-visitstothecusp-claim2}.

If $R_K \intersect T^K(Z_\mathcal{O}^+(V))$ is empty, there is nothing to do.
Otherwise, choose a point of reference $y \in R_K \intersect T^K(Z_\mathcal{O}^+(V))$. 
The claim above implies that $R_K \intersect T^K(Z_\mathcal{O}^+(V))$ is contained in $yB_{p^{-S/2}}^{U^+}a^{-K}B_{\eta/2}^{U^-U^0}a^K$. 
The image of this set under $T^S$ is covered by $\ll p^{-S/2}p^S$ sets of the form \eqref{eq:CM-images of Bowen balls} for $K+S$ as the ball $B_{p^{S/2}}^{U^+}$ is a disjoint union of $\ll p^{S/2}$ translates of the ball $B_{\eta/2}^{U^+}$.
\end{proof}
\section{Equidistribution of large subcollections}

In this short section we would like to explain how the proof of Theorem \ref{thm:main} generalizes to show equidistribution of very large subcollections with invariance.
We continue using the notation of Section \ref{section:formulation theorem} and begin with an important example.

\subsection{Equidistribution of squares}\label{section:squares}

In analogy to Section \ref{section:acting tori - general} we define for a pure vector $v \in \mathcal{O}$ the $\Q$-torus
\begin{align*}
\torus_\Fa^{(1)} = \setc{g \in \G^{(1)}}{g.v = v}.
\end{align*}
We then have the following variant of Theorem \ref{thm:main}.

\begin{theorem}\label{thm:simply connected version}
Let $p$ be an odd prime and 
let $(v_\ell)_\ell$ be a sequence of primitive vectors in $\mathcal{O}$, which is admissible at $p$. 
For any $\ell$ we choose $g_{\ell,\infty} \in \G^{(1)}(\R)$ such that $g_{\ell,\infty}K_\infty g_{\ell,\infty}^{-1} = \torus_{v_{\ell}}^{(1)}(\R)$ where $K_\infty$ is any choice of a proper maximal compact subgroup of $\G^{(1)}(\R)$.
Let $\mu_{v_\ell}$ be the normalized Haar measure on the packet $\G(\Q)\torus_{\Fa_\ell}^{(1)}(\adele)g_{\ell,\infty}$.
Then as $\ell\to\infty$ the measures $\mu_{v_\ell}$ converge to the Haar measure on $\lquot{\G^{(1)}(\Q)}{\G^{(1)}(\adele)}$
in the \wstar-topology.
\end{theorem}

The proof of Theorem \ref{thm:simply connected version} is along the lines of the proof of Theorem \ref{thm:main}.
For instance, the procedure of generating additional integer points (ideals) as in Section~\ref{section:generating intpts} (see also Lemma \ref{lemma:CM-generating ideals}) can be applied in the same fashion and the proof of Linnik's basic lemma is analogous.

The crucial point is to show that enough points are generated or equivalently that
\begin{align}\label{eq:volume for SL_2}
\vol(\G(\Q)\torus_{v_\ell}^{(1)}(\adele)) = \Nr(v_\ell)^{\frac{1}{2}+o(1)}
\end{align}
as from here Linnik's basic lemma (or the estimates for the mass in the cusp) follow analogously.
We prove \eqref{eq:volume for SL_2} in the next section.
The reader is advised to first read the proof in the case $\quat = \Mat_2$ (see also Proposition \ref{prop:CM-number of orbits}).

\subsubsection{Squaring in the Picard group}
Let $v \in \pure{\mathcal{O}}$ be a primitive vector of norm $D>0$.
In order to prove \eqref{eq:volume for SL_2} we need to show that the abelian group
\begin{align*}
\lrquot{\torus_v^{(1)}(\Q)}{\torus_v^{(1)}(\adele_f)}{\torus_v^{(1)}(\widehat{\Z})}
\end{align*}
has size $D^{\frac{1}{2}+o(1)}$.
For this estimate we will explicitly realize this group as a subset (and in fact subgroup) of the group
\begin{align*}
\lrquot{\torus_v(\Q)}{\torus_v(\adele_f)}{\torus_v(\widehat{\Z})}
\end{align*}
of which we already know the desired size estimate (cf.~Proposition \ref{prop:total volume}).
To do so, we will use the following general statement.

\begin{lemma}[Squaring]\label{lemma:SLvsPGL}
Let $\quat$ be a quaternion algebra over $\Q$ and define $\G,\G^{(1)}$ as in Section \ref{section:acting groups}.
Let $K \hookrightarrow \quat$ be an embedding of an imaginary quadratic field and let $\torus^{(1)} < \G^{(1)}$ and $\torus < \G$ be the respective centralizers.
Then the homomorphism $\torus^{(1)} \to \torus$ defined over $\Q$ has kernel $\set{\pm 1}$ and the image of $\torus^{(1)}(L)$ is exactly the set of squares in $\torus(L)$ for any field $L$ of characteristic zero.
\end{lemma}

\begin{proof}
If $t \in \torus^{(1)}$ has trivial image in $\torus$, then $t$ is in the center of $\quat$ i.e.~a scalar.
By the norm assumption on $t$ we have $t^2 = 1$ proving the claim about the kernel.

If $s^2 \in \torus(L)$ is a square, then $t = \frac{1}{\Nr(s)}s^2 \in \torus^{(1)}(L)$ maps to $s^2$. It remains to show that the image of any $t \in \torus^{(1)}(L)$ is a square, for which we distinguish two cases.

Suppose that $L' = K \otimes L$ is a field. In particular, $L'/L$ is quadratic extension.
Note that $\torus^{(1)}(L)$ is exactly the set of points in the image under $L \hookrightarrow \quat(L)$ which have norm one (and similarly for $\torus$).
If now $t \in \torus^{(1)}(L)$ then Hilbert's Theorem 90 applied to the norm one element in $L'$ corresponding to $t$ (for the relative norm of $L'/L$) yields that there exists $s \in L^\times$ for which
\begin{align*}
t = \frac{s}{\overline{s}} = s^2 \Nr(s)^{-2}.
\end{align*}
Viewing $s$ as an element of $\quat(L)$ the image of $t$ is thus equal to the image of $s^2$ proving the statement in this case.

Suppose now that $L' = K \otimes L$ is not a field i.e.~that $L' \cong L \oplus L$ as an $L$-algebra.
In particular, $\quat(L) $ is not a division algebra and hence we can identify $\quat(L) \cong \Mat_2(L)$.
Furthermore, $\torus^{(1)}$ is split over $L$ and we may replace $\torus^{(1)}(L)$ after conjugation with an element in $\quat^\times(L)$ by
\begin{align*}
\widetilde{\torus^{(1)}}(L) = \left\lbrace \operatorname{diag}(\lambda,\lambda^{-1}):\lambda \in L\setminus\set{0} \right\rbrace < \SL_2(L) = \G^{(1)}(L).
\end{align*}
The image of $\operatorname{diag}(\lambda,\lambda^{-1}) \in \tilde{\torus^{(1)}}(L)$ under the map $\SL_2(L) \to \PGL_2(L)$ is equal to $\operatorname{diag}(\lambda^2,1)$ and in particular a square.
\end{proof}

\begin{corollary}\label{cor:SL_2 gives squares}
Let $v \in \pure{\mathcal{O}}$ be a primitive vector.
Then the map
\begin{align*}
\lrquot{\torus_v^{(1)}(\Q)}{\torus_v^{(1)}(\adele_f)}{\torus_v^{(1)}(\widehat{\Z})} 
\to \lrquot{\torus_v(\Q)}{\torus_v(\adele_f)}{\torus_v(\widehat{\Z})}
\end{align*}
is injective and its image is the set of squares.
\end{corollary}

\begin{proof}
Notice first that the above map is indeed well-defined.
If $t \in \torus_v^{(1)}(\adele_f)$ has trivial image, we may write $ t = \lambda t_1t_2$ where $\lambda \in \adele_f^\times$ and where $t_1\in \quat^\times(\Q)$ and $t_2 \in \quat^\times(\widehat{\Z})$ stabilize $v$.
As $\G_m$ has class number one over $\Q$, we may write $\lambda = \lambda_1 \lambda_2$ where $\lambda_1 \in \Q^\times$ and $\lambda_2 \in \widehat{\Z}^\times$.
We may thus replace $t_1$ by $\lambda_1t_1$ and $t_2$ by $\lambda_2t_2$ and assume that $t = t_1t_2$.
As $\Nr(t) = 1$ we have $\Nr(t_1) = \Nr(t_2^{-1}) \in \Q^\times \cap \widehat{\Z}^\times = \Z^\times$.
Notice that $t_1$ is an element of the $\Q$-linear span of $1$ and $v$ and that the norm $\Nr$ restricted to this subspace is positive definite as $d>0$. Therefore, $\Nr(t_1)=1$ from which $\Nr(t_2) =1$ and the desired injectivity readily follow.

Lemma \ref{lemma:SLvsPGL} now shows that the image of the map is indeed the set of squares.
Here we implicitly used that for any square $t^2 \in \torus_v(\Zp)$ the preimage can be chosen in $\torus_v^{(1)}(\Zp)$ as is apparent from the proof of Lemma \ref{lemma:SLvsPGL}.
\end{proof}

\begin{proof}[Proof of \eqref{eq:volume for SL_2}]
Let $v \in \pure{\mathcal{O}}$ be a primitive vector of norm $D>0$. By Corollary~\ref{cor:SL_2 gives squares} the class number $|\lrquot{\torus_v^{(1)}(\Q)}{\torus_v^{(1)}(\adele_f)}{\torus_v^{(1)}(\widehat{\Z})}|$ is equal to the cardinality of the set of squares in the abelian group $H = \lrquot{\torus_v(\Q)}{\torus_v(\adele_f)}{\torus_v(\widehat{\Z})}$.
We therefore need to show that the $2$-torsion $H[2] = \setc{h\in H}{h^2 = 1}$ of $H$ satisfies $|H[2]| = D^{o(1)}$.

Define as in Section \ref{section:optimal embeddings} a discriminant $d<0$ via $d= -D$ if $D \equiv 3 \mod 4$ and $d = -4D$ if $D\equiv 0,1,2 \mod 4$.
It is shown as in \cite[Sec.~6.2]{Linnikthmexpander} that the optimal embedding
$\iota_v:K = \Q(\sqrt{d}) \to \quat(\Q)$ induced by $v$ yields a surjective map
$\Cl(R_d) \to H$ whose kernel depends only on the congruence properties of $d$ at the ramified primes of $\quat$.
Thus, $|H[2]|\ll |\Cl(R_d)[2]|\cdot|d|^{o(1)}$.

Recall (cf.~\cite[Sec.~14.4]{Cassels} or \cite[Prop.~3.11]{coxprimesoftheform}) that the $2$-torsion $\Cl(R_d)[2]$ of the Picard group $\Cl(R_d)$ has cardinality $\ll 2^{\omega}$ where $\omega$ is the number of distinct odd prime divisors of $d$.
Since $2^\omega$ is bounded by the value of the divisor function at $d$, we have $|\Cl(R_d)[2]| \ll_\varepsilon |d|^\varepsilon$ (cf.~Example \ref{exp: estimate divisor function}).
\end{proof}

\subsection{Large subcollections}

We now aim at formulating a theorem about subcollections of the packets as in Theorem \ref{thm:main}.

A \emph{subcollection} of a packet $\G(\Q)\torus_v(\adele)g_\infty$ for a pure vector $v \in \mathcal{O}$ and $g_\infty \in \G(\R)$ is a $g_\infty^{-1}\torus_v(\R \times \widehat{\Z})g_\infty$-invariant subset $\mathcal{S} \subset \G(\Q)\torus_v(\adele)g_\infty$.
As the $g_\infty^{-1}\torus_v(\R \times \widehat{\Z})g_\infty$-orbits in $\G(\Q)\torus_v(\adele)g_\infty$ (of which there are finitely many) correspond to points in the finite abelian group
\begin{align*}
\lrquot{\torus_v(\Q)}{\torus_v(\adele)}{\torus_v(\R\times\widehat{\Z})}
\simeq \lrquot{\torus_v(\Q)}{\torus_v(\adele_f)}{\torus_v(\widehat{\Z})},
\end{align*}
subcollections correspond to subsets of this group.
We shall call a subcollection \emph{large} if $\vol(\mathcal{S}) = \Nr(\tilde{v})^{\frac{1}{2}+o(1)}$ where $\tilde{v}$ is a primitive vector in $\Q v \cap \mathcal{O}$.

\begin{theorem}[On large subcollections]\label{thm:subcollections}
Let $p$ be an odd prime and let $(v_\ell)$ be an admissible sequence of primitive vectors in $\pure{\mathcal{O}}$.
For any $\ell$ choose $g_{\ell,\infty}\in \G(\R)$ with $\torus_{v_\ell}(\R) = g_{\ell,\infty} K_\infty g_{\ell,\infty}^{-1}$ where $K_\infty$ is any choice of a proper maximal compact subgroup of $\G(\R)$.

We choose for all $\ell$ a large subcollection $\mathcal{S}_\ell \subset \G(\Q)\torus_{v_\ell}(\adele)g_{\ell,\infty}$.
Assume that there exists $\lambda \in \Qp^\times$ with $|\lambda|_p \neq 1$  such that for any $\ell$ the torus $\torus_{v_\ell}(\Qp)$ contains an element $a_\ell$ with eigenvalues $\lambda,1,\lambda^{-1}$ (for the adjoint representation) under which $\mathcal{S}_\ell$ is invariant i.e. $a_\ell.\mathcal{S}_\ell \subset \mathcal{S}_\ell$.

Then any \wstar-limit of the measures $\mu_{v_\ell}|_{\mathcal{S}_\ell}$ is a probability measure and is invariant under $\sicover\G(\adele)$ where $\mu_{v_\ell}|_{\mathcal{S}_\ell}$ denotes the normalized restriction.
\end{theorem}

As in Section \ref{section:intpts on spheres} one can for instance apply this theorem to obtain equidistribution of certain subsets of integer points on spheres.
We note that stronger results than Theorem \ref{thm:subcollections} are known.
Harcos and Michel \cite[Thm.~6]{harcosmichelII} show equidistribution of CM points on the complex modular curve $Y_0(1)$ under a reduced exponent in the volume of the subcollections as long as the subcollections arise from subgroups of the Picard group attached to the packet.
A similar result for closed geodesics can be found in \cite[Thm.~6.5.1]{popasubcollections} and \cite{MEsubcollections} where in the latter this additional restriction on the subcollections is not required.

Theorem \ref{thm:subcollections} can be proven by the same method as Theorem \ref{thm:main}.

\begin{example}[$k$-th powers]
Let $\quat = \Mat_2$, let $\mathcal{O} = \Mat_2(\Z)$ and let $k\geq 2$ be an integer.
For any discriminant $d<0$ and any proper $R_d$-ideal $\Fa$ there is an isomorphism
\begin{align*}
\lrquot{\torus_\Fa(\Q)}{\torus_\Fa(\adele_f)}{\torus_\Fa(\widehat{\Z})}
\simeq \Cl(R_d)
\end{align*}
by Proposition \ref{prop:CM-number of orbits}.
We may thus let $\mathcal{S}_{\Fa,k}$ be the subcollection corresponding to the $k$-th powers in the Picard group $\Cl(R_d)$.
By construction, $\mathcal{S}_{\Fa,k}$ is invariant under the $k$-th power of any element in $\torus_\Fa(\Qp)$ and in particular under an element with eigenvalues $p^k,1,p^{-k}$.

The subcollection $\mathcal{S}_{\Fa,k}$ is large if and only if (cf.~the proof of \eqref{eq:volume for SL_2})
\begin{align*}
|\Cl(R_d)[k]| = d^{o(1)}.
\end{align*}
Such a bound however is only known for powers of two and conjectured otherwise (see for instance \cite{l-torsionEPW}).
\end{example}

\bibliographystyle{amsplain}
\providecommand{\bysame}{\leavevmode\hbox to3em{\hrulefill}\thinspace}
\providecommand{\MR}{\relax\ifhmode\unskip\space\fi MR }
\providecommand{\MRhref}[2]{%
  \href{http://www.ams.org/mathscinet-getitem?mr=#1}{#2}
}


\end{document}